\documentclass[11pt, a4paper,reqno]{amsart}
\pdfoutput=1
\usepackage{tikz,genyoungtabtikz,enumitem,rotating,latexsym,bm,stmaryrd}
\usetikzlibrary{positioning,intersections}
 \usepackage{amsmath,amsthm,amsfonts,amssymb,mathrsfs,pb-diagram}
\usepackage[
bookmarks=true,colorlinks=true,linktoc=page,citecolor=darkgreen,linkcolor=blue,urlcolor=mediumblue]{hyperref}
\usepackage{caption}
\usepackage{lipsum,wasysym}
\usepackage{mathtools}
\usepackage[a4paper,margin=0.87in]{geometry}
\usepackage{cleveref}
 \usepackage{amsmath}
\mathchardef\mhyphen="2D
\usepackage{color}
\definecolor{mediumblue}{rgb}{0.0, 0.0, 0.8}
 \colorlet{darkgreen}{green!50!black}
\newcommand\mptn[2]{\mathscr{P}^{#1}_{#2}}
\renewcommand{\geq}{\geqslant}
\renewcommand{\leq}{\leqslant}
\renewcommand{\trianglerighteq}{\trianglerighteqslant}
\renewcommand{\trianglelefteq}{\trianglelefteqslant}
\newcommand{\Q}{{\mathcal Q}}

\tikzset{wei/.style= 
{red,double=red,double
distance=0.5pt}}

\tikzset{wei2/.style={red,double=red,double
distance=0.5pt}}

\allowdisplaybreaks
\numberwithin{equation}{section}
\usepackage{scalefnt}

\newtheorem{thm}{Theorem}[section]
\newtheorem{cor}[thm]{Corollary}

\newtheorem{conj}[thm]{Conjecture}
\newtheorem{lem}[thm]{Lemma}
\newtheorem{prop}[thm]{Proposition}
\newtheorem*{prop*}{Proposition}
\newtheorem*{thmA*}{Theorem A}
\newtheorem*{thmB*}{Theorem B}
\newtheorem*{thmC*}{Theorem C}\newtheorem*{thm*}{Theorem D}
\newtheorem*{cor*}{Corollary}

\newtheorem*{conj*}{Conjecture A}
\newtheorem*{conj1*}{Conjecture B}
\newtheorem*{Acknowledgements*}{Acknowledgements}

\theoremstyle{remark}

\theoremstyle{definition}
\newtheorem{rmk}[thm]{Remark}
\newtheorem{defn}[thm]{Definition}
\newtheorem{eg}[thm]{Example}

\newcommand{\great}{>}
\newcommand{\less}{<}
\newcommand{\greatoreq}{\geq}
\newcommand{\lessoreq}{\leq}
\newcommand{\degr}{\mathrm{deg}}

\newcommand{\rad}{\mathrm{rad}}
\newcommand{\res}{\mathrm{res}}
\newcommand{\ik}{{k}}
\newcommand{\Std}{\operatorname{Std}}
\newcommand{\SStd}{{\rm SStd}}
\renewcommand{\det}{{\rm det}_h}

\newcommand{\TSStd}{\operatorname{\mathcal{T}}}
\newcommand{\Shape}{\operatorname{Shape}} 
\newcommand{\Path}{{\rm Path}}
\newcommand{\la}{\lambda}
\newcommand{\I}{i}
\newcommand{\J}{j}

\newcommand{\M}{m}

\newcommand{\SSTS}{\mathsf{S}}

\newcommand{\SSTT}{\mathsf{T}}  
\newcommand{\SSTU}{\mathsf{U}}  
\newcommand{\SSTV}{\mathsf{V}}  
\newcommand{\SSTQ}{\mathsf{Q}}  
\newcommand{\SSTR}{\mathsf{R}}  
\newcommand{\sts}{\mathsf{s}}  
\newcommand{\stt}{\mathsf{t}}  
\newcommand{\stu}{\mathsf{u}}  
\newcommand{\ZZ}{{\mathbb Z}}

\newcommand{\g}{\ell}
\newcommand{\Rem}{{\rm Rem}}

\newcommand{\RR}{{\mathbb R}}
\newcommand{\Hyp}{{\mathbb E}_{\alpha,me}}

\DeclareMathOperator{\Hom}{Hom}

\let\<=\langle
\let\>=\rangle

\newcommand\Dec[1][A]{\mathbf{D}_{#1}(t)}

\tikzset{
ultra thin/.style= {line width=0.05pt},
very thin/.style=  {line width=0.2pt},
thin/.style=       {line width=0.1pt},
semithick/.style=  {line width=0.6pt},
thick/.style=      {line width=0.8pt},
very thick/.style= {line width=1.2pt},
ultra thick/.style={line width=1.6pt}
}

\crefname{defn}{Definition}{Definitions}
\crefname{thm}{Theorem}{Theorems}
\crefname{prop}{Proposition}{Propositions}
\crefname{lem}{Lemma}{Lemmas}
\crefname{cor}{Corollary}{Corollaries}
\crefname{conj}{Conjecture}{Conjectures}
\crefname{section}{Section}{Sections}
\crefname{subsection}{Subsection}{Subsections}
\crefname{eg}{Example}{Examples}
\crefname{figure}{Figure}{Figures}
\crefname{rem}{Remark}{Remarks}
\crefname{rmk}{Remark}{Remarks}
\crefname{equation}{equation}{equation}

\Crefname{defn}{Definition}{Definitions}
\Crefname{thm}{Theorem}{Theorems}
\Crefname{prop}{Proposition}{Propositions}
\Crefname{lem}{Lemma}{Lemmas}
\Crefname{cor}{Corollary}{Corollaries}
\Crefname{conj}{Conjecture}{Conjectures}
\Crefname{section}{Section}{Sections}
\Crefname{subsection}{Subsection}{Subsections}
\Crefname{eg}{Example}{Examples}
\Crefname{figure}{Figure}{Figures}
\Crefname{rem}{Remark}{Remarks}
\Crefname{rmk}{Remark}{Remarks}

\def\Item{\item\abovedisplayskip=0pt\abovedisplayshortskip=5pt~\vspace*{-\baselineskip}} 

\begin{document}

 \title[Decomposition numbers of cyclotomic Hecke   and diagrammatic Cherednik algebras]{Modular decomposition numbers of cyclotomic Hecke \\ and diagrammatic Cherednik algebras: \\ A path theoretic approach	 	}
 
\author{C. Bowman} 
\address{School of Mathematics, Statistics and Actuarial Science University of Kent Canterbury
CT2 7NF, United Kingdom}
\email{C.D.Bowman@kent.ac.uk}

\author {A.~G.~Cox}
\address{Department of Mathematics, City, University of London,   London, United Kingdom}
\email{A.G.Cox@city.ac.uk}

 \maketitle

\!\!\!\!\!\!\!\!\!\!\! 
\begin{abstract}
 We introduce a path-theoretic framework for understanding the representation theory of (quantum) symmetric and general linear groups and their higher level generalisations over fields of arbitrary characteristic.  
Our first main result is a ``super-strong linkage principle'' which provides degree-wise upper bounds for graded decomposition numbers
    (this is new even in the  case of symmetric    groups).  
Next,    we  generalise the notion of   homomorphisms between   Weyl/Specht modules which are ``generically'' placed (within the associated alcove geometries) to cyclotomic Hecke and diagrammatic Cherednik algebras. 
    Finally, we  provide  
   evidence for a higher-level analogue of the classical Lusztig conjecture  over fields of sufficiently large characteristic.  
 \end{abstract}

 \section*{Introduction}

Cyclotomic  quiver  Hecke algebras (and their quasi-herediary covers, the diagrammatic Cherednik algebras) are of central interest in Khovanov homology, knot theory,  group theory, 
 and higher representation theory. 
Rouquier's  conjecture \cite{MR2422270} (recently solved in a flurry of publications
\cite{RSVV,losev,MR3732238})
 allows us to understand  the complex representation theory of 
these algebras in terms of Kazhdan--Lusztig theory.  
This paper  seeks to generalise  their work
 to the  modular representation theory of these algebras, where 
 almost nothing is known or even conjectured.

Our 
  approach  provides new insight  even in the classical case; 
 in particular it allows us to provide strong  new 
 degree-wise upper bounds for the graded decomposition numbers
of   symmetric groups.    This combinatorial bound is given in terms of folding-up paths in Euclidean space under the action of an affine Weyl group. 
This seems to be the first result of its kind in the literature,
  and so we state it now in this simplified form (for the full statement in higher levels, see \cref{strongerman}).   

\begin{thmA*}[The Super-Strong Linkage Principle for symmetric groups]
Let $\la,\mu $  be  partitions with at most $h$ columns and $\Bbbk$ be a field of characteristic $p>h$.  
The graded decomposition numbers of the symmetric group are bounded as follows,
\begin{equation}\label{introductoy}\tag{1}
[S(\lambda):D(\mu)\langle k \rangle]
\leq 
|\{\sts \mid \sts \in \Path^+(\lambda,\stt^\mu) ,  \deg(\sts) = k\}|
\end{equation}
 for $k \in \mathbb{Z}$. In particular if $[S(\lambda):D(\mu)]\neq 0$, then 
 $\lambda$ and  $ \mu$ are strongly linked with $\lambda \uparrow \mu$.  
\end{thmA*}

  Theorem A provides a two-fold strengthening of the famous strong linkage principle for  symmetric (and general linear) groups \cite{MR564523}.
Firstly, if $\Path^+(\lambda,\stt^\mu)\neq \emptyset$ this implies that $\lambda \uparrow \mu$ and so we obtain infinitely many new zeroes of the decomposition matrix not covered by \cite[Theorem 1]{MR564523}.   
 Secondly, \cref{introductoy} clearly 
provides a   wealth of new and  more complicated bounds on    
these multiplicities -- in addition  it  incorporates the grading into the picture for the first time.   We expect  this result to be of independent interest and so we have included  illustrative   examples   in \cref{sec:6}.

  In the case of symmetric groups,  
 it is common practice to   restrict  ones attention to the representations with at most $h$  columns.  In so doing, 
we obtain  a category of 
    representations 
which remains   
rich in structure but, thanks to revolutionary work of Riche--Williamson  \cite{w16}, is now known to     stabilise and become 
understandable over fields of characteristic $p\gg h$.  
The principal aim of this paper is to identify  a higher level analogue of this   category with a  similarly  rich  structure and to generalise 
the vast array of powerful ideas and results  developed by   Andersen, Carter, Jantzen, Kleshchev, Koppinen,  Lusztig and others
 over the past forty  years (in particular \cite{MR564523,MR1670762,CP,MR860709,klesh1box,MR591724,MR0486093}) 
 and hence  cast questions concerning the representation theory of these higher level  algebras   
in terms of their associated alcove geometries.

The quotient algebra  of the cyclotomic quiver Hecke algebra $H_n(\kappa)$ (and hence subcategory of 
$H_n(\kappa)\mhyphen{\rm mod}$) of interest to us is  
$$
\Q _{\ell,h, n}(\kappa) = H_n(\kappa) / 
\langle e(\underline{i}) 
\mid 
\underline{i}\in  I^{\ell} \text{ and }
\underline{i}_{k+1} = \underline{i}_{k}+1 \text{ for }1\leq k \leq h \rangle 
$$
for $e>h\ell$ and an $h$-admissible $\kappa\in   I^\ell$ as in \cref{admissible}.  
We shall see that this algebra is Morita equivalent to   a certain quotient, $A_h(n,\theta,\kappa)$, of the diagrammatic Cherednik algebra   
associated to the weighting $\theta=(1,2,\dots ,\ell) \in \ZZ^\ell$ (over an arbitrary field $\Bbbk$).  
In particular, the   simple representations of both of these algebras are
indexed by  the set of  multipartitions whose components each have at most $h$ columns, denoted $\mptn \ell n(h)$, and the graded decomposition matrices, $\mathbf{D}=(d_{\lambda\mu}(t))_{\lambda,\mu \in \mptn \ell n}$, coincide. 
   We  cast   representation theoretic questions concerning  these algebras in the setting of an alcove geometry  of type 
\begin{equation*} 
{A}_{ h-1} \times {A}_{ h-1} \times \dots \times {A}_{ h-1} \subseteq 
\widehat{A}_{\ell h-1}. 
 \end{equation*}
We first show that 
  the    algebra    $A_h(n,\theta,\kappa)$ has a graded cellular   basis indexed by 
  orbits of  paths in this geometry.  
 For each  $ \lambda \in \mptn \ell n$,  we hence obtain a basis 
 of the Weyl module, $ \Delta(\lambda) $, 
  which encodes a great deal  of representation-theoretic information.  
 This allows us to provide incredibly simple proofs of  a number of new structural results over arbitrary fields.  
 The first of which is our higher level analogue of the super-strong linkage principle.

 We then  consider  the idea of {\em generic behaviour}.  This generalises the idea (originally due to Jantzen and later Lusztig \cite{MR591724,MR0486093}) that
when we are ``sufficiently far away from the walls of the dominant region''   representation theoretic questions
 simplify greatly.    
 We encounter   higher level analogues of 
the familiar generic sets of points which are ``close together'' in the geometry  (for example, points ``around a Steinberg vertex").  
 In higher levels, there is  
also  a    striking
  new 
  kind of generic behaviour  involving points ``as far apart as possible'' in the geometry. 
   For such generic sets, one of our main results is the following (see \cref{howtobuildshit3} for the full statement).   

\begin{thmB*}[Generic Homomorphisms] 
  For  $(\la,\mu)$   a generic pair, 
 we have that
\begin{align*}
\dim_{t}(\Hom_{A_h(n,\theta,\kappa)}(  \Delta(\mu) ,\Delta(\lambda))) 
= t ^{ \ell(\mu)-\ell(\lambda)} + \cdots 	
\end{align*}
where the other terms are of strictly smaller degree.  
 We  provide an explicit construction of these homomorphisms and 
 prove results concerning their composition.  

 \end{thmB*}
 
 Theorem B 
 generalises results due to Carter--Lusztig \cite[Section4]{MR0354887}  and Koppinen \cite[Theorem 6.1]{MR860709} for $\ell=1$ 
 and  is utilised in   \cite{norton} in order to construct  the first 
 family of BGG resolutions of simple modules of Hecke and Cherednik algebras (indeed the first examples of such resolutions   anywhere in
  modular representation theory) and  to generalise and lift all  the results  of  \cite{MR2266877,MR1383482} to a structural level.   
    For higher levels, we find that there are arbitrarily large generic sets (as $n\to \infty$) 
  and we hence obtain arbitrarily  long chains of homomorphisms whose composition is non-zero (\cref{nonparabolic}).  
Finally, in \cref{dettensor} we  obtain  a higher-level analogue of the   stability for representations of general linear groups obtained by tensoring with the determinant, as follows.  

\begin{thmC*}
Given a partition $\lambda=(\lambda_1 ,\lambda _2,\dots)$ and $h\in \mathbb{N}$, we set 
  ${\rm det}_h(\lambda) =  	(h, \lambda_1 ,\lambda _2,\dots) .$  We have 
 an injective map  
${\rm det}_h:\mptn \ell n (h) \hookrightarrow \mptn \ell {n+h\ell} (h)  $ given by 
 $${\rm det}_h(\lambda^{(1)},\lambda^{(2)},\dots ,\lambda^{(\ell)}) =  
 	({\rm det}_h(\lambda^{(1)}),{\rm det}_h(\lambda^{(2)}),\dots,{\rm det}_h(\lambda^{(\ell)})) .$$
and   a corresponding injective homomorphism   of graded $R$-algebras
 $
 A_{h}(n,\theta,\kappa) \hookrightarrow
 A_{h}(n+h\ell,\theta,\kappa).	 $ 
 In particular, 
 $d_{\lambda,\mu}(t)=d_{{\rm det}_h(\lambda),{\rm det}_h(\mu)}(t)$ 
for all $\lambda,\mu\in \mptn \ell {n } (h)$.  
\end{thmC*}

Inspired by  these results,    we provide the first conjectural framework   
for calculating the (graded) decomposition numbers of cyclotomic Hecke algebras over fields of sufficiently large characteristic in \cref{onject}.  
 We verify these conjectures in the cases  of
  $(i)$    maximal finite parabolic   orbits  with $\Bbbk$  arbitrary -- generalising all the results of \cite{klesh1box,cmt08,tanteo13} 
  $(ii)$    $\ell=2$ or $3$ with  $e>n$ and $\Bbbk$  arbitrary  -- generalising   \cite[Theorems B3 \& B5]{MR3356809}  and  \cite[Section 9]{MR2781018}
  $(iii)$ $\ell=2$ and $h=1$ and $\Bbbk$  arbitrary \cite[Section 8]{MR1995129} 
  $(iv)$    $\Bbbk=\mathbb{C}$.  
  For $\ell=1$   the conjecture was  
 formulated by Andersen \cite[Section 5]{MR1670762} and  proved by Riche and Williamson  in \cite[Theorem 1.9]{w16}.

\begin{conj*}
Let   
$e>h\ell$,   $\kappa\in I^\ell$ be an $h$-admissible  multicharge, and   $\Bbbk$ be  field of  characteristic $p \gg  h\ell$.  
 The decomposition numbers  of   $A_h(n,\theta,\kappa)$    are    given by 
$$d_{\lambda\mu}(t) = n_{\lambda\mu} (t)  $$ 
for   $\lambda,\mu\in \mptn \ell n(h)$  in 
 the first $ep$-alcove   and $n_{\lambda\mu} (t)$ 
  the associated  affine  parabolic Kazhdan--Lusztig polynomial    
of type ${A}_{ h-1} \times {A}_{ h-1} \times \dots \times {A}_{ h-1} \subseteq 
\widehat{A}_{\ell h-1}.$ 
 
\end{conj*}

Our treatment covers the quotient algebras $\Q_{\ell,h,n}(\kappa)$ for $e\in \mathbb{Z}_{>0}$ uniformly alongside the algebras 
 $H_n(\kappa)$ for $e=\infty$ (of more generally $e>n$).  
The alcove geometries controlling the latter family of algebras are seen as 
simple sub-cases of those controlling the former family of algebras.  
 The following conjecture (a refinement of \cite[Conjecture 7.3]{MR2777040} which was debunked in \cite[Section 4.2]{MR3163410}) is the corresponding simplification of Conjecture A.

\begin{conj1*}
Let   $e=\infty$ or more generally $e>n$ and suppose  $\kappa \in I^\ell$ has no repeated entries.  
Let     $\Bbbk$ be  field of  characteristic $p \gg   \ell$.  
The decomposition numbers  of    $A(n,\theta,\kappa)$   are    given by 
$$d_{\lambda\mu}(t) = n_{\lambda\mu} (t)  $$ 
for  $\lambda,\mu\in \mptn \ell n $ and   $ n_{\lambda\mu} (t)$ the associated  parabolic Kazhdan--Lusztig polynomials   
of type  
   ${A}_{ n-1} \times {A}_{ n-1} \times \dots \times {A}_{ n-1} \subseteq 
{A}_{\ell n-1}. $ 
\end{conj1*}

    Our approach provides  the first 
        general framework for studying the modular representation theory of 
these algebras (for  admissible $\kappa \in I^\ell$).  
Indeed, while cyclotomic Hecke  algebras have been extensively studied over the past twenty years, 
surprisingly little is known about their  representation theory over fields of positive characteristic.  
  The blocks of these algebras were determined a decade ago by
   Lyle and Mathas \cite{MR2351381}.   
More recently, homomorphisms between certain pairs of 
 Specht  modules  were  constructed     in
     \cite{MR3213527} and  
  reduction theorems   between certain pairs of 
  Specht modules were   given in \cite{MR3533560,bs16}.  
  Apart from blocks of small weight \cite{MR2253658,MR3448789}, nothing else is known or even conjectured concerning the decomposition numbers and homomorphism spaces of these algebras.  
 Other powerful results concerning their representation theory do exist, including  
  explicit branching rules \cite{MR1319521,MR1643413,MR2271584} and a generalised Jantzen sum formula \cite{MR1665333} but they provide little information about general decomposition numbers.

The paper is structured as follows.  The first three sections introduce the main protagonists of this paper.   In \cref{sec:2} we 
 explicitly review    the construction of 
$\Q_{\ell,h,n}(\kappa)$ in terms of the classical KLR generators and
its coset-like cellular basis from \cite[Section 8]{manycell}.  
In \cref{sec:3}, we recall Webster's definition of the diagrammatic Cherednik algebras.    
In \cref{sec:4} we prove  that   $A_h(n,\theta,\kappa)$ and  $\Q_{\ell,h,n}(\kappa)$ are (graded) Morita equivalent.  We also discuss in detail why our choice of weighting is optimal for the purposes of understanding as much of the modular representation theory of $H_{n}(\kappa)$ as possible.  
  \cref{sec:2,sec:4} have been written so as to make the paper intelligible 
 to those studying the classical representation theory of cyclotomic Hecke algebras
  (without any prior knowledge of how diagrammatic Cherednik algebras fit into the picture).   \cref{sec:5}  provides the crux of this paper: we prove that the algebra $A_h(n,\theta,\kappa)$ possesses a cellular basis indexed by pairs of paths in the alcove geometry of type
  ${A}_{ h-1} \times {A}_{ h-1} \times \dots \times {A}_{ h-1} \subseteq 
\widehat{A}_{\ell h-1}.$ 
  In \cref{determinant} we consider the higher-level analogue of tensoring with the determinant.
    In \cref{sec:6}, we state and prove   the super-strong linkage principle for our algebras and illustrate how it significantly improves on the classical strong linkage principle.
    In \cref{sec:7}, we construct homomorphisms between the Weyl modules for $A_h(n,\theta,\kappa)$ and consider when  the composition  of these homomorphisms is non-zero.  
  In \cref{sec:8}   we formulate and provide evidence for our conjectures for calculating the decomposition numbers of  $A_h(n,\theta,\kappa)$ over fields of sufficiently large characteristic.  
  We remark that over $\mathbb{C}$, the approach presented here is in the same spirit of our earlier work \cite{bcs15} (which covered the case $\Q_{1,\ell,n}(\kappa)$).  However our main focus in this paper is the modular case.  
  Finally, in \cref{sec:9} we  demonstrate the kind of geometries we encounter (in particular, 
  these are more exotic than those seen in classical Lie theory) and  illustrate    how 
  one can use the tools of this paper to understand 
   a detailed example  for the Hecke algebras of type $B$ over an arbitrary field.

\section{Graded  cellular algebras }\label{sec:1}
We now recall the definition and first properties of graded cellular algebras following \cite[Section 2]{hm10}.
  Let $R$ denote an arbitrary commutative  integral domain and $\Bbbk$ denote an arbitrary field.

\begin{defn}[{\cite[Definition 2.1]{hm10}}]\label{defn1}  
Suppose that $A$ is a $\ZZ$-graded $R$-algebra which is of finite rank over $R$.
We say that $A$ is a 
{\sf    graded cellular algebra} if the following conditions hold.
 The algebra is equipped with a {\sf    cell datum} $(\Lambda,\TSStd,C,\degr)$, where $(\Lambda,\trianglerighteq )$ is the {\sf    weight poset}.
For each $\lambda \in\Lambda$  we have a finite set, denoted $\TSStd(\lambda )$.  
There exist maps
\[
C:{\textstyle\coprod_{\lambda\in\Lambda}\TSStd(\lambda)\times \TSStd(\lambda)}\to A 
      \quad\text{and}\quad
     \degr: {\textstyle\coprod_{\lambda\in\Lambda}\TSStd(\lambda)} \to \ZZ
\]
such that $C$ is injective. We denote $C(\SSTS,\SSTT) = c^\lambda_{\SSTS\SSTT}$  for $\SSTS,\SSTT\in\TSStd(\lambda)$, and require
  \begin{enumerate}[leftmargin=*]
    \item Each   element $c^\lambda_{\SSTS\SSTT}$ is homogeneous
	of degree 
$
        (c^\lambda_{\SSTS\SSTT})=\degr(\SSTS)+\degr(\SSTT),$ for
        $\lambda\in\Lambda$, 
      $\SSTS,\SSTT\in \TSStd(\lambda)$.
    \item The set $\{c^\lambda_{\SSTS\SSTT}\mid\SSTS,\SSTT\in \TSStd(\lambda), \,
      \lambda\in\Lambda \}$ is a  
      $R$-basis of $A$.
    \item  If $\SSTS,\SSTT\in \TSStd(\lambda)$, for some
      $\lambda\in\Lambda$, and $a\in A$ then 
    there exist scalars $r_{\SSTS\SSTU}(a)$, which do not depend on
    $\SSTT$, such that 
      \[ac^\lambda_{\SSTS\SSTT}  =\sum_{\SSTU\in
      \TSStd(\lambda)}r_{\SSTS\SSTU}(a)c^\lambda_{\SSTU\SSTT}\pmod 
      {A^{\vartriangleright  \lambda}},\]
      where $A^{\vartriangleright  \lambda}$ is the $R$-submodule of $A$ spanned by
      $\{c^\mu_{\SSTQ\SSTR}\mid\mu \vartriangleright  \lambda\text{ and }\SSTQ,\SSTR\in \TSStd(\mu )\}.$
    \item  The $R$-linear map $*:A\to A$ determined by
      $(c^\lambda_{\SSTS\SSTT})^*=c^\lambda_{\SSTT\SSTS}$, for all $\lambda\in\Lambda$ and
      all $\SSTS,\SSTT\in\TSStd(\lambda)$, is an anti-isomorphism of $A$.
   \end{enumerate}
\end{defn}

 This graded cellular structure allows us to immediately define a natural family of so-called   {\sf    graded cell modules} as follows.  Given any $\lambda\in\Lambda$, the   {\sf    graded cell module} $\Delta^A(\lambda)$ is the graded left $A$-module
  with basis
    $\{c^\lambda_{\SSTS } \mid \SSTS\in \TSStd(\lambda )\}$.
    The action of $A$ on $\Delta^A(\lambda)$ is given by
   \begin{equation}
   \label{cellmodule}
   a c^\lambda_{ \SSTS  }  =\sum_{ \SSTU \in \TSStd(\lambda)}r_{\SSTS\SSTU}(a) c^\lambda_{\SSTU},
   \end{equation}
    where the scalars $r_{\SSTS\SSTU}(a)$ are the scalars appearing in condition (3) of Definition \ref{defn1}.
  Let $t$ be an indeterminate over $\ZZ_{\geq 0}$. If $M=\oplus_{z\in\ZZ}M_z$ is
a free graded $R$-module,  
 then its {\sf    graded dimension} is the Laurent  polynomial
\[\dim_t{(M)}=\sum_{k\in\ZZ}(\dim_{\Bbbk}M_k)t^k.\]

If $M$ is a graded $A$-module and $k\in\ZZ$, define $M\langle k \rangle$ to be the same module with $(M\langle k \rangle)_i = M_{i-k}$ for all $i\in\ZZ$. We call this a {\sf    degree shift} by $k$.
If $M$ is a graded
$A$-module and $L^A$ is a graded simple module let $[M:L^A\langle k\rangle]$ be the
multiplicity of 
$L^A\langle k\rangle$ as a graded composition factor
of $M$, for $k\in\ZZ$.

We now recall the method by which one can, at least in principle,  construct all simple  modules of a 
graded cellular algebra.  This construction uses only basic linear algebra. 
 Suppose that $\lambda \in \Lambda$. There is a bilinear form $\langle\ ,\ \rangle_\lambda$ on $\Delta^A(\lambda) $ which
is determined by
\[c^\lambda_{\SSTU\SSTS}c^\lambda_{\SSTT\SSTV}\equiv
  \langle c^\lambda_\SSTS,c^\lambda_\SSTT\rangle_\lambda c^\lambda_{\SSTU\SSTV}\pmod{A^{\vartriangleright \lambda}},\]
for any $\SSTS,\SSTT, \SSTU,\SSTV\in \TSStd(\lambda  )$.  
For every $\lambda \in \Lambda$,   we let   $\langle\ ,\ \rangle_\lambda$ denote the bilinear form on $\Delta^A(\la)$ and $\rad \langle\ ,\ \rangle_\lambda$ denote the  radical of this bilinear form.
Given any $\la \in \Lambda$ such that $\rad \langle\ ,\ \rangle_\lambda\neq \Delta^A(\la)$, we  
set   $L^A(\lambda)=\Delta^A (\lambda) / \rad \langle\ ,\ \rangle_\lambda$.  
 This module is graded (by \cite[Lemma 2.7]{hm10}) and simple, and in fact every  simple  module is of this form, up to grading shift.  

\begin{prop}[\cite{hm10}]\label{humathasprop}
If $\mu\in\Lambda$ then $\dim_t{(L^A(\mu))} \in \ZZ_{\geq 0}[t+t^{-1}]$.   \end{prop}

    The passage between the (graded) cell and  simple  modules is recorded in the 
  {\sf    graded decomposition matrix}, $ \Dec=(d^A_{\lambda\mu}(t))$, of $A$ where    \[d_{\lambda\mu}^A(t)=\sum_{k\in\ZZ} [\Delta^A(\lambda):L^A(\mu)\langle k\rangle]\,t^k,\]
  for $\lambda,\mu\in\Lambda$.  This matrix is uni-triangular with respect to the partial ordering $\trianglerighteq$ on $\Lambda$.

\section{Cyclotomic  quiver  Hecke algebras}\label{sec:2}
We let $R$ denote an arbitrary commutative integral domain.  We let $\mathfrak{S}_n$ denote the symmetric group on $n$ letters, with   presentation
$$\mathfrak{S}_n=\langle s_1,\dots ,s_{n-1}\mid s_i s_{i+1}s_i- s_{i+1} s_{i}s_{i+1}, s_is_j - s_j s_i \text{ for }|i-j|>1\rangle.$$
We shall be interested in the representation theory (over $R$) of the 
reflection groups ${(\ZZ/\ell \ZZ)\wr \mathfrak{S}_n}$ and their 
deformations.  
The background material from this section is lifted from \cite[Section 2]{MR3732238}
 and \cite[Section 1]{manycell}.

\subsection{The quiver Hecke algebra}
 Throughout this paper $e$ is a fixed element of the set $\{3,4,5,\dots\}$ $\cup\{\infty\}$. 
This excludes the case $e=2$ as we are only interested in ``large" characteristics (where geometries have non-empty alcoves, see \cref{sec:5}).  
 If $e=\infty$ then we set $I=\ZZ$, while if $e<\infty$ then we set $I=\ZZ/e\ZZ$. 
 We  let $\Gamma_e$ be the
  quiver with vertex set $I $ and edges $i\longrightarrow i+1$, for
  $i\in I$.   Hence, we are considering either the linear quiver~$\ZZ$ ($e=\infty$)
  or a cyclic quiver ($e<\infty$):
   To the quiver $\Gamma_e$ we attach the symmetric Cartan matrix with  entries 
    $(a_{i,j})_{i,j\in I}$   defined by $a_{ij}=2\delta_{ij}-\delta_{i(j+1)}-\delta_{i(j-1)}$.
  
  Following \cite[Chapter~1]{Kac}, let $\widehat{\mathfrak{sl}}_e$ be the
  Kac-Moody algebra of~$\Gamma_e$  with simple roots
  $\{\alpha_i \mid i\in I\}$, fundamental weights $\{\Lambda_i\mid i\in I\}$,
  positive weight lattice $P^+=\bigoplus_{i\in I}\ZZ_{\geq 0}\Lambda_i$ and
  positive root lattice $Q^+=\bigoplus_{i\in I}\ZZ_{\geq 0}\alpha_i$. Let
  $(\cdot,\cdot)$ be the usual invariant form associated with this data,
  normalised so that $(\alpha_i,\alpha_j)=a_{ij}$ and
  $(\Lambda_i,\alpha_j)=\delta_{ij}$, for $i,j\in I$.   
  Fix a sequence $\kappa=(\kappa_1,\dots,\kappa_\ell)\in I^\ell$, the
  $e$-{\sf multicharge}, and define
  $\Lambda=\Lambda(\kappa) =
      \Lambda_{{\kappa}_1}+\dots+\Lambda_{{\kappa}_\ell}      $. Then $\Lambda\in P^+$ is
  dominant weight of {\sf level}~$\ell$.

\begin{defn}[\cite{MR2551762,MR2525917,ROUQ}]\label{defintino1}
Suppose $\alpha$ is a positive root of height $n$, and set
$I^\alpha=\{i \in  I^n  \mid  \alpha_{i_1}+\dots+\alpha_{i_n} =\alpha\}$.  Define   $H^\alpha(\kappa)$  to be the unital, associative $R$-algebra with generators
\begin{align}\label{gnrs}\tag{$\dagger$}
\{e(\underline{i}) \ | \ \underline{i}=(i_1,\dots,i_n)\in   I^\alpha\}\cup\{y_1,\dots,y_n\}\cup\{\psi_1,\dots,\psi_{n-1}\},
\end{align}
subject to the    relations 
\begin{align}
\label{rel1.1} e(\underline{i})e(\underline{j})&=\delta_{\underline{i},\underline{j}} e(\underline{i});
\\
\label{rel1.2}
\textstyle\sum_{\underline{i} \in   I^\alpha } e(\underline{i})&=1;\\ 
\label{rel1.3}
y_r		e(\underline{i})&=e(\underline{i})y_r;
\\
\label{rel1.4}
\psi_r e(\underline{i}) &= e(s_r\underline{i}) \psi_r;
\\ 
\label{rel1.5}
y_ry_s&=y_sy_r;\\ 
\label{rel1.6}
\psi_ry	_s&=\mathrlap{y_s\psi_r}\hphantom{\smash{\begin{cases}(\psi_{r+1}\psi_r\psi_{r+1}+y_r-2y_{r+1}+y_{r+2})e(\underline{i})\\\\\\\end{cases}}}\kern-\nulldelimiterspace\text{if } s\neq r,r+1;\\
\label{rel1.7}
\psi_r\psi_s&=\mathrlap{\psi_s\psi_r}\hphantom{\smash{\begin{cases}(\psi_{r+1}\psi_r\psi_{r+1}+y_r-2y_{r+1}+y_{r+2})e(\underline{i})\\\\\\\end{cases}}}
\kern-\nulldelimiterspace\text{if } |r-s|>1;
\\
\label{rel1.8}
y_r \psi_r e(\underline{i}) &=(\psi_r y_{r+1} + 
\delta_{i_r,i_{r+1}})e(\underline{i});\\
\label{rel1.9}
y_{r+1} \psi_r e(\underline{i}) &=(\psi_r y_r -  
\delta_{i_r,i_{r+1}})e(\underline{i});\\
\label{rel1.10}
\psi_r^2 e(\underline{i})&=\begin{cases}
\mathrlap0\phantom{(\psi_{r+1}\psi_r\psi_{r+1}+y_r-2y_{r+1}+y_{r+2})e(\underline{i})}& \text{if }i_r=i_{r+1},\\
e(\underline{i}) & \text{if }i_{r+1}\neq i_r, i_r\pm1,\\
(y_{r+1} - y_r) e(\underline{i}) & \text{if }i_{r+1}=i_r-1 ,\\
(y_r - y_{r+1}) e(\underline{i}) & \text{if }i_{r+1}=i_r+1; 
\end{cases}\\
\label{rel1.11}
\psi_r \psi_{r+1} \psi_r &=\begin{cases}
(\psi_{r+1}\psi_r\psi_{r+1} + 1)e(\underline{i})& \text{if }i_r=i_{r+2}=i_{r+1}+1 ,\qquad \\
(\psi_{r+1}\psi_r\psi_{r+1} - 1)e(\underline{i})& \text{if }i_r=i_{r+2}=i_{r+1}-1 ,\\
 (\psi_{r+1}\psi_r\psi_{r+1})e(\underline{i})& \text{otherwise;}
\end{cases} 
\end{align}
{for all admissible $r,s,i,j$. Finally, we have the cyclotomic relation}
\begin{align}
\label{rel1.12} y_1^{\langle \Lambda (\kappa) \mid \alpha_{i_1} \rangle} e(\underline{i}) &=0 \quad \text{for $\underline{i}\in I^\alpha$.}
\end{align}
   The {\sf quiver Hecke algebra} is the sum $H (\kappa):=\bigoplus_{\alpha}  H^\alpha(\kappa)$      over all positive roots of height $n$.  
 \end{defn}

\begin{thm}[\cite{MR2551762,MR2525917,ROUQ}]
We have a grading on $H_n(\kappa)$  given by
$$
{\rm deg}(e(\underline{i}))=0 \quad 
{\rm deg}( y_r)=2\quad 
{\rm deg}(\psi_r e(\underline{i}))=
\begin{cases}
-2		&\text{if }i_r=i_{r+1} \\
1		&\text{if }i_r=i_{r+1}\pm 1 \\
0 &\text{otherwise} 
\end{cases} $$
\end{thm} 

  For
  each $i\in I$ define the {\sf $i$-string of length~$h$} to be
  $\alpha_{i,h}=\alpha_i+\alpha_{i+1}+\dots+\alpha_{i+h-1} \in Q^+$.

   \begin{defn}\label{admissible}  For $h\in \mathbb{Z}_{>0}$ and $\kappa\in I^\ell$, we say that $\kappa$ is  
 $h$-{\sf admissible}   if  
$(\Lambda,\alpha_{h,i})\le 1$  for all $ i\in I$.  
\end{defn}

\begin{defn}
Let $e>h\ell$ and $\kappa\in I^\ell$ be an $h$-admissible multicharge.  
 We set 
$$\Q _{\ell,h, n}(\kappa) = H_n(\kappa) /
\langle e(\underline{i}) 
\mid 
\underline{i}\in  I^n \text{ and }
\underline{i}_{k+1} = \underline{i}_{k}+1 \text{ for }1\leq k \leq h \rangle.
$$
\end{defn}

\begin{rmk}\label{ringel1}
For $\ell=1$ and $ p=e$ the algebra $\Q_{1,h,n}(\kappa)$ is isomorphic 
to the image of the symmetric group on $n$ letters   in ${\rm End}_{\Bbbk}((\Bbbk^h)^{\otimes n})$.  
Therefore $\Q _{1,h, n}(\kappa)$ is the generalised Temperley--Lieb algebra of \cite[Section 1]{MR1680384} 
and 
is the Ringel dual (see \cite[Section 4]{MR1611011}) of the classical Schur algebra. 
 \end{rmk}

\subsection{Weighted  standard tableaux} 
The background material from this section is lifted from \cite[Section 2]{MR3732238}
and \cite[Section 1]{manycell}.   
Fix integers $\ell ,n\in\ZZ_{\geq 0} $ 
and $e\in \mathbb{Z}_{>0}\cup\{\infty\}$.  
We define    
  a {\sf weighting} to be any $\theta  \in \ZZ^\ell$ such that $ \theta_i-\theta_{j} \not\in \ell\ZZ$ for $1\leq i < j \leq \ell$. 
Throughout this section, we assume that the weighting $\theta  \in \ZZ^\ell$ and $e$-multicharge $\kappa \in I^\ell$ are fixed.

We define a {\sf partition}, $\lambda$,  of $n$ to be a   finite weakly decreasing sequence  of non-negative integers $ (\lambda_1,\lambda_2, \ldots)$ whose sum, $|\lambda| = \lambda_1+\lambda_2 + \dots$, equals $n$.   
An    {\sf $\ell $-multipartition}  $\lambda=(\lambda^{(1)},\dots,\lambda^{(\ell)})$ of $n$ is an $\ell $-tuple of     partitions  such that $|\lambda^{(1)}|+\dots+ |\lambda^{(\ell)}|=n$. 
We will denote the set of $\ell $-multipartitions of $n$ by $\mptn {\ell}n$.
Given  $\lambda=(\lambda^{(1)},\lambda^{(2)},\ldots ,\lambda^{(\ell)}) \in \mptn {\ell}n$, the {\sf Young diagram} of $\lambda$    is defined to be the set of nodes, 
\[
\{(r,c,m) \mid  1\leq  c\leq \lambda^{(m)}_r\}.
\]
We do not distinguish between the multipartition and its Young diagram.  
  We refer to a node $(r,c,m)$ as being in the $r$th row and $c$th column of the $m$th component of $\lambda$.  
 Given a node, $(r,c,m)$,  
we define the {\sf residue} of this node  to be  ${\rm res}(r,c,m)  = \kappa_m+  c - r\pmod e$.  
We refer to a node of residue  $i\in I$ as an $i$-node.

Given $\lambda\in \mptn \ell n$,  the associated  {\sf  $\theta$-Russian array} 
is defined as follows.
For each $1\leq m\leq \ell$, we place a point on the real line at $\theta_m$ and consider the region bounded by  half-lines at angles $3\pi/4 $ and $\pi/4$.
		We tile the resulting quadrant with a lattice of squares, each with diagonal of length $2 \ell $.  
We place a box  $(1,1,m)\in \lambda$   at the point  $\theta_m$ on the real line, with rows going northwest from this node, and columns going northeast. 
We do not distinguish between $\la$ 
and its $\theta$-Russian array.

\begin{defn}
Let    $\theta\in \ZZ^\ell $ be a weighting and $\kappa \in I^\ell$.  
Given $\lambda\in \mptn \ell n$, we define a {\sf tableau} of shape $\lambda$ to be a filling of the boxes of 
the $\theta$-Russian array of  $\la $ with the numbers 
$\{1,\dots , n\}$.  We define a {\sf  standard tableau} to be a tableau  in which    the entries increase along the rows and columns of each component.  
We let $\Std (\lambda)$ denote the set of all standard tableaux of shape $\lambda\in\mptn\ell n$. 
Given 
$\stt\in \Std (\lambda)$, we set $\Shape(\stt)=\la$.  
Given $1\leq k \leq n$, we let $\stt{\downarrow}_{\{1,\dots ,k\}}$ be the subtableau of $\stt$ whose  entries belong to the set
$\{1,\dots,k\}$.  \end{defn}

\begin{defn} \label{domdef}
Let  $(r,c,m), (r',c',m')    $ be two $i$-boxes  and    $\theta\in \ZZ^\ell$ be our fixed weighting.  
We write $(r,c,m) \rhd_\theta  (r',c',m')$ if 
$\theta_m +\ell(r-c) <\theta_{m'} +\ell(r'-c')$
or $\theta_m +\ell(r-c) = \theta_{m'} +\ell(r'-c')$
and $r+c < r'+c'$.  
\end{defn}

\begin{defn}  
Given  $\la,\mu\in \mptn \ell n$, we say that  
$\la$ $\theta$-{\sf dominates} $\mu$ (and write $\mu \trianglelefteq_\theta \la$) 
if for every $i$-box  $(r,c,m) \in \mu$, there exist at least as many $i$-boxes $(r',c',m') \in \la$ which $\theta$-dominate $(r,c,m)$ than  there do 
$i$-boxes $(r'',c'',m'') \in \mu$ which $\theta$-dominate $(r,c,m)$.  
\end{defn}

Given $\la \in \mptn \ell n$, we let   ${\rm Rem} (\la)$ (respectively ${\rm Add}  (\la)$) 
denote the set of all    removable  respectively addable) boxes  of the Young diagram of $\la$ so that the resulting diagram is the Young diagram of a multipartition.  Given $i\in \ZZ/e\ZZ$, we let  ${\rm Rem}_i  (\la)\subseteq {\rm Rem}  (\la)$ (respectively  ${\rm Add}_i  (\la)\subseteq {\rm Add}  (\la)$) denote the subset of boxes of residue $i\in I$.

\begin{defn} 
Let $\la\in\mptn\ell n$ and $\stt \in \Std(\la)$. 
We let $\stt^{-1}(k)$ denote the box in $\stt$ containing the integer $1\leq k\leq n$.   
Given $1\leq k\leq n$, we let ${\mathcal A}_\stt(k)$, 
(respectively ${\mathcal R}_\stt(k)$)  denote the set of   all addable $\res (\stt^{-1}(k))$-boxes (respectively all
removable   $\res (\stt^{-1}(k))$-boxes)  of the multipartition $\Shape(\stt{\downarrow}_{\{1,\dots ,k\}})$ which
 are less than  $\stt^{-1}(k)$ in the $\theta$-dominance order  (i.e  those which appear to the {\em right} of $\stt^{-1}(k)$).    

\end{defn}

\begin{defn} Let $\la\in\mptn\ell n$ and $\stt \in \Std(\la)$.  We define the degree of $\stt$ as follows,
$$
\deg(\stt) = \sum_{k=1}^n \left(	|{\mathcal A}_\stt(k)|-|{\mathcal R}_\stt(k)|	\right).
$$
\end{defn}


\begin{eg}
Let $e=7$ and  $\kappa=(0,3)$ and $\theta=(0,1)$.  
The rightmost  tableau $  \sts  \in \Std((3,2,1),(3,2,2))$ depicted in \cref{cyclic tableau} is of degree   $1$.  
 The boxes of non-zero degree are   $\sts^{-1}(k)$ for 
$k=5$, $12$, and $13$.  
We have that $\deg(\stt^{-1}(5))=1$,
 $\deg(\sts^{-1}(12))=-1$, 
 and $\deg(\sts^{-1}(13))=1$.  
For example,   the box $ \sts^{-1}(10) $ appears to the right of 
$ \sts^{-1}(12) $ and both are of residue $5 \in \ZZ/7\ZZ$.  
\end{eg}

\subsection{The graded cellular basis }
In this section we recall the construction of the  graded cellular basis of the algebra  $\Q _{\ell,h, n}(\kappa)$ from \cite[Section 8]{manycell}. 
For this section, we  fix   $\theta=(1,2,\dots ,\ell) \in \ZZ^\ell$.   

\begin{defn}
Given $\la \in \mptn \ell n$ we let $\stt^\lambda \in \Std(\la)$ be the tableau   obtained by placing the entry $n$ in the least dominant removable box $(r,c,m) \in \la $ (in the $\theta$-dominance order) and then placing the entry $n-1$ in the least dominant removable box of $\la  \setminus \{(r,c,m) \}$ and continuing in this fashion.  
\end{defn}

\begin{eg}
We have that $\stt^\lambda\in \Std(\la)$ for $\la=((3,2,1),(3,2,2))$ is the leftmost tableau depicted in  
\cref{cyclic tableau}.  
\end{eg}

\!\!\!\!\!\!\!\!\!\!
\begin{figure}[ht]\captionsetup{width=0.9\textwidth}
\[\scalefont{0.9}
\begin{tikzpicture}[scale=0.9]
\begin{scope}
\draw[thick](2,-1.95)--(-1.5,-1.95); 
{     \path (0,0.8) coordinate (origin);     \draw[wei2] (origin)   circle (2pt);
 
\clip(-2.2,-1.95)--(-2.2,3)--(2,3)--(2,-1.95)--(-2.2,-1.95); 
\draw[wei2] (origin)   --(0,-2); 
\draw[thick] (origin) 
--++(130:3*0.7)
--++(40:1*0.7)
--++(-50:1*0.7)	
--++(40:1*0.7) --++(-50:1*0.7)
--++(40:1*0.7) --++(-50:1*0.7) 
--++(-50-90:3*0.7);
 \clip (origin)  
--++(130:3*0.7)
--++(40:1*0.7)
--++(-50:1*0.7)	
--++(40:1*0.7) --++(-50:1*0.7)
--++(40:1*0.7) --++(-50:1*0.7) 
--++(-50-90:3*0.7);
\path  (origin)--++(40:0.35)--++(130:0.35)  coordinate (111); 
   \node at  (111)  {  $1$};
\path  (origin)--++(40:0.35)--++(130:0.35+0.7)  coordinate (121); 
   \node at  (121)  {  $2$};
     \path  (origin)--++(40:0.35)--++(130:0.35+0.7*2)  coordinate (131); 
   \node at  (131)  {  $3$};
     \path  (origin)--++(40:0.35+0.7)--++(130:0.35 )  coordinate (211); 
   \node at  (211)  {  $7$};
            \path  (origin)--++(40:0.35+0.7)--++(130:0.35+0.7 )  coordinate (221); 
   \node at  (221)  {  $8$};
           \path  (origin)--++(40:0.35+0.7*2)--++(130:0.35 )  coordinate (211); 
   \node at  (211)  {  $13$};
\path (40:1cm) coordinate (A1);
\path (40:2cm) coordinate (A2);
\path (40:3cm) coordinate (A3);
\path (40:4cm) coordinate (A4);
\path (130:1cm) coordinate (B1);
\path (130:2cm) coordinate (B2);
\path (130:3cm) coordinate (B3);
\path (130:4cm) coordinate (B4);
\path (A1) ++(130:3cm) coordinate (C1);
\path (A2) ++(130:2cm) coordinate (C2);
\path (A3) ++(130:1cm) coordinate (C3);
\foreach \i in {1,...,19}
{
\path (origin)++(40:0.7*\i cm)  coordinate (a\i);
\path (origin)++(130:0.7*\i cm)  coordinate (b\i);
\path (a\i)++(130:4cm) coordinate (ca\i);
\path (b\i)++(40:4cm) coordinate (cb\i);
\draw[thin,gray] (a\i) -- (ca\i)  (b\i) -- (cb\i); } 
}
\end{scope}
\begin{scope}
{   
\path (0,-1.3)++(40:0.3*0.7)++(-50:0.3*0.7) coordinate (origin);  
\draw[wei2]  (origin)   circle (2pt);
\clip(-2.2,-1.95)--(-2.2,3)--(3,3)--(3,-1.95)--(-2.2,-1.95); 
\draw[wei2] (origin)   --++(-90:4cm); 
\draw[thick] (origin) 
--++(130:3*0.7)
--++(40:1*0.7)
--++(-50:1*0.7)	
--++(40:2*0.7)   --++(-50:2*0.7) 
--++(-50-90:3*0.7);
 \clip (origin) 
--++(130:3*0.7)
--++(40:1*0.7)
--++(-50:1*0.7)	
--++(40:2*0.7)   --++(-50:2*0.7) 
--++(-50-90:3*0.7);
  \path  (origin)--++(40:0.35)--++(130:0.35)  coordinate (111); 
   \node at  (111)  {  $4$};
\path  (origin)--++(40:0.35)--++(130:0.35+0.7)  coordinate (121); 
   \node at  (121)  {  $5$};
     \path  (origin)--++(40:0.35)--++(130:0.35+0.7*2)  coordinate (131); 
   \node at  (131)  {  $6$};
     \path  (origin)--++(40:0.35+0.7)--++(130:0.35 )  coordinate (211); 
   \node at  (211)  {  $9$};
            \path  (origin)--++(40:0.35+0.7)--++(130:0.35+0.7 )  coordinate (221); 
   \node at  (221)  {  $10$};
           \path  (origin)--++(40:0.35+2*0.7)--++(130:0.35  )  coordinate (224); 
   \node at  (224)  {  $11$};       
   \path  (origin)--++(40:0.35+2*0.7)--++(130:0.35+0.7 )  coordinate (223); 
   \node at  (223)  {  $12$};
\path (40:1cm) coordinate (A1);
\path (40:2cm) coordinate (A2);
\path (40:3cm) coordinate (A3);
\path (40:4cm) coordinate (A4);
\path (130:1cm) coordinate (B1);
\path (130:2cm) coordinate (B2);
\path (130:3cm) coordinate (B3);
\path (130:4cm) coordinate (B4);
\path (A1) ++(130:3cm) coordinate (C1);
\path (A2) ++(130:2cm) coordinate (C2);
\path (A3) ++(130:1cm) coordinate (C3);
\foreach \i in {1,...,19}
{
\path (origin)++(40:0.7*\i cm)  coordinate (a\i);
\path (origin)++(130:0.7*\i cm)  coordinate (b\i);
\path (a\i)++(130:4cm) coordinate (ca\i);
\path (b\i)++(40:4cm) coordinate (cb\i);
\draw[thin,gray] (a\i) -- (ca\i)  (b\i) -- (cb\i); } 
}  \end{scope}
\end{tikzpicture}
\qquad
\begin{tikzpicture}[scale=0.9]
\begin{scope}
\draw[thick](2,-1.95)--(-1.5,-1.95); 
{     \path (0,0.8) coordinate (origin);     \draw[wei2] (origin)   circle (2pt);
 
\clip(-2.2,-1.95)--(-2.2,3)--(2,3)--(2,-1.95)--(-2.2,-1.95); 
\draw[wei2] (origin)   --(0,-2); 
\draw[thick] (origin) 
--++(130:3*0.7)
--++(40:1*0.7)
--++(-50:1*0.7)	
--++(40:1*0.7) --++(-50:1*0.7)
--++(40:1*0.7) --++(-50:1*0.7) 
--++(-50-90:3*0.7);
 \clip (origin)  
--++(130:3*0.7)
--++(40:1*0.7)
--++(-50:1*0.7)	
--++(40:1*0.7) --++(-50:1*0.7)
--++(40:1*0.7) --++(-50:1*0.7) 
--++(-50-90:3*0.7);
\path  (origin)--++(40:0.35)--++(130:0.35)  coordinate (111); 
   \node at  (111)  {  $3$};
\path  (origin)--++(40:0.35)--++(130:0.35+0.7)  coordinate (121); 
   \node at  (121)  {  $5$};
     \path  (origin)--++(40:0.35)--++(130:0.35+0.7*2)  coordinate (131); 
   \node at  (131)  {  $8 
    $};
     \path  (origin)--++(40:0.35+0.7)--++(130:0.35 )  coordinate (211); 
   \node at  (211)  {  $7$};
            \path  (origin)--++(40:0.35+0.7)--++(130:0.35+0.7 )  coordinate (221); 
   \node at  (221)  {  $13$};
           \path  (origin)--++(40:0.35+0.7*2)--++(130:0.35 )  coordinate (211); 
   \node at  (211)  {  $10$};
\path (40:1cm) coordinate (A1);
\path (40:2cm) coordinate (A2);
\path (40:3cm) coordinate (A3);
\path (40:4cm) coordinate (A4);
\path (130:1cm) coordinate (B1);
\path (130:2cm) coordinate (B2);
\path (130:3cm) coordinate (B3);
\path (130:4cm) coordinate (B4);
\path (A1) ++(130:3cm) coordinate (C1);
\path (A2) ++(130:2cm) coordinate (C2);
\path (A3) ++(130:1cm) coordinate (C3);
\foreach \i in {1,...,19}
{
\path (origin)++(40:0.7*\i cm)  coordinate (a\i);
\path (origin)++(130:0.7*\i cm)  coordinate (b\i);
\path (a\i)++(130:4cm) coordinate (ca\i);
\path (b\i)++(40:4cm) coordinate (cb\i);
\draw[thin,gray] (a\i) -- (ca\i)  (b\i) -- (cb\i); } 
}
\end{scope}
\begin{scope}
{   
\path (0,-1.3)++(40:0.3*0.7)++(-50:0.3*0.7) coordinate (origin);  
\draw[wei2]  (origin)   circle (2pt);
\clip(-2.2,-1.95)--(-2.2,3)--(3,3)--(3,-1.95)--(-2.2,-1.95); 
\draw[wei2] (origin)   --++(-90:4cm); 
\draw[thick] (origin) 
--++(130:3*0.7)
--++(40:1*0.7)
--++(-50:1*0.7)	
--++(40:2*0.7)   --++(-50:2*0.7) 
--++(-50-90:3*0.7);
 \clip (origin) 
--++(130:3*0.7)
--++(40:1*0.7)
--++(-50:1*0.7)	
--++(40:2*0.7)   --++(-50:2*0.7) 
--++(-50-90:3*0.7);
  \path  (origin)--++(40:0.35)--++(130:0.35)  coordinate (111); 
   \node at  (111)  {  $1$};
\path  (origin)--++(40:0.35)--++(130:0.35+0.7)  coordinate (121); 
   \node at  (121)  {  $4$};
     \path  (origin)--++(40:0.35)--++(130:0.35+0.7*2)  coordinate (131); 
   \node at  (131)  {  $12$};
     \path  (origin)--++(40:0.35+0.7)--++(130:0.35 )  coordinate (211); 
   \node at  (211)  {  $2$};
            \path  (origin)--++(40:0.35+0.7)--++(130:0.35+0.7 )  coordinate (221); 
   \node at  (221)  {  $6 
   $};
           \path  (origin)--++(40:0.35+2*0.7)--++(130:0.35  )  coordinate (224); 
   \node at  (224)  {  $9$};       
   \path  (origin)--++(40:0.35+2*0.7)--++(130:0.35+0.7 )  coordinate (223); 
   \node at  (223)  {  $11$};
\path (40:1cm) coordinate (A1);
\path (40:2cm) coordinate (A2);
\path (40:3cm) coordinate (A3);
\path (40:4cm) coordinate (A4);
\path (130:1cm) coordinate (B1);
\path (130:2cm) coordinate (B2);
\path (130:3cm) coordinate (B3);
\path (130:4cm) coordinate (B4);
\path (A1) ++(130:3cm) coordinate (C1);
\path (A2) ++(130:2cm) coordinate (C2);
\path (A3) ++(130:1cm) coordinate (C3);
\foreach \i in {1,...,19}
{
\path (origin)++(40:0.7*\i cm)  coordinate (a\i);
\path (origin)++(130:0.7*\i cm)  coordinate (b\i);
\path (a\i)++(130:4cm) coordinate (ca\i);
\path (b\i)++(40:4cm) coordinate (cb\i);
\draw[thin,gray] (a\i) -- (ca\i)  (b\i) -- (cb\i); } 
}  \end{scope}\end{tikzpicture}
\]
\caption{The tableau $  \sts,\stt^\lambda \in \Std((3,2,1),(3,2,2))$}
\label{cyclic tableau}
\end{figure}

\begin{defn}Given $\stt \in \Std(\lambda)$, we define the associated {\sf residue sequence} as follows,  
$$ 
{\underline{i}}_\stt = ({\rm res} (\stt^{-1}(1)), {\rm res} (\stt^{-1}(2)),  \dots, {\rm res}(\stt^{-1}(n)) )    \in I^n.  
$$
\end{defn}

\begin{eg}    Given $  \sts,\stt^\lambda \in \Std(\lambda)$ for $\lambda=(3,2,1),(3,2,2)$, we  have that 
$$ 
{\underline{i}}_{\stt^\la} =
(0,1,2,3,4,5,6,0,2,3,1,2,5)
\qquad 
{\underline{i}}_{\sts} =
 (3,2,0,4,1,3,6,2,1,5,2,5,0)
  $$
\end{eg}

  \begin{defn}
Given $\la \in \mptn \ell n$ and  $\sts \in \Std (\la)$, we let $d _\sts \in \mathfrak{S}_n$ denote any 
element such that $d  _\sts  (\stt^\lambda) =\sts$ under the place permutation action of the symmetric group on standard tableaux.
\end{defn}

\begin{defn}\label{psielements}
Given $\la \in \mptn \ell n $
and $\sts,\stt \in \Std(\la)$ we
fix    
reduced expressions 
$d_\sts= s_{i_1}s_{i_2}\dots s_{i_k}$ and 
$d_\stt= s_{j_1}s_{j_2}\dots s_{j_m}$.  
We 
set 
$$ \psi _{\sts   \stt}^{\theta,\kappa} =
 \psi _{\sts}  e(\underline{i}_{\stt^\lambda} ) (\psi _{  \stt}) ^\ast $$
where
$\psi _\sts =\psi_{i_1}\psi_{{i_2}}\dots \psi_{i_k}$, 
$\psi _\stt =\psi_{j_1}\psi_{{j_2}}\dots \psi_{j_m}$.       \end{defn}

\begin{defn}\label{hrestrictedpartitions}
 Let $\mptn \ell n (h)\subseteq \mptn \ell n$ denote the   subset consisting of all multipartitions with at most $h$ columns in any given component, that is
 $$\mptn \ell n (h)= \{\la=(\la^{(1)}, \dots , \la^{(\ell)}) \mid \lambda^{(m)}_1\leq h \text{ for all }1\leq m \leq \ell\}.$$
  \end{defn}


\begin{thm}[{\cite[Theorem 8.2]{manycell}}]\label{cor11} 
The algebra  $\Q _{h,\ell,n}(\kappa)$ admits a  graded cellular     basis
\begin{equation*}    \{ \psi^{\theta,\kappa}_{\sts   \stt}  \mid \lambda \in \mptn \ell n(h), \sts,\stt \in \Std(\lambda)\}
\end{equation*}
with respect to the $\theta$-dominance order on $\mptn \ell n(h)$ and the involution $\ast$.   We refer to the resulting 
cell-modules (as in \cref{cellmodule}) as the  {\sf Specht modules} of $\Q_{h,\ell,n}(\kappa)$  and denote them by 
$$
S (\lambda) = \{ \psi_\sts^{\theta,\kappa}     \mid \sts \in \Std(\lambda)\} 
$$
for $\lambda\in \mptn \ell n (h)$. The modules $S (\lambda) $ lift to modules of $H_n(\kappa)$ and the 
  decomposition matrix ${\bf D}^{\Q_{h,\ell,n}(\kappa)}$ appears as a (square) submatrix of ${\bf D}^{H_n(\kappa)}$.  
  \end{thm}
\begin{rmk}
As in \cite[Definition 5.1]{hm10} and \cite[Definition 7.11]{manycell}, 
   the elements $\psi^{\theta,\kappa}_\sts $ for $\sts \in \Std(\lambda)$ are defined up to a choice of reduced expression for $d_\sts$, however   any such reduced expression can be chosen arbitrarily.  
\end{rmk}
\begin{rmk}
 It is shown in \cite[Corollary 7.3]{manycell} that 
for $h=1$    the algebra $ \Q_{1,\ell,n}(\kappa)$ is isomorphic to the generalised blob algebra of \cite[Section 1.2]{MW03}
(in particular $ \Q_{1,2,n}(\kappa)$ is isomorphic to the blob algebra of \cite[Section 2]{MW00}).  
\end{rmk}

\section{Diagrammatic Cherednik algebras}\label{sec:3}

We  recall the definitions and important properties of  diagrammatic Cherednik algebras from \cite{MR3732238}.  
  We first tilt  the $\theta$-Russian array of $\lambda\in \mptn\ell n$   ever-so-slightly in the clockwise direction so that the top vertex of the box $(r,c,m)\in \lambda$  
   has $x$-coordinate $\mathbf{I}_{(r,c,m)}=\theta_m + \ell(r-c) + (r+c)\varepsilon$ for 
$\varepsilon\ll \tfrac{1}{n}$ (using standard small-angle identities to approximate the coordinate to order  $\varepsilon^2$).   
Given $\la\in\mptn \ell n$, we let ${\bf I}_\la$ denote   
 the disjoint union over the $\mathbf{I}_{(r,c,m)}$ for $(r,c,m)\in \la$.  
    
    \begin{defn} \label{semistandard:defn} Let $\lambda,\mu \in \mptn {\ell}n$.  A $\lambda$-tableau of weight $\mu$ is a
  bijective map $\SSTT :  \lambda  \to {\mathbf I}_\mu$. 
%
%
   We say that a tableau $\SSTT$ is {\sf semistandard} if it  satisfies the following additional properties   \begin{itemize}
\item[$(i)$]     $\SSTT(1,1,m)>\theta_m$,
\item[$(ii)$]    $\SSTT(r,c,m)> \SSTT(r-1,c,m)  +\g$,
\item[$(iii)$] 
$\SSTT(r,c,m)> \SSTT(r,c-1,m) -\g$.   
\end{itemize}
We  denote the set of all  semistandard tableaux of shape $\lambda$
and weight $\mu$ by $\SStd_{\theta,\kappa}(\lambda,\mu)$.   Given $\SSTT \in 
\SStd_{\theta,\kappa}(\lambda,\mu)$, we   write $\Shape(\SSTT)=\lambda$.  
When the context is clear we write $\SStd (\lambda,\mu)$ or $\SStd_n (\lambda,\mu)$ for $ \SStd_{\theta,\kappa}(\lambda,\mu)$.   \end{defn}

  \begin{defn}
 We let $\SStd^+_n(\lambda,\mu) \subseteq \SStd_n(\lambda,\mu) $ denote the subset of 
 tableaux which respect residues.  
  In other words,  if $
  \SSTT(r,c,m)= \mathbf{I}_{(r',c',m')} \in {\mathbf I}_\mu$ for   $(r,c,m) \in  \lambda  $ 
  and $(r',c',m')\in \mu$, then  $\kappa_m+c-r = \kappa_{m'}+c'-r' \pmod e$.    
 \end{defn}

\!\!\!\!\!\!\!  
\begin{figure}[ht!]
 \[
\begin{tikzpicture}[baseline, thick,yscale=1,xscale=4]\scalefont{0.8}

  \draw[wei]  (-1.1, -2)  to[out=90,in=-90] (-1.1, 2) ;
  \draw[red] (-1.1, -2.05) [below] node{$\ \kappa_2$};

  \draw[wei]  (-1.2, -2) to[out=90,in=-90] (-1.2, 2);
   \draw[red] (-1.2, -2.05) [below] node{$\ \kappa_1$};
 
 \draw(-1.5,-2) rectangle (1.4,2);

  \draw[gray, densely dotted] (-0.4+-.5,-2) to[out=90,in=-90]   (-0.4+-1,-1)  to[out=90,in=-90](-0.4+1,.3) to[out=90,in=-90]  
    (-0.4+0,2);
  \draw[gray, densely dotted] (-0.4+.5,-2) to[out=90,in=-90]  (-0.4+.5,0) to[out=90,in=-90] (-0.4+1,2);
  \draw[gray, densely dotted]  (-0.4+1,-2) to[out=90,in=-90]  (-0.4+-1,1) to[out=90,in=-90] (-0.4+.5,2);
  \draw[gray, densely dotted] (-0.4+0,-2) to[out=90,in=-90]  
  (-0.4+0.1,0) to[out=90,in=-90]
  (-0.4+-.5,2);
  \draw[gray, densely dotted]  (-0.4+-1, -2) to[out=95,in=-90]  
  (-0.4+-.5,-1.1) to[out=90,in=-90] (-0.4+-1,0) to[out=90,in=-90] (-0.4+-.5,1)
  to[out=90,in=-90]   (-0.4+-1,2);

  \draw (-.5,-2) to[out=90,in=-90] node[below,at start]{$\ i_2$} (-1,-1)  to[out=90,in=-90](1,.2) to[out=90,in=-90]    (0,2);
  \draw (.5,-2) to[out=90,in=-90] node[below,at start]{$\ i_4$} (.5,0) to[out=90,in=-90] node[midway,circle,fill=black,inner sep=2pt]{} (1,2);
  \draw  (1,-2) to[out=90,in=-90]  node[below,at start]{$\ i_5$} (-1,1) to[out=90,in=-90]
   (.5,2);
  \draw (0,-2) to[out=90,in=-90] node[below,at start]{$\ i_3$}node[midway,below,circle,fill=black,inner sep=2pt]{}
  (0.05,0) to[out=90,in=-90]
  (-.5,2);
  \draw  (-1, -2) to[out=90,in=-90] node[below,at start]{$\ i_1$}
  (-.5,-1) to[out=90,in=-90] (-1,0) to[out=90,in=-90] (-.5,1)
  to[out=90,in=-90]   (-1,2);

\end{tikzpicture} 
\]
\caption{A $\theta$-diagram  for $\theta=(0,1)$  with northern and southern loading
given by ${\bf I}_\omega$ where $\omega= (\varnothing,(1^5))$. }
\label{adsfhlkjasdfhlkjsadfhkjlasdfhlkjadfshlkjdsafhljkadfshjlkasdf}
\end{figure}

\begin{defn}
We define a  $\theta$-{\sf diagram} {of type} $G(\ell,1,n)$ to  be a  {\sf frame}
$\mathbb{R}\times [0,1]$ with  distinguished  solid points on the northern and southern boundaries
given by   $\mathbf{I}_\mu$ and $\mathbf{I}_\lambda$ for some
$\lambda,\mu  \in \mptn \ell n$ and a 
collection of {\sf solid strands} each of which starts at a northern point, ${\bf I}_{(r,c,m)}$ for $(r,c,m)\in \mu$,  
 and ends at a southern
  point, ${\bf I}_{(r',c',m')}$ for $(r',c',m')\in \la$. 
  Each strand carries some residue, $i \in I$ say, and is referred to as a {\sf solid $i$-strand}.  
 We further require that each solid strand has a 
mapping diffeomorphically to $[0,1]$ via the projection to the $y$-axis.  
Each solid strand is allowed to carry any number of dots.  We draw
\begin{itemize}[leftmargin=*]
\item  a dashed line $\ell$ units to the left of each solid $i$-strand, which we call a {\sf ghost $i$-strand} or $i$-{\sf ghost};
\item vertical red lines at $\theta_m   \in 
\ZZ$ each of which carries a 
residue $\kappa_m$ for $1\leq m\leq \ell$ which we call a {\sf red $\kappa_m$-strand}. 
\end{itemize}
We refer to a point at which two strands cross as a {\sf double point}.  
We require that there are no triple points (points at which three strands cross) or tangencies
(points at which  two curves intersect without crossing one another at that point) involving any combination of strands, ghosts or red lines and no dots lie on crossings. \end{defn}

An example of a $\theta$-{\sf diagram} is given in  \cref{adsfhlkjasdfhlkjsadfhkjlasdfhlkjadfshlkjdsafhljkadfshjlkasdf}.

\begin{defn}[{\cite[Definition 4.1]{MR3732238}}]\label{defintino2}
The {\sf diagrammatic Cherednik algebra}, ${\bf A}(n,\theta,\kappa)$,  is the  $R$-algebra spanned by   all $\theta$-diagrams modulo the following local relations (here a local relation means one that 
can be applied on a small region of the diagram).
\begin{enumerate}[label=(2.\arabic*)] 
\item\label{rel1}  Any diagram may be deformed isotopically; that is,
 by a continuous deformation
 of the diagram which 
 avoids  tangencies, double points and dots on crossings. 
\item\label{rel2} 
For $i\neq j$ we have that dots pass through crossings. 
\[   \scalefont{0.8}\begin{tikzpicture}[scale=.5,baseline]
  \draw[very thick](-4,0) +(-1,-1) -- +(1,1) node[below,at start]
  {$i$}; \draw[very thick](-4,0) +(1,-1) -- +(-1,1) node[below,at
  start] {$j$}; \fill (-4.5,.5) circle (5pt);
    \node at (-2,0){=}; \draw[very thick](0,0) +(-1,-1) -- +(1,1)
  node[below,at start] {$i$}; \draw[very thick](0,0) +(1,-1) --
  +(-1,1) node[below,at start] {$j$}; \fill (.5,-.5) circle (5pt);
  \node at (4,0){ };
\end{tikzpicture}\]
 \item\label{rel3}  For two like-labelled strands we get an error term.
\[
\scalefont{0.8}\begin{tikzpicture}[scale=.5,baseline]
  \draw[very thick](-4,0) +(-1,-1) -- +(1,1) node[below,at start]
  {$i$}; \draw[very thick](-4,0) +(1,-1) -- +(-1,1) node[below,at
  start] {$i$}; \fill (-4.5,.5) circle (5pt);
   \node at (-2,0){=}; \draw[very thick](0,0) +(-1,-1) -- +(1,1)
  node[below,at start] {$i$}; \draw[very thick](0,0) +(1,-1) --
  +(-1,1) node[below,at start] {$i$}; \fill (.5,-.5) circle (5pt);
  \node at (2,0){+}; \draw[very thick](4,0) +(-1,-1) -- +(-1,1)
  node[below,at start] {$i$}; \draw[very thick](4,0) +(0,-1) --
  +(0,1) node[below,at start] {$i$};
\end{tikzpicture}  \quad \quad \quad
\scalefont{0.8}\begin{tikzpicture}[scale=.5,baseline]
  \draw[very thick](-4,0) +(-1,-1) -- +(1,1) node[below,at start]
  {$i$}; \draw[very thick](-4,0) +(1,-1) -- +(-1,1) node[below,at
  start] {$i$}; \fill (-4.5,-.5) circle (5pt);
       \node at (-2,0){=}; \draw[very thick](0,0) +(-1,-1) -- +(1,1)
  node[below,at start] {$i$}; \draw[very thick](0,0) +(1,-1) --
  +(-1,1) node[below,at start] {$i$}; \fill (.5,.5) circle (5pt);
  \node at (2,0){+}; \draw[very thick](4,0) +(-1,-1) -- +(-1,1)
  node[below,at start] {$i$}; \draw[very thick](4,0) +(0,-1) --
  +(0,1) node[below,at start] {$i$};
\end{tikzpicture}\]
\item\label{rel4} For double-crossings of solid strands with $i\neq j$, we have the following.
\[
\scalefont{0.8}\begin{tikzpicture}[very thick,scale=.5,baseline]
\draw (-2.8,-1) .. controls (-1.2,0) ..  +(0,2)
node[below,at start]{$i$};
\draw (-1.2,-1) .. controls (-2.8,0) ..  +(0,2) node[below,at start]{$i$};
\node at (-.5,0) {=};
\node at (0.4,0) {$0$};
\end{tikzpicture}
\hspace{.7cm}
\scalefont{0.8}\begin{tikzpicture}[very thick,scale=.5,baseline]
\draw (-2.8,-1) .. controls (-1.2,0) ..  +(0,2)
node[below,at start]{$i$};
\draw (-1.2,-1) .. controls (-2.8,0) ..  +(0,2)
node[below,at start]{$j$};
\node at (-.5,0) {=}; 

\draw (1.8,-1) -- +(0,2) node[below,at start]{$j$};
\draw (1,-1) -- +(0,2) node[below,at start]{$i$}; 
\end{tikzpicture}
\]
\end{enumerate}
\begin{enumerate}[resume, label=(2.\arabic*)]  
\item\label{rel5} If $j\neq i-1$,  then we can pass ghosts through solid strands.
\[\begin{tikzpicture}[very thick,xscale=1,yscale=.5,baseline]
\draw (1,-1) to[in=-90,out=90]  node[below, at start]{$i$} (1.5,0) to[in=-90,out=90] (1,1)
;
\draw[densely dashed] (1.5,-1) to[in=-90,out=90] (1,0) to[in=-90,out=90] (1.5,1);
\draw (2.5,-1) to[in=-90,out=90]  node[below, at start]{$j$} (2,0) to[in=-90,out=90] (2.5,1);
\node at (3,0) {=};
\draw (3.7,-1) -- (3.7,1) node[below, at start]{$i$}
;
\draw[densely dashed] (4.2,-1) to (4.2,1);
\draw (5.2,-1) -- (5.2,1) node[below, at start]{$j$};
\end{tikzpicture} \quad\quad \quad \quad 
\scalefont{0.8}\begin{tikzpicture}[very thick,xscale=1,yscale=.5,baseline]
\draw[densely dashed] (1,-1) to[in=-90,out=90] (1.5,0) to[in=-90,out=90] (1,1)
;
\draw (1.5,-1) to[in=-90,out=90] node[below, at start]{$i$} (1,0) to[in=-90,out=90] (1.5,1);
\draw (2,-1) to[in=-90,out=90]  node[below, at start]{$\tiny j$} (2.5,0) to[in=-90,out=90] (2,1);
\node at (3,0) {=};
\draw (4.2,-1) -- (4.2,1) node[below, at start]{$i$}
;
\draw[densely dashed] (3.7,-1) to (3.7,1);
\draw (5.2,-1) -- (5.2,1) node[below, at start]{$j$};
\end{tikzpicture}
\]
\end{enumerate} 
\begin{enumerate}[resume, label=(2.\arabic*)]  
\item\label{rel6} On the other hand, in the case where $j= i-1$, we have the following.
\[\scalefont{0.8}\begin{tikzpicture}[very thick,xscale=1,yscale=.5,baseline]
\draw (1,-1) to[in=-90,out=90]  node[below, at start]{$i$} (1.5,0) to[in=-90,out=90] (1,1)
;
\draw[densely dashed] (1.5,-1) to[in=-90,out=90] (1,0) to[in=-90,out=90] (1.5,1);
\draw (2.5,-1) to[in=-90,out=90]  node[below, at start]{$\tiny i\!-\!1$} (2,0) to[in=-90,out=90] (2.5,1);
\node at (3,0) {=};
\draw (3.7,-1) -- (3.7,1) node[below, at start]{$i$}
;
\draw[densely dashed] (4.2,-1) to (4.2,1);
\draw (5.2,-1) -- (5.2,1) node[below, at start]{$\tiny i\!-\!1$} node[midway,fill,inner sep=2.5pt,circle]{};
\node at (5.75,0) {$-$};

\draw (6.2,-1) -- (6.2,1) node[below, at start]{$i$} node[midway,fill,inner sep=2.5pt,circle]{};
\draw[densely dashed] (6.7,-1)-- (6.7,1);
\draw (7.7,-1) -- (7.7,1) node[below, at start]{$\tiny i\!-\!1$};
\end{tikzpicture}\]

\item\label{rel7} We also have the relation below, obtained by symmetry.  \[
\scalefont{0.8}\begin{tikzpicture}[very thick,xscale=1,yscale=.5,baseline]
\draw[densely dashed] (1,-1) to[in=-90,out=90] (1.5,0) to[in=-90,out=90] (1,1)
;
\draw (1.5,-1) to[in=-90,out=90] node[below, at start]{$i$} (1,0) to[in=-90,out=90] (1.5,1);
\draw (2,-1) to[in=-90,out=90]  node[below, at start]{$\tiny i\!-\!1$} (2.5,0) to[in=-90,out=90] (2,1);
\node at (3,0) {=};
\draw[densely dashed] (3.7,-1) -- (3.7,1);
\draw (4.2,-1) -- (4.2,1) node[below, at start]{$i$};
\draw (5.2,-1) -- (5.2,1) node[below, at start]{$\tiny i\!-\!1$} node[midway,fill,inner sep=2.5pt,circle]{};
\node at (5.75,0) {$-$};

\draw[densely dashed] (6.2,-1) -- (6.2,1);
\draw (6.7,-1)-- (6.7,1) node[midway,fill,inner sep=2.5pt,circle]{} node[below, at start]{$i$};
\draw (7.7,-1) -- (7.7,1) node[below, at start]{$\tiny i\!-\!1$};
\end{tikzpicture}\]
\end{enumerate}
\begin{enumerate}[resume, label=(2.\arabic*)]
\item\label{rel8} Strands can move through crossings of solid strands freely.
\[
\scalefont{0.8}\begin{tikzpicture}[very thick,scale=.5 ,baseline]
\draw (-2,-1) -- +(-2,2) node[below,at start]{$k$};
\draw (-4,-1) -- +(2,2) node[below,at start]{$i$};
\draw (-3,-1) .. controls (-4,0) ..  +(0,2)
node[below,at start]{$j$};
\node at (-1,0) {=};
\draw (2,-1) -- +(-2,2)
node[below,at start]{$k$};
\draw (0,-1) -- +(2,2)
node[below,at start]{$i$};
\draw (1,-1) .. controls (2,0) ..  +(0,2)
node[below,at start]{$j$};
\end{tikzpicture}
\]
\end{enumerate}
 for any $i,j,k\in I$.    Similarly, this holds for triple points involving ghosts, except for the following relations when $j=i-1$.
\begin{enumerate}[resume, label=(2.\arabic*)]  \Item\label{rel9}
\[
\scalefont{0.8}\begin{tikzpicture}[very thick,xscale=1,yscale=.5,baseline]
\draw[densely dashed] (-2.6,-1) -- +(-.8,2);
\draw[densely dashed] (-3.4,-1) -- +(.8,2); 
\draw (-1.1,-1) -- +(-.8,2)
node[below,at start]{$\tiny j$};
\draw (-1.9,-1) -- +(.8,2)
node[below,at start]{$\tiny j$}; 
\draw (-3,-1) .. controls (-3.5,0) ..  +(0,2)
node[below,at start]{$i$};

\node at (-.75,0) {=};

\draw[densely dashed] (.4,-1) -- +(-.8,2);
\draw[densely dashed] (-.4,-1) -- +(.8,2);
\draw (1.9,-1) -- +(-.8,2)
node[below,at start]{$\tiny j$};
\draw (1.1,-1) -- +(.8,2)
node[below,at start]{$\tiny j$};
\draw (0,-1) .. controls (.5,0) ..  +(0,2)
node[below,at start]{$i$};

\node at (2.25,0) {$-$};

\draw (4.9,-1) -- +(0,2)
node[below,at start]{$\tiny j$};
\draw (4.1,-1) -- +(0,2)
node[below,at start]{$\tiny j$};
\draw[densely dashed] (3.4,-1) -- +(0,2);
\draw[densely dashed] (2.6,-1) -- +(0,2);
\draw (3,-1) -- +(0,2) node[below,at start]{$i$};
\end{tikzpicture}
\]
\Item\label{rel10}
\[
\scalefont{0.8}\begin{tikzpicture}[very thick,xscale=1,yscale=.5,baseline]
\draw[densely dashed] (-3,-1) .. controls (-3.5,0) ..  +(0,2);  
\draw (-2.6,-1) -- +(-.8,2)
node[below,at start]{$i$};
\draw (-3.4,-1) -- +(.8,2)
node[below,at start]{$i$};
\draw (-1.5,-1) .. controls (-2,0) ..  +(0,2)
node[below,at start]{$\tiny j$};

\node at (-.75,0) {=};

\draw (0.4,-1) -- +(-.8,2)
node[below,at start]{$i$};
\draw (-.4,-1) -- +(.8,2)
node[below,at start]{$i$};
\draw[densely dashed] (0,-1) .. controls (.5,0) ..  +(0,2);
\draw (1.5,-1) .. controls (2,0) ..  +(0,2)
node[below,at start]{$\tiny j$};

\node at (2.25,0) {$+$};

\draw (3.4,-1) -- +(0,2)
node[below,at start]{$i$};
\draw (2.6,-1) -- +(0,2)
node[below,at start]{$i$}; 
\draw[densely dashed] (3,-1) -- +(0,2);
\draw (4.5,-1) -- +(0,2)
node[below,at start]{$\tiny j$};
\end{tikzpicture}
\]
\end{enumerate}
In the diagrams with crossings in \ref{rel9} and \ref{rel10}, we say that the solid (respectively ghost) strand bypasses the crossing of ghost strands (respectively solid strands).
The ghost strands may pass through red strands freely.
For $i\neq j$, the solid $i$-strands may pass through red $j$-strands freely. If the red and solid strands have the same label, a dot is added to the solid strand when straightening.  Diagrammatically, these relations are given by the following diagrams and their mirror images
\begin{enumerate}[resume, label=(2.\arabic*)] \Item\label{rel11}
\[
\scalefont{0.8}\begin{tikzpicture}[very thick,baseline,scale=.5]
\draw (-2.8,-1)  .. controls (-1.2,0) ..  +(0,2)
node[below,at start]{$i$};
\draw[wei] (-1.2,-1)  .. controls (-2.8,0) ..  +(0,2)
node[below,at start]{$i$};

\node at (-.3,0) {=};

\draw[wei] (2.8,-1) -- +(0,2)
node[below,at start]{$i$};
\draw (1.2,-1) -- +(0,2)
node[below,at start]{$i$};
\fill (1.2,0) circle (3pt);

\draw[wei] (6.8,-1)  .. controls (5.2,0) ..  +(0,2)
node[below,at start]{$j$};
\draw (5.2,-1)  .. controls (6.8,0) ..  +(0,2)
node[below,at start]{$i$};

\node at (7.7,0) {=};

\draw (9.2,-1) -- +(0,2)
node[below,at start]{$i$};
\draw[wei] (10.8,-1) -- +(0,2)
node[below,at start]{$j$};
\end{tikzpicture}
\]
\end{enumerate}\newpage
\noindent for $i\neq j$. All solid crossings and dots can pass through  red strands, with a
correction term:
\begin{enumerate}[resume, label=(2.\arabic*)]
\Item\label{rel12}
\[
\scalefont{0.8}\begin{tikzpicture}[very thick,baseline,scale=.5]
\draw (-2,-1)  -- +(-2,2)
node[at start,below]{$i$};
\draw (-4,-1) -- +(2,2)
node [at start,below]{$j$};
\draw[wei] (-3,-1) .. controls (-4,0) ..  +(0,2)
node[at start,below]{$k$};

\node at (-1,0) {=};

\draw (2,-1) -- +(-2,2)
node[at start,below]{$i$};
\draw (0,-1) -- +(2,2)
node [at start,below]{$j$};
\draw[wei] (1,-1) .. controls (2,0) .. +(0,2)
node[at start,below]{$k$};

\node at (2.8,0) {$+ $};

\draw (7.5,-1) -- +(0,2)
node[at start,below]{$i$};
\draw (5.5,-1) -- +(0,2)
node [at start,below]{$j$};
\draw[wei] (6.5,-1) -- +(0,2)
node[at start,below]{$k$};
\node at (3.8,-.2){$\delta_{i,j,k}$};
\end{tikzpicture}
\]
\Item\label{rel13}
\[
\scalefont{0.8}\begin{tikzpicture}[scale=.5,very thick,baseline=2cm]
\draw[wei] (-2,2) -- +(-2,2);
\draw (-3,2) .. controls (-4,3) ..  +(0,2);
\draw (-4,2) -- +(2,2);

\node at (-1,3) {=};

\draw[wei] (2,2) -- +(-2,2);
\draw (1,2) .. controls (2,3) ..  +(0,2);
\draw (0,2) -- +(2,2);

\draw (6,2) -- +(-2,2);
\draw (5,2) .. controls (6,3) ..  +(0,2);
\draw[wei] (4,2) -- +(2,2);

\node at (7,3) {=};

\draw (10,2) -- +(-2,2);
\draw (9,2) .. controls (10,3) ..  +(0,2);
\draw[wei] (8,2) -- +(2,2);
\end{tikzpicture}
\]

\Item\label{rel14}
\[
\scalefont{0.8}\begin{tikzpicture}[very thick,baseline,scale=.5]
\draw(-3,0) +(-1,-1) -- +(1,1);
\draw[wei](-3,0) +(1,-1) -- +(-1,1);
\fill (-3.5,-.5) circle (3pt);
\node at (-1,0) {=};
\draw(1,0) +(-1,-1) -- +(1,1);
\draw[wei](1,0) +(1,-1) -- +(-1,1);
\fill (1.5,.5) circle (3pt);

\draw[wei](1,0) +(3,-1) -- +(5,1);
\draw(1,0) +(5,-1) -- +(3,1);
\fill (5.5,-.5) circle (3pt);
\node at (7,0) {=};
\draw[wei](5,0) +(3,-1) -- +(5,1);
\draw(5,0) +(5,-1) -- +(3,1);
\fill (8.5,.5) circle (3pt);
\end{tikzpicture}
\]
\end{enumerate}
Finally, we have the following non-local idempotent relation.
\begin{enumerate}[resume, label=(2.\arabic*)]
\item\label{rel15}
Any idempotent in which a solid strand is $\ell n$ units to the left of the leftmost red-strand is referred to as unsteady and set to be equal to zero.
\end{enumerate}
The product $d_1 d_2$ of two diagrams $d_1,d_2 \in A(n, \theta, \kappa)$ is given by putting $d_1$ on top of $d_2$.
This product is defined to be $0$ unless the southern border of $d_1$ is given by the same loading as the northern border of $d_2$ with residues of strands matching in the obvious manner, in which case we obtain a new diagram with loading and labels inherited from those of $d_1$ and $d_2$.  
\end{defn}

\begin{prop}[{\cite[Section 4]{MR3732238}}]\label{grsubsec}
There is a $\ZZ$-grading on the  algebra ${\bf A}(n,\theta, \kappa)$   as follows:
 $(i)$ dots have degree 2;
 $(ii)$ 
the crossing of two strands has degree 0, unless they have the
same label, in which case it has degree $-2$;
 $(iii)$  the crossing of a solid strand with label $i$ and a ghost has degree 1 if the ghost has label $i-1$ and 0 otherwise;
 $(iv)$ 
the crossing of a solid strand with a red strand has degree 0, unless they have the same label, in which case it has degree 1.  
In other words,
\[
\deg\tikz[baseline,very thick,scale=1.5]{\draw
(0,.3) -- (0,-.1) node[at end,below,scale=.8]{$i$}
node[midway,circle,fill=black,inner
sep=2pt]{};}=
2 \qquad \deg\tikz[baseline,very thick,scale=1.5]
{\draw (.2,.3) --
(-.2,-.1) node[at end,below, scale=.8]{$i$}; \draw
(.2,-.1) -- (-.2,.3) node[at start,below,scale=.8]{$j$};} =-2\delta_{i,j} \qquad  
\deg\tikz[baseline,very thick,scale=1.5]{\draw[densely dashed] 
(-.2,-.1)-- (.2,.3) node[at start,below, scale=.8]{$i$}; \draw
(.2,-.1) -- (-.2,.3) node[at start,below,scale=.8]{$j$};} =\delta_{j,i+1} \qquad \deg\tikz[baseline,very thick,scale=1.5]{\draw (.2,.3) --
(-.2,-.1) node[at end,below, scale=.8]{$i$}; \draw [densely dashed]
(.2,-.1) -- (-.2,.3) node[at start,below,scale=.8]{$j$};} =\delta_{j,i-1}\]
\[
\deg\tikz[baseline,very thick,scale=1.5]{ \draw[wei]
(-.2,-.1)-- (.2,.3) node[at start,below, scale=.8]{$i$}; \draw
(.2,-.1) -- (-.2,.3) node[at start,below,scale=.8]{$j$};} =\delta_{i,j} 
\qquad \deg\tikz[baseline,very thick,scale=1.5] {\draw (.2,.3) --
(-.2,-.1) node[at end,below, scale=.8]{$i$}; \draw[wei]   (.2,-.1) -- (-.2,.3) node[at start,below,scale=.8]{$j$};} =\delta_{j,i}.
\]
\end{prop}

  Let  $\SSTT$ be any tableau of shape $\lambda$ and weight $\mu$. Associated to 
 $\SSTT$, we have a $\theta$-diagram $C_\SSTT$   with 
distinguished  solid points on the  northern and  southern borders given by ${\bf I}_\la$ and  ${\bf I}_\mu$ respectively;  
 the $n$ solid strands each connect  a northern and southern distinguished point and  
 are drawn so that they trace out the bijection determined by $\SSTT$ in such a way that we use the minimal number of
  crossings; 
 the strand terminating at point   $({\bf I}_{(r,c,m)},0)$ for $(r,c,m) \in \lambda$ carries residue equal to  $\res(r,c,m)\in I$.  
 This diagram is not unique up to isotopy, but we can choose one such diagram arbitrarily.  
Given a pair of semistandard tableaux of the same shape
$(\SSTS,\SSTT)\in\SStd (\lambda,\mu)\times\SStd (\lambda,\nu)$, we have
a diagram  ${C}_{\SSTS  \SSTT}=C_\SSTS C_\SSTT^\ast$ where $C^\ast_\SSTT$
is the diagram obtained from $C_\SSTT $ by flipping it through the
horizontal axis.
 
 \begin{thm}[{\cite[Section 2.6]{MR3732238},\cite[Theorem 3.19]{manycell}}] \label{cellularitybreedscontempt}
The $R$-algebra ${\bf A}(n,\theta,\kappa)$  is a graded cellular algebra with  basis 
\[
 \{C_{\SSTS  \SSTT} \mid \SSTS \in \SStd (\lambda,\mu), \SSTT\in \SStd (\lambda,\nu), 
\lambda,\mu, \nu \in \mptn {\ell}n\}
\]
with respect to the $\theta$-dominance order on $\mptn \ell n(h)$ and the involution $\ast$.  

 \end{thm}

The  {\sf radical} of  a finite-dimensional $A$-module $M$,  denoted   $\rad M$, is     the smallest submodule of $M$ such that the corresponding quotient is semisimple.  We then let $\rad^2 M = \rad (\rad M)$ and inductively define the {\sf    radical series}, $ \rad^i M $, of $M$ by $\rad^{i+1} M = \rad(\rad^i M)$.  We have a finite chain
\[
M  \supset \rad (M) \supset \rad^2 (M) \supset \cdots \supset \rad^i (M) \supset \rad^{i+1} (M) \supset \cdots  \supset\rad^{s} (M)= 0.
\]

\begin{thm}[{\cite[Theorem 6.2]{MR3732238}}]\label{step1}
Over $\mathbb C$, the 
  diagrammatic Cherednik algebra  ${\bf A}(n,\theta,\kappa)$ is standard Koszul. The grading coincides with the radical filtration on standard modules as follows,
$$
d_{\lambda\mu}^{{\bf A}(n,\theta,\kappa)}(t) = \sum_k [ \rad^k(\Delta^{{\bf A}(n,\theta,\kappa)}(\la))  /  \rad^{k+1}(\Delta^{{\bf A}(n,\theta,\kappa)}(\la))    : L^{{\bf A}(n,\theta,\kappa)}(\mu)] t^k   
$$
and hence $d_{\lambda\mu}(t)\in t\ZZ_{\geq 0}[t]$ for $\lambda\neq \mu \in \mptn {\ell}n$.
\end{thm}

\section{Modular representations of cyclotomic Hecke and \\  diagrammatic Cherednik algebras}\label{sec:4}

Let $\Bbbk$ be an arbitrary field of characteristic $p>0$.  
We wish to understand as much of the representation theory of    symmetric groups  and cyclotomic Hecke algebras over $\Bbbk$ as   possible. 
 As made precise  in \cite{MR3732238,manycell}, this is equivalent to understanding the representation theory of 
 diagrammatic Cherednik algebras for arbitrary weightings $\theta\in\ZZ^\ell$.
 A long-standing belief in modular representation theory of algebraic groups is that we should (first)
  restrict our attention to    fields whose characteristic is greater  than the Coxeter number of the group.  
  This is equivalent (via Ringel duality) 
to considering the  sub/quotient category of symmetric group representations labelled by partitions with at most 
  $h$  columns over a field of characteristic $p>h$.

 \begin{defn}
Let  $Q\subseteq\mptn \ell n$.
We say that $Q$ is {\sf saturated} if for any $\alpha\in Q$ and $\beta \in \mptn \ell n$ with $\beta \vartriangleleft_\theta \alpha$, we have that $\beta \in Q$. 
We say that $Q$ is {\sf cosaturated} if its complement in $\mptn \ell n$ is saturated.
We say that $Q$ is {\sf closed} if it is the intersection of a saturated  set and a co-saturated set.
\end{defn}

\begin{defn}[see {\cite[Proposition 2.4]{bs16}}]\label{bs16iguess}
Let $E$ (respectively $F$) denote a saturated (respectively co-saturated) subset of  $\mptn \ell n$, so that $E\cap F \subseteq \mptn \ell n$ is closed.  We let 
\[
e = \sum_{\mu \in E \cap F} {\sf1}_\mu \quad \text{and} \quad
f = \sum_{\mu \in F\setminus E} {\sf1}_\mu
\]
in ${\bf A} (n,\theta,\kappa)  $.
We let ${\bf A}_{E\cap F}(n,\theta,\kappa) $ denote the subquotient of ${\bf A}(n,\theta,\kappa)$ given by
\[
{\bf A}_{E\cap F}(n,\theta,\kappa) = e({\bf A}(n,\theta,\kappa)/({\bf A}(n,\theta,\kappa) f {\bf A}(n,\theta,\kappa)))e 
\]
which is  cellular with respect to the basis 
\[
\{C_{\SSTS \SSTT}\mid \SSTS \in \SStd (\lambda,\mu), \, \SSTT\in\SStd (\lambda,\nu), \, \lambda, \mu, \nu \in E \cap F\},
\]
\end{defn}

\begin{eg}
Let $\ell=1$ and let $\theta\in \ZZ$ and $\kappa \in I$
 be arbitrary.
  The set $ \mptn 1 n(h)\subseteq \mptn 1 n$ of partitions with at most $h$ columns is saturated in the $\theta$-dominance ordering.
  Over fields of characteristic $p>h$ the representation theory of the algebra 
${\bf A}_  {\mptn 1 n(h)}(n,\theta,\kappa)$
 can be understood in terms of the $p$-canonical basis \cite[Theorem 1.9]{w16}.  
 For $p\gg h$ this simplifies a great deal and can be understood combinatorially in terms of Kazhdan--Lusztig theory.  
 An iterative approach for passing from $p\gg h$ to smaller primes is the subject of  \cite{MR3316914}.  
\end{eg}

Let $\theta\in \ZZ^\ell$ be such that $ \theta_i - \theta_j >n\ell$ for any $1\leq i < j \leq \ell$.  We say that such a weighting is {\sf well separated}.  
For such a weighting, we have that ${\bf A}(n,\theta,\kappa)$ is Morita equivalent 
to the cyclotomic $q$-Schur algebra of Dipper, James, and Mathas \cite[Definition 6.1]{MR1658581} (see \cite[Section 3.3]{MR3732238}).  
This is the 
 most classical example of a $\theta$-dominance ordering.  
  This order is the worst possible ordering for our purposes.  
This is because if $Q$ is a   ``natural''  subset of $\mptn\ell n$ which is 
closed under the $\theta$-dominance order, then understanding the decomposition matrices of 
 ${{\bf A}_{Q }(n,\theta,\kappa)}   $ for $n\in \mathbb{Z}_{\geq 0}$  contains as a subproblem 
 that of understanding the entire decomposition matrix of $\mathfrak{S}_n$.   

\begin{eg}\label{funkyexample}
For example let $\ell=3$ and $\theta\in \ZZ^3$ be a  well-separated weighting and suppose 
 $Q\subseteq \mptn \ell n$ is a closed  subset.  If 
$(\lambda, \varnothing , \varnothing ), (  \varnothing , \varnothing , \nu) \in Q\subseteq  \mptn 3 n$, then  
$\{ (  \varnothing  , \mu, \varnothing ) \mid \mu \in \mptn 1 n\}\subseteq Q \subseteq  \mptn 3 n$.
In particular, for  $n\gg p$ we find that understanding ${\bf A}_Q(n,\theta,\kappa)$ is at least as difficult as 
understanding the decomposition numbers of symmetric groups without restriction on the number of columns (and therefore the problem of understanding tilting characters of 
general linear groups  over infinite fields of  characteristic strictly less than the Coxeter number \cite[Section 4]{MR1611011}).  
\end{eg}

As a remedy to the problem encountered in  \cref{funkyexample}, we propose 
studying the largest subset  $Q\subseteq \mptn\ell n$  such that no  element  of $Q$ has  
more than $h$ columns in any given component.
  Namely, we consider the set of $h$-restricted multipartitions $\mptn \ell n (h)$ as in 
\cref{hrestrictedpartitions}.   
We wish to identify  a corresponding  subcategory of  $H_n(\kappa)\mhyphen\rm mod$.  
We therefore require a weighting, $\theta\in \ZZ^\ell$, such that  $\mptn \ell n (h)\subseteq  \mptn \ell n  $ 
is a saturated subset in the $\theta$-dominance order.  
 This leads us to consider weightings of the form $\theta\in\ZZ^\ell$ such that 
$0\leq  \theta_i-\theta_j  <\ell$ for all $1\leq i < j\leq \ell$. 
  We now show that  
  it does not matter which of these weightings we choose.

\label{grmorita}
 
 
 \begin{lem}\label{seperating}
Let $e>h\ell$ and let 
  $\kappa\in I^\ell$ be    $h$-admissible.  
For   a   given  $i \in I$, an  interval of the form  $[x,x+(h+1)\ell]\subseteq \RR$ contains at most one $i$-diagonal of boxes 
$$\{
(r,c,m)
 \mid
\mathbf{I}_{(r,c,m)} \in [x,x+ h \ell] , \kappa_m+c-r=i \in I\}.$$  
\end{lem}

\begin{proof}
 
This simply follows because the distinct entries of the multicharge differ by at least  $h$  and the boxes have width $\ell$ and $e>h\ell$.  
\end{proof}

\begin{prop}\label{chooose}
Let 
 $\theta=(1,2,\dots ,\ell)$  and let $\kappa 
=
(\kappa_{1},\kappa_{2},\dots , \kappa_{\ell}) \in I^\ell$ be   $h$-admissible.  
 We have that $$ 
  {\bf A} (n,\theta, \kappa+(1^\ell)) 
  \cong  {\bf A} (n,\theta,\kappa )
\cong {\bf A} (n,\theta^\sigma, \kappa) 
\cong {\bf A} (n,\theta, \kappa^\sigma ) 
$$   
where 
  $$\kappa^\sigma=
(\kappa_{\sigma(1)},\kappa_{\sigma(2)},\dots , \kappa_{\sigma(\ell)}), 
 \theta^\sigma=
(\theta_{\sigma(1)},\theta_{\sigma(2)},\dots , \theta_{\sigma(\ell)})  ,
 \kappa+(1^\ell)=
(\kappa_{1}+1,\kappa_{2}+1,\dots , \kappa_{\ell}+1)$$  for $\sigma \in \mathfrak{S}_\ell$.   
 \end{prop}
\begin{proof}
The first isomorphism simply amounts to relabelling the underlying quiver.   The   isomorphism
$  {\bf A} (n,\theta,\kappa )
\cong {\bf A} (n,\theta^\sigma, \kappa ) $  is trivial.  
We now consider the   isomorphism $A(n,\theta,\kappa) \cong {\bf A} (n,\theta, \kappa^\sigma ) $.  For arbitrary $\kappa \in I^\ell$, it is easy to check that the natural     map 
$$
 \SStd _{\theta  ,\kappa}(\la,\mu)
\leftrightarrow
 \SStd _{\theta  ,\kappa^\sigma}(\la^\sigma,\mu^\sigma) 
  $$ is both bijective and degree preserving. 
 Therefore the result holds on the level of graded $R$-modules.  
   Given     $\alpha\in \mptn \ell n (h)$, we let $D^\alpha_\sigma $ denote any diagram
which has solid strands connecting  points $(\mathbf{I}_{(r,c,m)},0)$ to the  points 
$(\mathbf{I}_{(r,c,\sigma(m))},1)$  for $(r,c,m)\in \alpha$ and which has red strands connecting  the points $(\theta_m,0)$ to the  points 
$(\theta_{\sigma(m)},1)$ in such a way as to have the minimal number of crossings.    
We define $\varphi:  A (n,\theta,\kappa )
\to A  (n,\theta, \kappa^\sigma) $ by
$\varphi(C_{\SSTS})=D_\sigma^\mu 
C_{\SSTS} (D_\sigma^\lambda)^\ast$ for $\SSTS \in \SStd_{\theta,\kappa} (\lambda,\mu)$.  
 By \cref{seperating},   any crossings in $D_\sigma$ involve strands of non-adjacent (and non-equal) residues.  
Therefore, $C_{\SSTS} D_\sigma^\lambda= C_{\SSTS^\sigma}$ and so this is an algebra isomorphism, as required.  
 \end{proof}

\begin{rmk}\label{choiceofmulti}
Recall that we are interested in $
{\bf A}_{\mptn \ell n(h)}(n,\theta,\kappa)$ 
(as in \cref{bs16iguess}) for $e>h\ell$ and $\kappa\in I^\ell$ $h$-admissible.  
By \cref{chooose}, we can assume that 
$\theta=(1,2,\dots,\ell)$.   
 We choose to   only consider our multicharge  up to rotation and shifting residues.  
 Given $e > h\ell$,  we can suppose that 
$\kappa=(\kappa_1, \kappa_2, \dots, \kappa_\ell) \in I^\ell$   is such that 
 $$  \kappa_1<\kappa_{2}< \dots < \kappa_\ell  $$
  and such that $\kappa_\ell + h  \neq \kappa_1 \pmod e$.  
  (For $e\neq \infty$, we have abused notation by placing an ordering $\ZZ/e\ZZ$ via the natural ordering on $\{0,1,\dots,e-1\}$.)   
   \end{rmk}

 We let $\omega=(\varnothing,\dots, \varnothing,(1^n))$ and $\underline{i} \in I^n$.  We let 
 ${\sf E}_\omega^{\underline{i}}$ denote the diagram with northern and southern distinguished points given by ${\bf I}_\omega$, no crossing strands, and the $k$th solid strand decorated with the residue $\underline{i}_k \in I$.  
 We set ${\sf E}_\omega = \sum_{\underline{i} \in I^n} {\sf E}_\omega^{\underline{i}}$.  It is shown in \cite[Theorem 4.5]{manycell} that 
 ${\sf E}_\omega {\bf A}(n,\theta,\kappa) {\sf E}_\omega = H_n(\kappa)$ for any weighting $\theta\in \ZZ^\ell$.  
For any fixed weighting and multicharge,  it is easy to see that there is a corresponding  degree-preserving bijective map   $\varphi: \Std(\lambda)\to \SStd(\lambda,\omega)$ (see \cite[Proposition 4.4]{manycell}).
 
 \begin{defn} \label{algebradefn}
For $\lambda \in \mptn \ell n$,  we set  ${\sf 1}_\lambda= C_{\SSTT^\lambda\SSTT^\lambda}$ for $\SSTT^\lambda$ the unique element of $\SStd(\lambda,\lambda)$.  
We set
\begin{equation}\label{firstiso1}
A_h(n,\theta,\kappa)=(\textstyle \sum_{\alpha\in\mptn\ell n}{\sf 1}_\alpha){\bf A}_{\mptn \ell n (h)}(n,\theta,\kappa)(\textstyle\sum_{\alpha\in\mptn\ell n}{\sf 1}_\alpha).
\end{equation}
 
\end{defn}

\begin{thm}\label{decompnumbersarethesame}
Let $\theta=(1,2,\dots,\ell)$ and $\kappa \in I^\ell$ be $h$-admissible.  
We have that 
\begin{equation}\label{firstiso}
{\sf E}_\omega  {\bf A}_{\mptn \ell n (h)}(n,\theta,\kappa)   
 {\sf E}_\omega  \cong {\Q _{h,\ell,n}(\kappa)}
\end{equation}
and 
\begin{equation}\label{firstiso2}
  {\bf A}_{\mptn \ell n (h)}(n,\theta,\kappa)  {\sf E}_\omega
  \otimes_{\Q _{h,\ell,n}(\kappa)}  {\sf E}_\omega
  {\bf A}_{\mptn \ell n (h)}(n,\theta,\kappa)
\cong  {\bf A}_{\mptn \ell n (h)}(n,\theta,\kappa) \end{equation}
 \begin{equation}\label{firstiso3}
  {\bf A}_{\mptn \ell n (h)}(n,\theta,\kappa) (\textstyle \sum_{\alpha\in\mptn\ell n}{\sf 1}_\alpha) 
  \otimes_{A_h(n,\theta,\kappa)  }  (\textstyle \sum_{\alpha\in\mptn\ell n}{\sf 1}_\alpha) 
  {\bf A}_{\mptn \ell n (h)}(n,\theta,\kappa)
\cong  {\bf A}_{\mptn \ell n (h)}(n,\theta,\kappa) \end{equation} 
 as graded  $R$-algebras.  Therefore all three algebras  are graded Morita equivalent.   
 The cellular structures on 
the  subalgebras      are obtained from that of $  {\bf A}_{\mptn \ell n (h)}(n,\theta,\kappa)  $ by truncation as follows, 
\begin{align*}
 {\sf E}_\omega C_{\SSTS  \SSTT}  {\sf E}_\omega =&\begin{cases}
\psi_{\sts\stt}^{\theta,\kappa} &\text{if }\varphi(\sts)=\SSTS  \in \SStd(\la,\omega) , \varphi(\stt)=\SSTT\in \SStd(\la,\omega) \\
 0		&\text{otherwise}
 \end{cases}
 \\
(\textstyle \sum_{\lambda\in\mptn\ell n}{\sf 1}_\alpha) C_{\SSTS  \SSTT}  (\textstyle \sum_{\alpha\in\mptn\ell n}{\sf 1}_\alpha) =&\begin{cases}
 C_{\SSTS  \SSTT} &\text{if }\SSTS \in \SStd^+(\lambda,\mu), \SSTT  \in \SStd^+(\lambda,\nu) \\
 0		&\text{otherwise.}
 \end{cases}
\end{align*}

  \end{thm}

\begin{proof} 
The truncation of cellular bases follows as in \cite[Theorem 3.12]{manycell}.
For the   isomorphisms in \cref{firstiso}, see \cite[Theorem 3.12]{manycell}.  
We now prove the   isomorphism \cref{firstiso2} (the isomorphism \cref{firstiso3} can be  proven in an identical fashion).  
 That this map is injective is clear.  
To prove that the map is surjective, it is enough to check that 
\begin{equation}\label{notcimple}C^\ast_{\stt^\lambda}C_{\stt^\lambda}=C_{\SSTT^\lambda } \end{equation}
for 
$ {\stt^\lambda }\in \Std (\lambda )=\SStd(\lambda,\omega)$
and 
$ {\SSTT^\lambda }\in \SStd (\lambda,\lambda)$.  In   \cite[Proof of Theorem 8.2]{manycell}  it is shown that 
$  C_{\stt^\lambda}C^\ast_{\stt^\lambda} $ is an idempotent, and so 
$
  C_{\stt^\lambda}C^\ast_{\stt^\lambda} = ( C_{\stt^\lambda}C^\ast_{\stt^\lambda})( C_{\stt^\lambda}C^\ast_{\stt^\lambda})
  =
 C_{\stt^\lambda}	(C^\ast_{\stt^\lambda} C_{\stt^\lambda})C^\ast_{\stt^\lambda}
$ and therefore $ C^\ast_{\stt^\lambda} C_{\stt^\lambda} =C_{\SSTT^\lambda } $, as required.   \end{proof}

\begin{cor}
 
The algebra ${A_h(n,\theta,\kappa)  }$ is a graded cellular algebra with respect to the basis 
\[
 \{C_{\SSTS  \SSTT} \mid \SSTS \in \SStd^+ (\lambda,\mu), \SSTT\in \SStd^+ (\lambda,\nu), 
\lambda,\mu, \nu \in \mptn {\ell}n(h)\}
\]
with respect to the $\theta$-dominance order on $\mptn \ell n(h)$ and the involution $\ast$.  
\end{cor}
We refer to the resulting 
cell-modules (as in \cref{cellmodule}) as the  {\sf Weyl modules} of $A_h(n,\theta,\kappa)$, and denote them by 
$$
\Delta(\lambda) = \{ C_\SSTS    \mid \SSTS \in \SStd^+(\lambda,\mu) \text{ for some }\mu \in \mptn \ell n(h)\}.
$$

\begin{rmk}\label{Ringel2}
We note that the Weyl/simple modules of  $A_h(n,\theta,\kappa)$
are indexed by multipartitions with at most $h$ columns (rather than at most $h$ rows) and that $A_h(n,\theta,\kappa)$ is constructed as a quotient (not a subalgebra) of $A(n,\theta,\kappa)$.  
Specialising to the case that $\ell=1$ and $e=p$, we have that 
 $A_h(n,\theta,\kappa)$ 
 is the {\em Ringel dual} of the classical  Schur algebra. 
 There is a well-known duality underlying the combinatorics 
 of the Schur algebra  and its Ringel dual,   given by  
 identifying a  partition with its transpose \cite[Chapter 4]{Donkin}. 
\end{rmk}

\begin{rmk}
Note that \cref{notcimple} does not hold for $\la \not \in \mptn \ell n(h)$.  In particular, it does not hold 
 if $\lambda$ labels a simple module of  ${\bf A}(n,\theta,\kappa)$ which is  
 killed by the Schur functor.  
\end{rmk}

For the remainder of the paper, we shall develop  the combinatorics  of  
$A_{h}(n,\theta,\kappa)$ and prove results concerning their representation theory.  
We  hence deduce results  (though the graded Morita equivalence) concerning  
the  algebras $\Q _{h,\ell,n}(\kappa)$.  

%
%
%

\section{Alcove geometries   and path bases for  diagrammatic Cherednik algebras}\label{sec:5} 

Following our discussion in \cref{sec:4},   for the remainder of the paper  we set 
$\theta=(1,\dots,\ell) \in \ZZ^\ell$.  
In the main result of this section, we prove that $A_h(n,\theta,\kappa)$ has a basis indexed by pairs of paths in a certain alcove geometry.  
This involves first providing an inductive construction of semistandard tableaux and then embedding these tableaux into Euclidean space via the inductive construction.  

\subsection{An inductive construction of semistandard tableaux }
The purpose of this section is to provide an inductive construction of semistandard tableaux (and hence the basis of $A_h(n,\theta,\kappa)$) under the assumption that $e>h\ell$.  We also show that it is impossible to construct tableaux in an inductive fashion for $e <   h\ell$ in general.

\begin{thm} \label{789324578935247832592378375893978524978235497835241}
Let    $\lambda , \mu \in  \mptn \ell {n}(h)$. 
We let   
$X$ denote the  least dominant element of   ${\rm Rem}(\mu)$.  
Suppose that $\res(X)= x\in I$.  We have 
a  bijection   $$  
\bigsqcup_{
 Y\in {\rm Rem}_x(\lambda)
}
\SStd^+_{n-1}(\lambda - Y, \mu-X)  \to  \SStd^+_n({\lambda},{\mu})  $$ given by $\SSTT \to \SSTT_Y$ where
$$ {\SSTT}_Y (r,c,m)=
\begin{cases}
X  &\text{if }(r,c,m)=Y \\
\SSTT(r,c,m) &\text{otherwise.}
\end{cases}
$$  
 \end{thm}

\begin{proof}

We have assumed that $X=(r',c',m')$ is the box  minimal  in the dominance ordering in ${\rm Rem}({\mu})$ and 
that ${\mu} \in \mptn \ell n$. Therefore 
\begin{equation}\label{label}[\mathbf{I}_X \pm  \ell,\infty) \cap \{\mathbf{I}_{(r,c,m)} \mid (r,c,m)\in  {\mu} \} \subset   [\mathbf{I}_X \pm  \ell,\mathbf{I}_X +  h \ell).  
\end{equation}
By assumption, the boxes   $\{(r'+i,c'+1+i,m') \mid i \in \ZZ \}$
and  $\{(r'+1+i,c'+i,m') \mid i \in \ZZ \}$ consist of boxes  of residue $x+1$ and $x-1$   
respectively.
We claim that these are the unique $(x\pm1)$-diagonals in the region $[\mathbf{I}_X \pm  \ell,\infty)$.  
For the $(x+1)$-diagonal, this follows immediately from  \cref{seperating,label}.  
 We now consider the $(x-1)$-diagonal. 
    By \cref{seperating,label}, it is enough to show that there is no $(x-1)$-diagonal in the region 
$     [\mathbf{I}_X+h\ell,\mathbf{I}_X +  (h+1) \ell)$.     
Given $X \in \mu$, this region contains an $(x-1)$-diagonal if and only if 
 $X=(r,h,m)$ for $1\leq m<\ell$ and
$(r,1,m+1)  \in \mu$, with 
  $\res(r,h,m)= \res(r+1,1,m+1)$.  
However, 
  $(r+1,c,m+1)\in \mu$ is to the right of $(r,h,m)$   for all $1\leq c\leq h$.  
  This contradicts our assumption on $X$ and so 
the claim follows.  

We first show surjectivity.    
Namely, we shall  let $\SSTS \in \SStd^+_n( {\lambda}, {\mu})$ and 
we shall show that $\SSTS^{-1}(X)=(r,c,m)\in {\rm Rem}(\lambda)$.  
Assume that $X=(r',c',m') \in \mu$ and that 
$\SSTS^{-1}(X)=(r,c,m)\in \lambda$ is {\em not}  a removable box.  
In which case either   $(r,c+1,m)$ or $(r+1,c,m)$ is a box in the Young diagram of $ {\lambda}$. 
If 
$(r,c+1,m)$  (respectively $(r+1,c,m)$) is a box in the Young diagram of ${\lambda}$, then $\SSTS(r,c+1,m) > \SSTS(r,c,m)-\ell$ 
(respectively  $\SSTS(r+1,c,m) > \SSTS(r,c,m)+\ell$)  by condition $(iii)$
(respectively condition $(ii)$) 
of   \cref{semistandard:defn}.  
Therefore $\SSTS(r,c+1,m)=(r'+i,c'+1+i,m')$ and  $\SSTS(r+1,c,m)=(r'+j,c'+1+j,m')$ for some $i,j\in \ZZ_{\geq 0}$
as these are the unique $(x\pm1)$-diagonals in the region $ [\mathbf{I}_X - \ell,\infty)$  by the above. 
 However, this contradicts our assumption that $X \in {\rm Rem}(\mu)$.  
Therefore we conclude that the box $\SSTS^{-1}(X)$ is indeed a removable box of the partition ${\lambda}$.  
As $\SSTS^{-1}(X)$ is a removable box of ${\lambda}$ (with residue equal to that of $X$) 
it is clear that the map is surjective.  

It remains to show that the map is injective.  Let $Y\in {\rm Rem}_x(\lambda)$ and  $\SSTT \in \SStd^+_{n-1}(\lambda-Y,\mu-X)$ where $X\in \Rem_x(\mu)$ is the minimal removable box  in the dominance ordering.
By definition, each tableau $ {\SSTT}_Y$  is distinct (and non-zero) and therefore it only remains to check that 
${\SSTT}_Y$ satisfies the conditions of being semistandard. 
We recall that the only $(x\pm1)$-boxes in the region 
$ [\mathbf{I}_X - \ell,\infty)$  are of the form $\{(r'+1+i,c'+i,m') \mid i \in \mathbb{Z}_{\geq 0}\}$ and  $\{(r'+j,c'+1+j,m') \mid j \in \mathbb{Z}_{\geq 0}\}$ respectively.  
Therefore condition $(ii)$ (respectively $(iii)$) of \cref{semistandard:defn} is empty because   
the intersection of these respective sets with  $\mu$ is empty.     Therefore the result follows.   
\end{proof}

\begin{eg}
We now provide a counter-example to the above theorem for  $\ell=1$ and $e\leq  h$.  
We let $\lambda=(9,6^3)$ and $\mu=(6^2,5^3)$ and $e=3$ (the multicharge and weighting can be chosen arbitrarily).  We let $\SSTT$ be the semistandard tableau determined by
$ 
\SSTT(r,c,m)=\mathbf{I}_{(r,c,1)}
$ 
for $(r,c,1) \in \lambda \cap \mu$; and for $x \in \lambda \setminus \lambda \cap \mu$ we set 
$$\SSTT(1,7,1)=\mathbf{I}_{(5,2,1)}, 
\SSTT(1,8,1)=\mathbf{I}_{(5,3,1)}, \SSTT(1,9,1)=\mathbf{I}_{(5,4,1)},
\SSTT(3,6,1)=\mathbf{I}_{(5,5,1)}, \SSTT(4,6,1)=\mathbf{I}_{(5,1,1)}.
$$
We note that there are no removable nodes common to both $\lambda$ and $\mu$ so there is no obvious way to   construct the  tableau     inductively.   
\end{eg}

Using \cref{789324578935247832592378375893978524978235497835241}
we are now able to inductively define the component word of both a multipartition $\mu$ and a semistandard tableau of weight $\mu$.    
Our ability to build the tableau inductively according to (the component word of) its weight multipartition will be the key ingredient in our embedding of tableaux into Euclidean space.

\begin{defn} \label{hellotheremr!}
Let  $  \mu \in \mptn \ell n (h)$. 
We define the {\sf component word} of $\mu$  to be the series of multipartitions
\[
\varnothing=  \mu^{(0)}  \xrightarrow{+X_1}
\mu^{(1)}  \xrightarrow{+X_2}
\mu^{(2)}  \xrightarrow{+X_3} 
\cdots 
\xrightarrow{+X_{n-1}} 
\mu^{(n-1)}  \xrightarrow{+X_n}
\mu^{(n)}  = \mu
\]
where   $X_{k}=(r_k,c_k,m_k)$ is the least  dominant removable node of the partition $\mu^{(k)}\in\mptn \ell {k}(h)$. 
Let  $\lambda ,  \mu \in \mptn \ell n (h)$ and  $\SSTT\in\SStd^+_n(\lambda,\mu)$.
We define the  {\sf component word} of $\SSTT$  to be the series of semistandard tableaux 
$$
{\sf T}_{\leq 0},  {\sf T}_{\leq 1}, \dots ,  {\sf T}_{\leq n}
$$
such that $\SSTT_{\leq k}\in \SStd^+_k(-,\mu^{(k)} )$ for $0\leq k <n$ is obtained from 
$\SSTT_{\leq k+1}\in \SStd^+_{k+1}(-,\mu^{(k+1)} )$ via the isomorphism of \cref{789324578935247832592378375893978524978235497835241}.    
For $0\leq k \leq n$, we let $\lambda^{(k)}=\Shape(\SSTT_{\leq k})$ and  $Y_k\in {\rm Rem}_k(\lambda^{(k)})$ be such that $\lambda^{(k-1)}+\{Y_k\} = \lambda^{(k)}$.  We shall also refer to the ordered sequence of nodes $(Y_1,Y_2, \dots, Y_n)$ as the component word of $\SSTT$ 
(as each clearly determines the other: $\SSTT(Y_k)={\bf I}_{X_k}$ for $1\leq k \leq n$).  
 \end{defn}

\begin{rmk}Given $\mu \in \mptn \ell n$, we note that the component word of the multipartition $\mu$ is equal to the component word of the tableau  $\SSTT^\mu$.  
 \end{rmk}

\begin{eg}\label{egeg}
Let $e=6$, $h=1$, $\ell=3$, and $\kappa=(0,2,4)\in (\ZZ/7\ZZ)^3$.  
 For $\mu=((1^5),\varnothing,\varnothing)$, we have that the component word of $\mu$ is given by the ordered sequence of nodes 
$$
 ((1,1,1),(2,1,1),(3,1,1),(4,1,1),(5,1,1))
$$
for $\lambda=((1^2),(1),(1^2))$ there is a unique  element of $\SStd^+(\lambda,\mu)$ with component word  
$$
 ((1,1,1),(2,1,1),(1,1,3),(2,1,3),(1,1,2)).  
$$

\end{eg}

\begin{defn}
Given $\lambda,\mu\in \mptn \ell n(h)$, 
we define the {\sf    degree} of   tableau $\SSTT\in \SStd^+_n(\lambda,\mu)$ via the component word as follows.  
 If  $$
[{\sf T}_{\leq 0}] \xrightarrow{+Y_1}  [{\sf T}_{\leq 1}] \xrightarrow{+Y_2}  \cdots  \xrightarrow{+Y_n}   [{\sf T}_{\leq n}]
$$
is the component word of the tableau $\SSTT$, then we set 
\begin{align*}
\deg(Y_{k}) =&
\big|\{\text{addable $r_k$-nodes of $ \Shape(\SSTT_{\leq {k-1}})$ to the right of $Y_{k}$} \}\big| 
\\ &  -\big|\{\text{removable $r_k$-nodes  of $ \Shape(\SSTT_{\leq {k-1}})$ to the right of $Y_{k}$} \}\big| 
\end{align*}
where $r_k\in \ZZ/e\ZZ$ is the residue of the node $Y_{k} $.  We set $\deg(\SSTT) = \sum_{1\leq k\leq n}\deg(Y_k)$.  
\end{defn}
 
%

\begin{prop}
Given $\SSTT \in \SStd^+_n(\la,\mu)$, we have that $\deg(\SSTT)= \deg(C_\SSTT)$.  
\end{prop}

\begin{proof}
We can inductively  construct the element $C_\SSTT$ by drawing each strand, one at a time, according to the ordering of \cref{hellotheremr!} (see \cref{howtobuildshit}).    
In more detail: at the $k$th stage, we add a strand to the diagram connecting the southern point $(\mathbf{I}_{Y_k},0)$ to the northern point $(\mathbf{I}_{X_k},1)$   (making sure that we draw this strand so as to include the minimal number of crossings with strands from earlier in the process).  
We let $A_k$ denote the strand connecting the southern point $\mathbf{I}_{Y_k}$ to the northern point $\mathbf{I}_{X_k}$.
By construction, 
$\deg(C_{\SSTT_{\leq k}})-\deg(C_{\SSTT_{\leq k-1}}) $
is equal to the number of crossings of  $A_k$ with strands $A_i$ for $1\leq i <k$,  each crossing counted with degree given  by \cref{grsubsec}.  
We shall show that 
$$\deg(C_{\SSTT_{\leq k}})-\deg(C_{\SSTT_{\leq k-1}}) = \deg(Y_{k})$$
and hence deduce the result.  
We shall set $Y_k=(p_k,q_k,t_k)$ and  $X_k=(r_k,c_k,m_k)$ and $\res(A_k)=x_k \in \ZZ/e\ZZ$.  
Clearly, the only crossings  in $C_{{\SSTT^{(n)}}}$ which are not in $C_{{\SSTT^{(n-1)}}}$ involve 
strands labelled by $1 \leq k < n$ 
such that either 
\begin{itemize}[leftmargin=*] 
\item[ $(i)$]  $\mathbf{I}_{(p_k,q_k,t_k)} < \mathbf{I}_{(p_n,q_n,t_n)}$ and   $\mathbf{I}_{ (r_k,c_k,m_k)} > \mathbf{I}_{(r_n,c_n,m_n)}$ or 
\item[ $(ii)$] $\mathbf{I}_{(p_k,q_k,t_k)} > \mathbf{I}_{(p_n,q_n,t_n)}$ and   $\mathbf{I}_{ (r_k,c_k,m_k)} < \mathbf{I}_{(r_n,c_n,m_n)}$.
\end{itemize} 
We are only interested in those crossings labelled by a box $(p_k,q_k,t_k)$ of residue  $x_n-1$, $x_n$, or $x_{n} +1 \in I$ by \cref{grsubsec}.  
We write 
\begin{equation}\label{980980980098098098890}
C_{\SSTT_{\leq n}}
=       \overline{C}_{\SSTT_{\leq n-1}} \times {\sf1}_{\lambda}^{\lambda - Y_n +X_n}
\end{equation}
where 
\begin{itemize}[leftmargin=*] \item[$(1)$]  we obtain $\overline{C}_{\SSTT_{ \leq n-1}}$
from  
$C_{\SSTT_{ \leq n-1}}$ by adding a  vertical solid strand with $x$-coordinate 
$ {\bf I}_{X_n}$;
\item[$(2)$]    we obtain
${\sf1}_{\lambda}^{\lambda - Y_n +X_n} $  from  ${\sf 1}_{\la - Y_n }$
by adding a solid strand from        $(\mathbf{I}_{Y_n},0)$ to  $({\bf I}_{X_n} ,1)$ in such a way as to create no double-crossings.  
\end{itemize}  
By construction, any crossing as in $(i)$ (or $(ii)$) occurs in the diagram 
$C_{\SSTT_{ \leq n-1}}$ (or ${\sf1}_{\lambda}^{\lambda - Y_n +X_n} $) in the factorisation of \cref{980980980098098098890}.   
Now, recall  our assumption that $X_n=(r_n,c_n,m_n)$ is the rightmost removable node of $\mu$.   
Arguing as  in the proof of \cref{789324578935247832592378375893978524978235497835241}, we deduce that 
\begin{align}\label{2nsjljnkghjkfghjfhjkldghjkjhk}
(\mathbf{I}_{(r_n-1,c_n,m_n)}  ,\infty)&\cap \{ \mathbf{I}_{(r_k,c_k,m_k)}	\mid   x_k = x_n+1, 1\leq k <n 	\}= \varnothing, \\
\label{2nsjljnkghjkfghjfhjkldghjkjhk2}
(\mathbf{I}_{(r_n,c_n,m_n)},\infty)		&\cap 		\{ \mathbf{I}_{(r_k,c_k,m_k)}	\mid x_k = x_n , 1\leq k <n	\}= \varnothing, \\
\label{2nsjljnkghjkfghjfhjkldghjkjhk3}
(\mathbf{I}_{(r_n,c_n-1,m_n)}  ,   \infty) &\cap \{ \mathbf{I}_{(r_k,c_k,m_k)}	\mid x_k = x_n-1 , 1\leq k <n	\}= \varnothing.
\end{align}
Therefore  if $A_k$  and $A_n$ are crossing strands of adjacent (or equal) residue, then  we are in case $(ii)$ above.  
  In particular,  
$$\deg(\overline{C}_{\SSTT_{ \leq k-1}})=\deg( {C}_{\SSTT_{ \leq k-1}}).$$ 
Therefore we only need consider crossings of the strand $A_n$ within the  diagram ${\sf1}_{\lambda}^{\lambda - Y_n +X_n} $.
 By \cref{2nsjljnkghjkfghjfhjkldghjkjhk,2nsjljnkghjkfghjfhjkldghjkjhk2,2nsjljnkghjkfghjfhjkldghjkjhk3}, the
strands $A_n$  and  $A_k$  cross   if and only if 
 the box $Y_k$ occurs to the {\em right} of $Y_n$.  
 It remains to check that the total degree contribution  of these crossings is given by the total 
number of addable $x_n$-nodes minus the total number of removable $x_n$-nodes (to the right of $Y_n$)  as claimed.  
In order to do this, we first require some notation.  
Let $1\leq  m \leq \ell$  and let $(r,c,m) \in  \lambda  \cup  
{\rm Add}(\lambda  )$ 
be a box of residue  $x_n\in I$.  
We refer to the set of nodes 
\[
\mathbf{D}= \{(a,b,m)\in \lambda^{(k-1)}
   \mid a-b \in \{r-c-1,r-c,r-c+1\} \}
\]
as the associated {\sf    $x_n$-diagonal}.  
If $a-b$ is greater than, less than, or equal to zero,
we say that the $x_n$-diagonal is to  the left of, right of, or centred on $\theta_m$, respectively. 

Clearly all boxes of $\lambda$ of residue $x_n-1$, $x_n$, or $x_n+1$ belong  to some $x_n$-diagonal. 
 We say that the $x_n$-diagonal containing the node $(r,c,m)\in \lambda   \cup  
{\rm Add}(\lambda  )$ is to the right of  the node $Y_n$ if ${\bf I}_{(r,c,m)}> \mathbf{I}_{Y_n}$. 
We have already seen that each non-zero degree crossing of  the  $A_n$-strand (or its ghost) occurs with a strand belonging to an $x_n$-diagonal to the right of $Y_n$.  

Let ${\bf D}$ be an $x_n$-diagonal in $ \la$  to the right of $Y_{ k }$.  Suppose that  ${\bf D}$ has an addable $x_n$-node, which we denote by $(r,c,m) \in  \la$.  If   $(i)$  $r-c<0$, or $(ii)$  $r-c>0$, $(iii)$ $r-c=0$, then there are 
\begin{itemize}[leftmargin=*]
\item[$(i)$]  a total of $r$ distinct solid $x_n$-strands, $r+1$ distinct ghost  $(x_n-1)$-strands,    $r $   distinct solid $(x_n+1)$-strands, and no red strands,  
\item[$(ii)$] a total of  $c$ distinct solid $x_n$-strands, $c+1$  distinct solid $(x_n+1)$-strands,  and  $c $   distinct  ghost $(x_n-1)$-strands, and no red strands,  
\item[$(iii)$] a total of  $r$ distinct solid $x_n$-strands, $r$  distinct solid $(x_n+1)$-strands, and  $r$   distinct ghost $(x_n-1)$-strands, and 1 red strand (note that $r=c$), 
\end{itemize} 
within the region $$( [\mathbf{I}_{(r,c,m) }-2\ell,\mathbf{I}_{(r,c,m) }+2\ell]\times \{0\} ) \cap  {\sf1}_{\lambda } .$$   
The solid $A_n$-strand   crosses all of these strands and   the sum over these crossings has total  degree equal to  1.  
The other cases (an $x_n$-diagonal with a removable $x_n$-node, or an $x_n$-diagonal with no addable or removable $x_n$-node) can be checked in a similar fashion (and have total degree contribution $-1$ or $0$ respectively).
 \end{proof}

%

\subsection{The geometry}\label{1.1}
In this section, we are going to consider a variant of the classical alcove geometries encountered in Lie theory.  
Fix integers $h,\ell \in \ZZ_{>0}$ and $e \in \ZZ_{>0} \cup \{\infty\}$.
For each $ 1\leq \I \leq h$ and $ 0\leq \M < \ell$ we let
$\varepsilon_{h\M+\I}$ denote a
formal symbol, and set
\[
{\mathbb E}_{h,\ell }
=\bigoplus_{
\begin{subarray}c 
1\leq \I \leq  h   \\ 
0 \leq \M <  \ell  
 \end{subarray}
} \mathbb{R}\varepsilon_{h\M+\I}
\]
to be the associated $\ell h$-dimensional real vector space.
We have an inner product $\langle \, , \, \rangle$ given by extending
linearly the relations 
\[
\langle \varepsilon_{h\M+\I} , \varepsilon_{ht+\J} \rangle= 
\delta_{\I,\J}\delta_{t,\M}
\]
for all $1\leq \I, \J \leq h$ and $0 \leq  \M , t< \ell $, where
$\delta_{i,j}$ is the Kronecker delta.  We let $\Phi$ (respectively
$\Phi_0$) denote the root system of type $A_{\ell h -1}$ (respectively
of type $A_{h -1} \times A_{h -1}  \times \dots A_{h -1} $) consisting of the
roots 
$$\{\varepsilon_{h\M+ \I}-\varepsilon_{ht+ \J}\mid 1\leq \I,\J
\leq h \text{ and } 0\leq \M,t < \ell\text{ with } (i,m)\neq (j,t)\}$$ 
respectively
$$\{\varepsilon_{h\M+ \I}-\varepsilon_{h\M+ \J}\mid 1\leq \I,\J \leq h
\text{ with } \I\neq \J \text{ and } 0\leq \M < \ell\}.$$
 Suppose that  $e >h \ell$.
    We identify $\lambda \in \mptn \ell n(h)$ with a point  in $\mathbb{E}_{h\ell} $ via the   map 
$$ ( \lambda^{(1)} ,\ldots, \lambda^{(\ell )} )\mapsto \sum_{\begin{subarray}c
1\leq i \leq \ell   	\\
1\leq j \leq h
\end{subarray}}( \lambda^{(\M)})^T_\I \varepsilon_{h(m-1)+i},$$ 
(where the $T$ denotes the transpose partition).  
For example,  
  we have that $((2^2,1^2),(2,1)) \mapsto 4\varepsilon_1+2\varepsilon_2+2\varepsilon_3+\varepsilon_4$.  
(For Lie theorists who find the appearance of the transpose of the partition peculiar, we refer   to  \cref{Ringel2}.) 
For $e\in\mathbb{Z}_{>0}$ (respectively $e=\infty$) we assume that $\kappa \in I^\ell$ is $h$-admissible (respectively $1$-admissible).  
Given    $r\in \ZZ$ and  $\alpha \in \Phi$ we let  
$ s_{\alpha,re }$ denote the reflection which acts on ${\mathbb E}_{h,\ell}$ by 
\[s_{\alpha, re}x =x-(\langle x, \alpha \rangle -re) \alpha.\]
Given  $e \neq \infty$ we let $W^e $ be the affine reflection group
generated by the reflections 
$$\mathcal{S} = \{s_{\alpha, re} \mid   \alpha \in \Phi, r\in\mathbb{Z}\}$$
and let $W^e_0$  denote the   parabolic subgroup generated by 
$$\mathcal{S}_0=\{s_{\alpha, 0} \mid   \alpha \in \Phi_0\}.$$
If $e=\infty$ then we let $W^e$ be the finite reflection group
generated by the reflections 
$$\mathcal{S} = \{s_{\alpha, 0} \mid   \alpha \in \Phi\}$$
and let $W^e_0$  denote the   parabolic subgroup generated by 
$$\mathcal{S}_0=\{s_{\alpha, 0} \mid   \alpha \in \Phi_0\}.$$
  We shall consider a shifted action  of these groups  on ${\mathbb E}_{hl}^{\circledcirc}$ by the element
   $\rho = (\rho_1,\rho_2, \ldots,\rho_\ell) \in I^{h }$ where 
 $$\rho_i = (e-\kappa_i, e-\kappa_i-1, \dots,e-\kappa_i-h+1)\in I^{h\ell}.$$   
  Given an element $w\in W  ^e$,  we define the ``dot'' action of 
  $w$ on ${\mathbb
   E}_{h,\ell}$ by
\[
w\cdot_{\rho}x=w(x+\rho)-\rho.
\]
%
We let ${\mathbb E} ({\alpha, re})$ denote  the affine hyperplane
consisting of the points  
$${\mathbb E} ({\alpha, re}) = 
\{ x\in{\mathbb E}_{h,\ell} \mid  s_{\alpha,r } \cdot x = x\} .$$
We say that a point $\lambda\in {\mathbb
  E}_{h,\ell}$ is $e$-regular if it does not lie on any hyperplane.  Our assumption that $e>h\ell$ 
  implies that $e$-regular lattice points do exist.  
In particular, we let $\circledcirc$ denote the origin $(0, \dots ,0) \in {\mathbb
  E}_{h,\ell}$; note that our choice of $\rho$ ensures that
$\circledcirc$ is $e$-regular.

Given a
hyperplane ${\mathbb E} ( \alpha,re)$ we 
remove the hyperplane from ${\mathbb E}_{h,\ell}$ to obtain two
distinct subsets ${\mathbb 
  E}^{\great}(\alpha,re)$ and ${\mathbb E}^{\less}(\alpha,re)$
where $\circledcirc \in {\mathbb E}^{\less }(\alpha,re)$.  We define
${\mathbb E}^{\greatoreq }(\alpha,re) = {\mathbb E} (\alpha,re)\cup
{\mathbb E}^{\great }(\alpha,re)$ and similarly ${\mathbb
  E}^{\lessoreq }(\alpha,re) = {\mathbb E} (\alpha,re)\cup {\mathbb
  E}^{\less}(\alpha,re)$.  The dominant Weyl chamber, denoted
${\mathbb E}_{h,\ell }^\circledcirc $, is set to be 
$${\mathbb E}_{h,\ell }^{\circledcirc}=\bigcap_{ \begin{subarray}c
 \alpha \in \Phi_0
 \end{subarray}
 } {\mathbb E}^{\less} (\alpha,0).$$

\begin{eg}
For $\ell=1$ we obtain the parabolic affine geometry  which
controls the representation theory of the (quantum) general linear
group of $(h\times h)$-matrices in (quantum) characteristic $e$. 
\end{eg}

\begin{eg}
Setting $h=1$ we have $W_0^e=1$ and we obtain the (non-parabolic)
affine geometry which controls the representation theory of the
Kac--Moody algebras of type $\widehat{A}_{\ell-1}$ and the quiver
Temperley--Lieb algebras in characteristic $e$.   
\end{eg}

\begin{defn}
Let $\lambda \in {\mathbb E }_{h,\ell} $.     There are only finitely many hyperplanes   
lying {\em strictly} between  the   point  $\lambda  \in {\mathbb E}^\circledcirc_{h,\ell}$
and the origin $\circledcirc\in {\mathbb E}^\circledcirc_{h,\ell}$.  
For  $\alpha \in \Phi$, we let $\ell_\alpha (\la )$ denote the total number of these hyperplanes which are perpendicular to $\alpha\in \Phi$.
  We let $\ell (\la )= \sum_{\alpha \in \Phi}\ell_\alpha(\la)$.  
\end{defn}

\begin{rmk}
Note that we do not count any hyperplane upon which $\lambda$ actually lies. 
 \end{rmk}

\subsection{Paths in the geometry for $e\neq \infty$ }\label{1.2}
Let $e\in \mathbb{Z}_{>0}$.  We now introduce paths in our 
Euclidean space ${\mathbb E}_{h,\ell}$; the reader may find it helpful to consider the examples in \cref{sec:9}.  
We define a
degree function on such paths in terms of the hyperplanes in our
geometry.  We show how to identify these paths with semistandard tableaux and hence provide a graded path-theoretic basis of $A_h(n,\theta,\kappa)$.  

\begin{defn}  
 Given   a map  $s: \{1,\dots
  , n\}\to \{1,\dots , \ell h\}$ we define points ${\sf s}(k)\in
  {\mathbb E}_{h,\ell}$ by
\[
{\sf s}(k)=\sum_{1\leq i \leq k}\varepsilon_{s(i)} 
\]
for $1\leq i \leq n$. We define the associated path of length $n$ in
our alcove geometry ${\mathbb E}_{h,\ell}$ by $${\sf s}=({\sf
  s}(0),{\sf s}(1),{\sf s}(2), \ldots, {\sf s}(n)),$$  where we fix all
paths to begin at the origin, so that ${\sf s}(0)=\circledcirc \in
{\mathbb E}_{h,\ell}$.  We let ${\sf s}_{\leq \ik}$ denote the subpath
of ${\sf s}$ of length $\ik$ corresponding to the restriction of the
map $s$ to the domain $\{1,\dots , k\} \subseteq \{1,\dots , n\} $.
   \end{defn}

\begin{defn}\label{Soergeldegreee}
Given a path $\sts=(\sts(0),\sts(1),\sts(2), \ldots, \sts(n))$  we set
$\deg(\sts(0))=0$ and define  
 \[
 \deg(\sts ) = \sum_{1\leq k \leq n} d (\sts(k),\sts(k-1)), 
 \]
 where $d(\sts(k),\sts(k-1))$ is defined as follows. 
For $\alpha\in\Phi$ we set $d_{\alpha}(\sts(k),\sts(k-1))$ to be
\begin{itemize}
\item $+1$ if $\sts(k-1) \in 
   {\mathbb E}(\alpha,re)$ and 
   $\sts(k) \in 
   {\mathbb E}^{\less}(\alpha,re)$;
   
\item $-1$ if $\sts(k-1) \in 
   {\mathbb E}^{\great}(\alpha,re)$ and 
   $\sts(k) \in 
   {\mathbb E}(\alpha,re)$;
\item $0$ otherwise.  
   \end{itemize}
We let 
$$d (\sts(k-1),\sts(k))= \sum_{\alpha \in \Phi}d_\alpha(\sts(k-1),\sts(k)).$$  
\end{defn}

\begin{rmk}\label{assumponkappused}
 Let $\delta=((1^\ell),(1^\ell),\dots, (1^\ell)) \in \mptn \ell n(h)$.  
Importantly, there exists a degree zero path from the origin to $\delta$  
if and only if $e> h\ell$.  
In \cref{choiceofmulti},  we  chose   $\kappa \in I^\ell$ (by applying \cref{chooose})
 so that the path $ (+\varepsilon_1,+\varepsilon_2,\dots, +\varepsilon_{h\ell})$ from 
 the origin to $\delta$  is of degree zero.  
 We note that every point in this path is $e$-regular. 
 
 \end{rmk}

\begin{defn}
Let  $\mu \in \mptn \ell n (h)$  with 
component word of $\mu$  equal to 
\[
\varnothing=  \mu^{(0)}  \xrightarrow{+X_1}
\mu^{(1)}  \xrightarrow{+X_2}
\mu^{(2)}  \xrightarrow{+X_3} 
\cdots 
\xrightarrow{+X_{n-1}} 
\mu^{(n-1)}  \xrightarrow{+X_n}
\mu^{(n)}  = \mu.
\]
We fix a distinguished path $\stt^\mu$ from the origin to $\mu$  given by 
$$
\stt^\mu=(+\varepsilon_{X_1},+\varepsilon_{X_2}, \dots, +\varepsilon_{X_n}).
$$
Here we have abused notation slightly by identifying the addable box $X_k=(r_k,c_k,m_k)$ with the 
corresponding   $ \varepsilon_{X_k}=  \varepsilon_{h(m_k-1)+c_k}$.
\end{defn}

Let $\sts $ be a path which passes through a hyperplane ${\mathbb
  E}_{\alpha,re}$ at point $\sts(\ik)$ (note that $\ik$ is not
necessarily unique).  Then, let $\stt$ be the path obtained from
$\sts$ by applying the reflection $s_{\alpha,re}$ to all the steps in
$\sts $ after the point $\sts(\ik)$.  In other words,
$\stt(i)=\sts(i)$ for all $1\leq i \leq \ik$ and
$\stt(i)=s_{\alpha,re} \cdot \sts(i)$ for $\ik \leq i \leq n$.  We
refer to the path $\stt$ as the reflection of $\sts$ in ${\mathbb
  E}_{\alpha,re}$ at point $\sts(\ik)$ and denote this by
$s_{\alpha,re}^{\ik}\cdot \sts$.  We write $\sts \sim \stt$ if the
path $\stt$ can be obtained from $\sts$ by a series of such
reflections.

\begin{defn}
  We let $\Path_n(\lambda,\stt^\mu)$
  denote the set of all paths from the origin to $\lambda$ which may
  be obtained by applying repeated reflections to $\stt^\mu$, in other
  words
 $$\Path_n(\lambda,\stt^\mu) = \{\sts \mid \sts(n)=\lambda, \sts \sim
  \stt^{\mu}\}.$$ 
We let 
 $\Path_n^+(\lambda,\stt^\mu)
 \subseteq \Path_n(\lambda,\stt^\mu) 
 $ denote the  set of  paths  which at no point 
 leave the dominant Weyl chamber, in other words 
 $$\Path_n^+(\lambda,\stt^\mu)
 =
  \{\sts \in \Path_n(\lambda,\stt^\mu)
  \mid \sts(k) \in {\mathbb E}_{h,\ell}^\circledcirc
  \text{ for all }1\leq k \leq n
  \}.$$
      
\end{defn}

 \begin{defn}

Let  $\SSTT \in \SStd^+_n(\lambda,\mu)$ be a tableau with component reading word
$$
[{\sf T}_{\leq 0}] \xrightarrow{+Y_1}  [{\sf T}_{\leq 1}] \xrightarrow{+Y_2}  \cdots  \xrightarrow{+Y_n}   [{\sf T}_{\leq n}].
$$
We define a map $\omega:\SStd^+_n(\lambda,\mu) \to \Path_n^+(\lambda,\stt^\mu)$  where
$\omega(\SSTT)=\stt$ is the path in the alcove geometry given by 
$$
\stt=(+\varepsilon_{Y_1},+\varepsilon_{Y_2}, \dots, +\varepsilon_{Y_n}).
$$
\end{defn}

\begin{rmk}\label{asfjklaaksasgnkfafglkn}Given the unique  $\SSTT^\mu \in \SStd^+(\mu,\mu)$, it is clear that 
$\omega(\SSTT^\mu) = \stt^\mu $.  
 \end{rmk}

 \begin{lem}\label{useful-lemma}Given $\lambda\in \mptn \ell n (h)$, we have that     
\[
\langle \lambda +\rho ,   \varepsilon_{hm+i}  - \varepsilon_{hl+j}   \rangle = re
\]
for some $r\in\ZZ$, if and only if the  
the nodes $((\lambda^{(m+1)})^T_i,i,m), ((\lambda^{(l+1)})^T_j,j,l)  \in  \lambda $ have the same residue.
\end{lem}
\begin{proof} To see this, note that both statements are equivalent to
\[
 (\lambda^{(m+1)})^T_i+e-\kappa_i  \equiv  	(\lambda^{(l+1)})^T_j+e-\kappa_j  \pmod e. \qedhere
\]

\end{proof}

\begin{thm}\label{step3}
Let $e\in \mathbb{Z}_{>0}\cup\{\infty\}$ and $\kappa\in I^\ell$ be  $h$-admissible.  
 For $\lambda,\mu \in   \mptn \ell n (h)$ the map  $\omega: 
\SStd^+_n(\lambda,\mu) \to \Path^+_n(\lambda,\stt^\mu) 
$ 
is bijective and degree preserving.  
\end{thm}

\begin{proof}
The result clearly holds for $n=1$ and so we proceed by induction.    
We now let ${\mu}\in\mptn \ell {n+1}(h)$.  Let 
$X 
 $ denote the least dominant  element of ${\rm Rem}(\mu)$ and suppose this box has  residue $x \in I$.  By induction and \cref{789324578935247832592378375893978524978235497835241}, we may assume that 
\begin{align}\label{bij}
\SStd^+_{n+1}(\lambda,\mu) \leftrightarrow 
\bigsqcup_{Y \in \Rem_x(\lambda)}\SStd^+_{n }(\lambda-Y,\mu-X)
\leftrightarrow\bigsqcup_{
Y \in \Rem_x(\lambda)
}
\Path_{n }^+(\lambda-Y,\mu-X)
\end{align}
and that these bijections are degree preserving.
Given any $\SSTT \in 
\SStd^+_{n+1}( \lambda, \mu)$ we can let $Y\in {\rm Rem}(\lambda)$  denote the box containing the entry $n+1$.  
By induction, the pair 
$$\SSTT_{\leq n} \in 
\SStd^+_{n}( \lambda-Y, \mu-X)\quad \text{ and } \quad 
\omega(\SSTT_{\leq n}) \in 
\Path_{n}^+( \lambda-Y, \mu-X)$$ 
are identified under the map \ref{bij} and  the degrees coincide.  
Moreover  given any $\SSTT \in \SStd^+_{n+1}(\lambda,\mu)$, we can write  
$$
\omega(\SSTT_{\leq n } )
=
 s_{\varepsilon_{i_t}-{\varepsilon_{j_t}},m_te}^{(k_t)}
\dots 
s_{\varepsilon_{i_1}-{\varepsilon_{j_1}},m_1e}^{(k_1)} 
\cdot \omega(\SSTT^{\mu-X})
$$
for some $k_1\leq k_2 \leq \dots \leq k_t$.  
  Now, given $\stt\in \Path_n^+(\lambda,\stt^\mu)$ such that 
   $\stt(n) \in \mathbb{E}_{i_n-j_n,m_ne}$ and   $\stt(n+1) \not \in \mathbb{E}_{i_n-j_n,m_ne}$ 
we have that   \begin{align*}
\stt =&s_{\varepsilon_{i_n}- \varepsilon_{j_n},m_ne}^{(n)}
( s_{\varepsilon_{i_t}-{\varepsilon_{j_t}},m_te}^{(k_t)}
\dots 
s_{\varepsilon_{i_1}-{\varepsilon_{j_1}},m_1e}^{(k_1)})
\cdot \stt^\mu 
\\  =
  &s_{\varepsilon_{i_n}- \varepsilon_{j_n},m_ne}^{(n)}
( s_{\varepsilon_{i_t}-{\varepsilon_{j_t}},m_te}^{(k_t)}
\dots 
s_{\varepsilon_{i_1}-{\varepsilon_{j_1}},m_1e}^{(k_1)})
\cdot (\stt^\mu{\downarrow}_{\leq n}\circ (+\varepsilon_X))
\\  =
  &s_{\varepsilon_{i_n}- \varepsilon_{j_n},m_ne}^{(n)}
( s_{\varepsilon_{i_t}-{\varepsilon_{j_t}},m_te}^{(k_t)}
\dots 
s_{\varepsilon_{i_1}-{\varepsilon_{j_1}},m_1e}^{(k_1)})
\cdot \omega(\SSTT^{\mu}_{\leq n} + \{X\})
\\  =
  &s_{\varepsilon_{i_n}- \varepsilon_{j_n},m_{p,q}e}^{(n)}
 \cdot  \omega(\SSTT_{\leq n}  
+\{\sigma(X)\}
 )
\\  =
  &    \omega(\SSTT_{\leq n}  +\{\sigma'(X)\})
 \end{align*} 
 where $\sigma=
 s_{\varepsilon_{i_t}-{\varepsilon_{j_t}}} 
\dots 
s_{\varepsilon_{i_1}-{\varepsilon_{j_1}} }\in \mathfrak{S}_{h\ell} $, 
$\sigma'=
 s_{\varepsilon_{i_n}- \varepsilon_{j_n}}s_{\varepsilon_{i_t}-{\varepsilon_{j_t}}} 
\dots 
s_{\varepsilon_{i_1}-{\varepsilon_{j_1}} }\in \mathfrak{S}_{h\ell} $  
and $\sigma'(X)=   Y \in  {\rm Rem}_x(\lambda)$.   This gives us the required   bijection
\begin{align}
 \Path_{n }^+(\lambda ,\mu )
\leftrightarrow  
\bigsqcup_{
Y \in \Rem_x(\lambda)
}
\Path_{n }^+(\lambda-Y,\mu-X). 
 \end{align}
  It remains to verify that the bijection is degree preserving.   
Let  $A_{(r,c,m)} \rhd A_{(r',c',m')}$ be two addable nodes of  some $\lambda\in \mptn \ell n$ and suppose that $r-c\neq r'-c'$. 
This implies that 
the $c'$th column of the $m'$th component of $\lambda$ is {\em strictly} greater than
the $c $th column of the $m $th component of $\lambda$.  
 Therefore   
\begin{align}\label{pm1}\langle\lambda+A_{(r',c',m')} , \varepsilon_{c'} - \varepsilon_{c}\rangle=\langle \lambda , \varepsilon_{c'} - \varepsilon_{c}\rangle+1
\quad
\langle\lambda+A_{(r,c,m)} , \varepsilon_{c'} - \varepsilon_{c}\rangle=\langle \lambda , \varepsilon_{c'} - \varepsilon_{c}\rangle-1
\end{align}are both are strictly positive.  
For any   $Y 
 \in {\rm Rem}_{x}({\lambda})$, we let 
\begin{align*}{\rm Add}_{x}(\lambda-Y) =\{A_{(r_1,c_1,m_1)}\rhd A_{(r_2,c_2,m_2)} \rhd \dots \rhd  A_{(r_a,c_a,m_a)}\}	  ={\rm Add}_{x}(\lambda-Y)  
\setminus \{Y\}
\\{\rm Rem}_{x}(\lambda ) =\{R_{(p_1,q_1,t_1)}\rhd R_{(p_2,q_2,t_2)} \rhd \dots \rhd  R_{(p_b,q_b,t_b)}\}	 = {\rm Rem}_{x}( {\lambda}-Y)  
\setminus \{Y\}.
\end{align*}
We note that ${\rm Rem}_{z}( {\lambda}-Y)  ={\rm Rem}_{z}( {\lambda})  $ and ${\rm Add}_{z}( {\lambda}-Y)  ={\rm Rem}_{z}( {\lambda})  $ for all $x  \neq z \in I$.  
We further note that $r_i-c_i \neq r_j-c_j$ (respectively 
$p_i-q_i \neq p_j-q_j$) for  $1\leq i < j \leq a$ (respectively $1\leq i < j \leq b$) by  
\cref{seperating}.
We set $Y=A_{(r_y,c_y,m_y)}=R_{(p_{y'},q_{y'},t_{y'})}$ for some $1\leq y\leq a$ and $1\leq y'\leq b$.

If $\lambda \in \mathbb{E}_{\alpha,me}$  and $\lambda+\varepsilon_Y \not \in \mathbb{E}_{\alpha,me}$ for some $m\in \mathbb{Z}_{>0}$, then $\alpha= \varepsilon_{c_i}-\varepsilon_{c_j}$ 
   for some
 $1\leq i < j \leq a$.   Similarly,  if  $\lambda \not \in \mathbb{E}_{\alpha,me}$  and $\lambda+\varepsilon_Y   \in \mathbb{E}_{\alpha,me}$ for some $m\in \mathbb{Z}_{>0}$, then  $\alpha= \varepsilon_{q_i}-\varepsilon_{q_j}$ 
   for some
 $1\leq i < j \leq b$.   
 By \cref{pm1}, we have that 
 $$d_{\varepsilon_{c_i}-\varepsilon_{c_j}}(\lambda, \lambda+\varepsilon_Y)
 =
 \begin{cases}
 1 &\text{ if }  i = y < j 
 \\
 0&\text{ otherwise}
 \end{cases}
 \qquad
 d_{\varepsilon_{q_i}-\varepsilon_{q_j}}(\lambda, \lambda+\varepsilon_Y)
 =
 \begin{cases}
 -1 &\text{ if }i = y' < j 
 \\
 0&\text{ otherwise}
 \end{cases}
 $$
 and summing over all these terms we obtain 
 $$
 d (\lambda, \lambda+\varepsilon_Y)
 = (a-y)-(b-y')
$$
which is equal to the number of addable $x$-boxes to the right of $Y$ minus the number of 
removable  $x$-boxes to the right of $Y$, as required.   
   \end{proof}

\begin{thm}\label{quotienthghghg}
The $R$-algebra $A_h(n,\theta,\kappa)$  is a graded cellular algebra with  basis 
\[
 \{C_{\sts\stt}   \mid \sts \in \Path^+_n(\lambda,\stt^\mu), 
\stt \in \Path^+_n(\lambda,\stt^\nu), 
\lambda,\mu, \nu \in \mptn {\ell}n(h)\}
\]
with respect to the $\theta$-dominance order on $\mptn \ell n(h)$ and the involution $\ast$.  
Here $C_{\sts\stt}:=C_{\SSTS\SSTT}$ for $\omega(\SSTS)=\sts$ and  $\omega(\SSTT)=\stt$.  
\end{thm}

\subsection{The algebras $A(n,\theta,\kappa)$ for $e=\infty$ or $e>n$}

We now take a short detour to consider the algebras  $A(n,\theta,\kappa)$  with $e= \infty$ (or more generally $e>n$) and    
$\kappa\in I^\ell$ is  1-admissible (i.e a multicharge with no repeated entries).
 We can take the condition on the $e$-multicharge to be less restrictive because of our assumption that $e>n$.  
With minor technical modifications to the combinatorics, this can be treated in exactly the same way as the algebras
$A_h(n,\theta,\kappa)$ for $e>h\ell$ and $\kappa\in I^\ell$ $h$-admissible.  
 Let $\stt^\mu$ be the path from $\circledcirc$ to $\mu$, defined as above.  
We have two problems to address:
\begin{itemize}
\item[$(i)$] The  path $\stt^\mu$ is not necessarily of degree zero, however $\deg(\SSTT^\mu)=\deg(C_{\SSTT^\mu})=0$;
\item[$(ii)$] The map $\varphi:\SStd^+_n(\lambda,\mu)\to \Path_n^+(\lambda,\mu)$  is not necessarily surjective  and 
$\deg(\varphi(\SSTS))= \deg(\SSTS) - \deg(\stt^\mu)$ for $\SSTS\in \SStd^+_n(\lambda,\mu)$.  
\end{itemize}
We take care of these problems as follows. 
Let $$\Sigma(\mu)=\{k \mid   1\leq k \leq n, d_\alpha(\stt^\mu(k-1),\stt^\mu(k))\neq 0 \text{ for some }\alpha\in \Phi\}.$$   
 We let $\Path^\infty _n(\lambda,\mu)\subseteq  \Path^+_n(\lambda,\mu)$ denote the subset of paths $\sts$ such that 
 $$
 \sts \sim \sts^{(1)} \sim \sts^{(2)}\sim \dots \sim \sts^{(t)}=\stt^\mu
 $$
 where $\sts^{(i)}=s_{\alpha,re}^{k_i}\cdot \sts^{(i-1)}$ for some $k_i \not \in \Sigma(\mu)$.
 We then define the degree   as follows 
  \[
 \deg(\sts ) = \sum_{\begin{subarray}
c
1 \leq k \leq n \\
k \not \in \Sigma(\mu)
\end{subarray} } d (\sts(k),\sts(k-1)), 
 \]
 for   $\sts \in \Path^\infty _n(\lambda,\mu)$.  
   In the remainder of the paper, we shall only deal explicitly with the algebras $A_h(n,\theta,\kappa)$ with $e>h\ell$. 
   However all the results   can be easily generalised to   $A (n,\theta,\kappa)$ for 
 $e>n$ and  $\kappa\in I^\ell$   1-admissible.    
 
 \begin{thm}\label{eisinfinity}
 Let $e>n$ and  $\kappa\in I^\ell$ be   1-admissible. 
 The $R$-algebra $A (n,\theta,\kappa)$    is a graded cellular algebra with  basis 
\[
 \{C_{\sts\stt} \mid \sts \in \Path^\infty_n(\lambda,\stt^\mu), 
\stt \in \Path^\infty_n(\lambda,\stt^\nu), 
\lambda,\mu, \nu \in \mptn {\ell}n(\infty) \}
\]
with respect to the $\theta$-dominance order on $\mptn \ell n $ and the involution $\ast$.  
Here $C_{\sts\stt}:=C_{\SSTS\SSTT}$ for $\omega(\SSTS)=\sts$ and  $\omega(\SSTT)=\stt$.  
 \end{thm}
 \begin{rmk}\label{see}
We remark that the  case $e=\infty$ (or more  generally $e>n$) is expected to be far simpler  than the  case  $e\leq n$.   
Indeed, this prompted an optimistic conjecture of Kleshchev--Ram \cite[Conjecture 7.3]{MR2777040} (later proven false in   \cite[Section 4.2]{MR3163410}).  
From our point of view, this simplification is a consequence of the fact that the affine Weyl group ``controlling" the alcove geometry for $e=\infty$ is finite (whereas it is infinite  for $e<n$).  This is reflected in the fact that 
  \cref{eisinfinity} deals with the entire diagrammatic Cherednik algebra, rather than the quotient considered in \cref{quotienthghghg}.

 The KLR algebras $H_n(\kappa)$ for $e=\infty$ and $\ell=2$ 
have received a great deal of attention from Brundan--Stroppel    (see for example \cite{MR2600694,MR2781018}) 
and  some of their results were extended to the case $e>n$ by Mathas--Hu 
 \cite[Appendix B]{MR3356809}.   
\end{rmk}

%
%

 \begin{rmk}
Let $e>n$. Up to a trivial re-ordering of the weighting, any two diagrammatic  Cherednik algebras 
for distinct weightings (with $\kappa\in I^\ell$ fixed) are isomorphic as graded $R$-algebras.  
(This is certainly not true for $e<n$ and can be seen as another way in which the overall picture simplifies for $e=\infty$.)  
  Therefore,  the graded decomposition matrix  of $H_n(\kappa)$ does not depend on our choice of weighting 
by \cite[Corollary 5.3]{manycell}.  Therefore we can speak of calculating the graded decomposition matrix  of $H_n(\kappa)$ (as we shall in \cref{conj2}) without reference to our chosen weighting.    
\end{rmk}

\section{Inductively constructing basis elements from the path}\label{howtobuildshit}
We now pause in order to highlight how one can inductively construct a basis element
$C_\SSTS$ of the algebra directly from the corresponding path $\omega(\SSTS)=\sts \in\Path^+(\lambda,\stt^\mu)$.  
Let 
$$\stt^\mu=(+\varepsilon_{X_1},+\varepsilon_{X_2},\dots , +\varepsilon_{X_n}) $$
and let 
$$\sts =(+\varepsilon_{Y_1},+\varepsilon_{Y_2},\dots , +\varepsilon_{Y_n}). $$
Given $1\leq k\leq n$, 
we obtain $C_{\SSTT_{\leq k}}$ from $C_{\SSTT_{\leq k-1}}$
by adding a strand 
connecting the northern point at the end of the $i_k$th column  
to the southern point at the end of the $j_k$th column.

\begin{figure}[ht!] 
$$   \scalefont{0.7} \begin{tikzpicture}[scale=.65] 
\draw (-3,-1) rectangle (13,2);   
     \draw[wei2]   (0.1,-1)--(0.1,2) ;  
     \draw[wei2]   (0.5,-1)--(0.5,2) ;  
             \draw[wei2]   (0.9,-1)--(0.9,2) ;  

\draw     (0.2,-1)  to     (0.2,2) ;  
\draw[densely dashed]     (0.2-3+0.1,-1)   to   (0.2-3+0.1,2) ;  
\node [below] at (0.2,-1)  {  $0$};     
            \end{tikzpicture}
$$ 
$$   \scalefont{0.7} \begin{tikzpicture}[scale=.65] 
\draw (-3,-1) rectangle (13,2);   
     \draw[wei2]   (0.1,-1)--(0.1,2) ;  
     \draw[wei2]   (0.5,-1)--(0.5,2) ;  
             \draw[wei2]   (0.9,-1)--(0.9,2) ;  
\draw     (0.2,-1)  to     (0.2,2) ;  
\draw[densely dashed]     (0.2-3+0.1,-1)   to   (0.2-3+0.1,2) ;  
\draw     (3.2,-1)      to  (3.2,2) ;  
\draw[densely dashed]     (3.2-3+0.1,-1)   to   (3.2-3+0.1,2) ;  
\node [below] at (0.2,-1)  {  $0$};     
\node [below] at (3.2,-1)  {  $5$};     

           \end{tikzpicture}
$$ 
$$   \scalefont{0.7}\begin{tikzpicture}[scale=.65] 
\draw (-3,-1) rectangle (13,2);   
     \draw[wei2]   (0.1,-1)--(0.1,2) ;  
     \draw[wei2]   (0.5,-1)--(0.5,2) ;  
             \draw[wei2]   (0.9,-1)--(0.9,2) ;  
\draw     (0.2,-1)  to     (0.2,2) ;  
\draw[densely dashed]     (0.2-3+0.1,-1)   to   (0.2-3+0.1,2) ;  
\draw     (3.2,-1)      to  (3.2,2) ;  
\draw[densely dashed]     (3.2-3+0.1,-1)   to   (3.2-3+0.1,2) ;  
\draw     (1,-1)    to[in=-170,out=50]         (6.2,2) ;  
\draw[densely dashed]      (1-3+0.1,-1)    to[in=-160,out=50]   (6.2-3+0.2,2) ;

\node [below] at (0.2,-1)  {  $0$};     
 \node [below] at (1,-1)  {  $4$};     
\node [below] at (3.2,-1)  {  $5$};     
 
           \end{tikzpicture}
$$ 
$$   \scalefont{0.7} 
\begin{tikzpicture}[scale=.65] 
\draw (-3,-1) rectangle (13,2);   
     \draw[wei2]   (0.1,-1)--(0.1,2) ;  
     \draw[wei2]   (0.5,-1)--(0.5,2) ;  
             \draw[wei2]   (0.9,-1)--(0.9,2) ;  
\draw     (0.2,-1)  to     (0.2,2) ;  
\draw[densely dashed]     (0.2-3+0.1,-1)   to   (0.2-3+0.1,2) ;  
\draw     (3.2,-1)      to  (3.2,2) ;  
\draw[densely dashed]     (3.2-3+0.1,-1)   to   (3.2-3+0.1,2) ;  
\draw     (1,-1)    to[in=-170,out=50]         (6.2,2) ;  
\draw[densely dashed]      (1-3+0.1,-1)    to[in=-160,out=50]   (6.2-3+0.2,2) ;  
\draw     (4,-1)    to[in=-170,out=50]   (9.2,2) ;  
\draw[densely dashed]      (4-3+0.1,-1)    to[in=-160,out=50]     (9.2-3+0.3,2) ;  

\node [below] at (0.2,-1)  {  $0$};     
 \node [below] at (1,-1)  {  $4$};     
\node [below] at (3.2,-1)  {  $5$};     
\node [below] at (4,-1)  {  $3$};     

           \end{tikzpicture}
$$ 
$$   \scalefont{0.7} \begin{tikzpicture}[scale=.65] 
\draw (-3,-1) rectangle (13,2);   
     \draw[wei2]   (0.1,-1)--(0.1,2) ;  
     \draw[wei2]   (0.5,-1)--(0.5,2) ;  
             \draw[wei2]   (0.9,-1)--(0.9,2) ;  
\draw     (0.2,-1)  to     (0.2,2) ;  
\draw[densely dashed]     (0.2-3+0.1,-1)   to   (0.2-3+0.1,2) ;  
\draw     (3.2,-1)      to  (3.2,2) ;  
\draw[densely dashed]     (3.2-3+0.1,-1)   to   (3.2-3+0.1,2) ;  
\draw     (1,-1)    to[in=-170,out=50]         (6.2,2) ;  
\draw[densely dashed]      (1-3+0.1,-1)    to[in=-160,out=50]   (6.2-3+0.2,2) ;  
\draw     (4,-1)    to[in=-170,out=50]   (9.2,2) ;  
\draw[densely dashed]      (4-3+0.1,-1)    to[in=-160,out=50]     (9.2-3+0.3,2) ;  
\draw     (.6,-1)   to[in=-150,out=15]   (12.2,2) ;  
\draw[densely dashed]      (.6-3+0.1,-1)   to[in=-150,out=15]     (12.2-3+0.2,2) ;

\node [below] at (0.2,-1)  {  $0$};     
\node [below] at (0.6,-1)  {  $2$};     
\node [below] at (1,-1)  {  $4$};     
\node [below] at (3.2,-1)  {  $5$};     
\node [below] at (4,-1)  {  $3$};     

           \end{tikzpicture}
$$ 
\caption{The elements $C_{\sts{\downarrow}_{\leq k}}$ for $k=1,\dots 5$}
\label{sequence}
\end{figure}
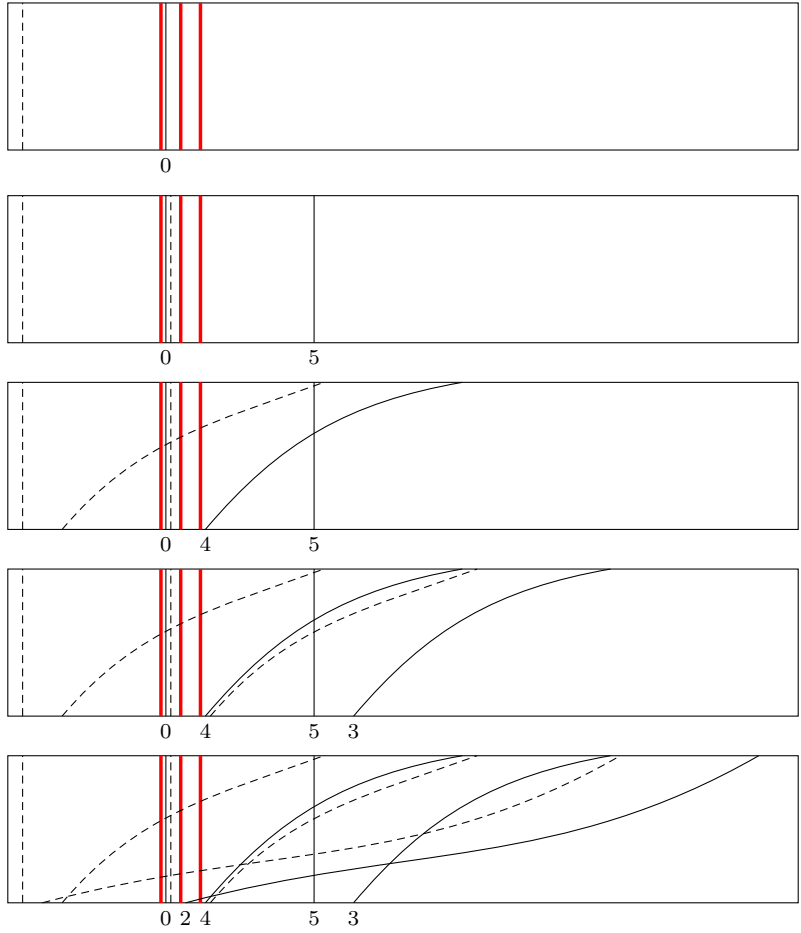

\begin{eg}
We continue with \cref{egeg}.  
Let $e=7$, $h=1$, $\ell=3$, and $\kappa=(0,2,4)\in (\ZZ/7\ZZ)^3$.   
For $\mu=((1^5),\varnothing,\varnothing)$, we have that 
$$
\stt^\mu=(+\varepsilon_1,+\varepsilon_1,+\varepsilon_1,+\varepsilon_1,+\varepsilon_1)
$$
for $\lambda=((1^2),(1),(1^2))$ we consider the path
$$
s_{\varepsilon_2-\varepsilon_3,-e}^{(4)}  s_{\varepsilon_1-\varepsilon_3,e}^{(2)} \cdot \stt^\mu= \sts=(+\varepsilon_1,+\varepsilon_1,+\varepsilon_3,+\varepsilon_3,+\varepsilon_2)
\in \Path^+_5(\lambda,\stt^\mu).$$
For each of the 5 steps in $\sts$, the corresponding  basis elements (corresponding to
$\sts{\downarrow}_{\leq k}$ and   $1\leq k \leq 5$) are depicted in \cref{sequence} below.  
For each $1\leq k\leq 5$, the  northern (respectively southern) residue sequence is given 
by $\Shape(\stt^\mu {\downarrow}_{\leq k})$ 
(respectively $\Shape(\sts{\downarrow}_{\leq k})$).  
For example, if $k=3$ then $\sts{\downarrow}_{\leq 3}\in \Path^+(((1^2), \varnothing,(1)),((1^3),\varnothing,\varnothing))$ has northern loading given by $((1^3),\varnothing,\varnothing)$ and southern loading given by $((1^2), \varnothing,(1))$.

\end{eg}

\section{Tensoring with the determinant}\label{determinant} We now   identify our higher-level analogue 
 of the stability obtained by ``tensoring with the determinant'' for general linear groups.  Notice that we are working in the Ringel dual setting and so ``tensoring with the determinant" means ``adding a row" as opposed to ``adding a column".  This has an obvious higher level generalisation, as we shall now see.
    For the remainder of the paper  we shall have to actually multiply diagrams together in order to prove various isomorphisms.  
   We therefore fix some notation regarding the manipulation of diagrams using the relations of \cref{defintino2}.

 \begin{rmk}\label{noninteracting}
We shall refer to the relations \ref{rel1}, \ref{rel2}, \ref{rel5}, \ref{rel8}, \ref{rel13} and \ref{rel14} and the latter relation in both \ref{rel4} and \ref{rel11} as {\sf non-interacting relations}.
These are   the relations given by pulling strands through one another in the na\"ive fashion (without acquiring     error terms or dots or sending the diagram to zero).  We refer to a  {\sf critical point} as any local neighbourhood in the diagram with   non-zero degree.   
When manipulating diagrams, we  focus on the ``critical points"
at    which we cannot use the non-interacting relations (as this is where things get tricky).
 Upon reaching a critical point (by manipulating the diagrams as much as possible using the non-interacting relations) 
 we  resolve  this critical point (if necessary) in order to obtain a linear combination of diagrams.  
\end{rmk}

\begin{thm}\label{dettensor}
Given a partition $\lambda=(\lambda_1 ,\lambda _2,\dots)$, we set 
  ${\rm det}_h(\lambda) =  	(h, \lambda_1 ,\lambda _2,\dots) .$  We have 
 an injective map  of partially ordered sets 
${\rm det}_h:\mptn \ell n (h) \hookrightarrow \mptn \ell {n+h\ell} (h)  $ given by 
 $${\rm det}_h(\lambda^{(1)},\lambda^{(2)},\dots ,\lambda^{(\ell)}) =  
 	({\rm det}_h(\lambda^{(1)}),{\rm det}_h(\lambda^{(2)}),\dots,{\rm det}_h(\lambda^{(\ell)})) .$$
The image, ${\rm det}_h(\mptn \ell n (h))$, is a closed subset of $
\mptn \ell {n+h\ell} (h)
 $ under the $\theta$-dominance ordering and we have a degree-preserving  bijective map
$${\rm det}_h:\Path_n(\lambda,\stt^\mu)
\to \Path_{n+h\ell}({\rm det}_h(\lambda),\stt^{{\rm det}_h(\mu)})
$$
given by  
$$\det(\sts)=(+\varepsilon_1,+\varepsilon_2, \dots, +\varepsilon_{ h\ell})  \circ \sts  \in   \Path_{n+h\ell}({\rm det}_h(\lambda),\stt^{{\rm det}_h(\mu)})$$   
for $\sts\in \Path_n(\lambda,\stt^\mu)$. We have an isomorphism of graded $R$-algebras
$$
 A_{{\rm det}_h(\mptn \ell {n+h\ell} (h))}(n+h\ell,\theta,\kappa)
 \cong 
 A_{\mptn \ell {n } (h)}(n,\theta,\kappa)	 $$
 In particular, over an arbitrary field $\Bbbk$ we have that 
$$d_{\lambda,\mu}(t)=d_{{\rm det}_h(\lambda),{\rm det}_h(\mu) }(t)$$
for all $\lambda,\mu\in \mptn \ell {n } (h)$.  
\end{thm}

\begin{proof}
The  map ${\rm det}_h$ is easily checked to be injective and its image a closed subset under the $\theta$-dominance ordering.   
Our assumption on $\kappa \in I^\ell$ (see \cref{assumponkappused}) implies that $\deg(\sts)=\deg(\det(\sts))$.  
As remarked in  \cref{assumponkappused}, the path 
$(+\varepsilon_1,+\varepsilon_2, \dots, +\varepsilon_{ h\ell})$  
does not pass through any  $W^e$-hyperplane   and so we obtain the required bijection.
Therefore
$$
 A_{{\rm det}_h(\mptn \ell {n+h\ell} (h))}(n+h\ell,\theta,\kappa)
 \cong 
 A_{\mptn \ell {n } (h)}(n,\theta,\kappa)	 $$
on the level of graded $R$-modules.  It remains to prove that the isomorphism holds on the level of $R$-algebras.  
Proceeding as in \cref{howtobuildshit}, we see that the diagram $C_{\det(\sts)\det(\stt)}$ 
is obtained from that of $C_{\sts\stt}$ by 
\begin{itemize}[leftmargin=*]
\item shifting any solid or ghost strand $X$ (note that we are excluding the case that $X$ is  a vertical red strand) rightwards by 
 $(\ell+\varepsilon)$-units 
  (we now refer to this strand as $\det(X)$) 
  \item and adding $h\ell$ `new' vertical solid strands (with their accompanying ghosts)
 with $x$-coordinates given by $\mathbf{I}_{(1,c,m)}$ for $1\leq c \leq h$ and $1\leq m \leq \ell$.  
 \end{itemize}
Let $X$ be a strand of residue $i \in I$   in $C_{ \sts \stt }$.  
Let $A$ denote any of the $h\ell$ distinct strands in 
$C_{\det(\sts)\det(\stt)}$
 which do  not appear in $C_{\sts \stt }$.  
There is no crossing of a solid and red strand of the same residue in either    $C_{\sts\stt}$ 
or $C_{\det(\sts)\det(\stt)}$, by \cref{seperating}.  
 Again by  \cref{seperating},
any crossing of $\det(X)$ with a {\em new} vertical strand, $Y$,  in  $C_{\det(\sts)\det(\stt)}$ is of degree zero and can be removed using only the non-interacting relations, i.e. the relations which do not annihilate the diagram (\ref{rel4}), change the number of dots on a strand, or which create  error terms.   
Any crossing of strands $\det(X)$  and  $\det(Y)$ in   $C_{\det(\sts)\det(\stt)}$  can be removed in exactly the same fashion as $X$ and $Y$ in 
$C_{\det(\sts)\det(\stt)}$.  
  The  $R$-algebra isomorphism follows.  
\end{proof}

\section{The super-strong linkage principle}\label{sec:6}
Throughout this section,   $\Bbbk$ is an arbitrary field.  
Let $\lambda,\mu \in {\mathbb E}_{h,\ell}^\circledcirc$.  We say that $\lambda$ and $\mu$ are $W^e$-{\sf linked} if they belong to the same orbit under the dot action of the affine Weyl group, that is if $\lambda \in W^e\cdot \mu$. 
 Given two polynomials $f,g \in \mathbb{N}[t,t^{-1}]$, we write $f \leq g$ if and only if $f-g \in  \mathbb{N}[t,t^{-1}]$.

\begin{thm} [The super-strong linkage principle]\label{strongerman}  
 We have that 
\begin{equation}\label{412321434321442314213}
d_{\lambda\mu}(t) \leq \textstyle \sum_{\sts \in \Path^+(\lambda,\stt^\mu)}t^{\deg(\sts)} 
\end{equation}
as degree-wise polynomials, in other words for every  $k \in \ZZ$ we have that 
$$
[\Delta(\lambda):L(\mu)\langle k \rangle]
\leq 
|\{\sts \mid \sts \in \Path^+(\lambda,\stt^\mu) ,  \deg(\sts) = k\}|
$$
for $\lambda,\mu \in \mptn \ell n (h)$.  
In particular, if   $d_{\lambda,\mu}(t)\neq 0$ then $\Path^+(\lambda,\stt^\mu)\neq \emptyset$.
 \end{thm}

\begin{proof}
By definition, we have that 
\[ 
d_{\lambda\mu}(t)=\textstyle\sum_k\dim_\Bbbk(\Hom_{A_h(n,\theta,\kappa)}(P(\mu),\Delta(\lambda)\langle k\rangle))t^k .\]
Now, the projective module $P(\mu)$  is a direct summand of $A_h(n,\theta,\kappa) {\sf 1}_\mu$  and so ${\sf 1}_\mu$ acts trivially on the image of any such homomorphism above.  
For any homomorphism $\varphi\in \Hom_{A_h(n,\theta,\kappa)}(P(\mu),\Delta(\lambda)\langle k\rangle)$,  $\varphi({\sf 1}_\mu)\in {\sf1}_\mu \Delta(\lambda)$.  Moreover, $P(\mu)$ is cyclic and so $\varphi$   is determined by $\varphi({\sf 1}_\mu)$.  
Therefore 
\[
\dim_\Bbbk(\Hom_{A_h(n,\theta,\kappa)}(P(\mu),\Delta(\lambda)\langle k\rangle))
\leq 
|\{ C_\SSTS \mid C_\SSTS \in  {\sf1}_\mu \Delta(\lambda), \deg( C_\SSTS)=k\}| 
 \]
and so the result follows.  \end{proof}

We now show how our super-strong linkage principle is a (considerable) strengthening of the usual ``strong linkage principle'' of \cite{MR1670762}.  
 

  \begin{defn}
Let $\lambda,\mu\in {\mathbb E}_{h,\ell}^\circledcirc$   be such that    
 $\lambda = s_{\alpha,me} \cdot \mu$ for some $\alpha= \varepsilon_{i }-\varepsilon_{j }\in \Phi$, $m\in \ZZ$.  
 We write   $    \lambda   \uparrow_{\alpha,me} \mu$ if  $\lambda \in  {\mathbb E}^{\less }(\alpha,me)$, $\mu \in {\mathbb E} ^{\great }(\alpha,me)$.  
We   write $\lambda \uparrow \mu$ if there exists a sequence
 $$
\lambda 
 = \lambda^{(0)} \uparrow_{\varepsilon_{i_1}-\varepsilon_{j_1},m_1e} 
 \lambda^{(1)} \uparrow  _{\varepsilon_{i_2}-\varepsilon_{j_2},m_2e} 
 \lambda^{(2)} \uparrow  _{\varepsilon_{i_3}-\varepsilon_{j_3},m_3e} 
 \dots	 \uparrow  
 \lambda^{(k-1)} \uparrow  _{\varepsilon_{i_k}-\varepsilon_{j_k},m_ke} 
 \lambda^{(k)} =  \mu $$
for some $k\geq 0$.   We 
say that $\lambda$ and $\mu$ are {\sf strongly linked} if $\lambda \uparrow \mu$ or 
$\mu\uparrow \lambda$.   
\end{defn}

  \begin{rmk}
Note that $\circledcirc$ is the most dominant 
point in  this ordering and that this is the opposite convention to that used in conventional Lie theory.  
  \end{rmk}

\begin{thm}  \label{strongman}
If $\Path_n(\lambda,\stt^\mu) \neq \emptyset$, then  $\la$ and $\mu$ are strongly linked with 
$\lambda \uparrow  \mu$.  
\end{thm}

\begin{proof}
Given  $\sts \in \Path_n(\lambda,\stt^\mu)$ for $\lambda \neq \mu$, we let $1\leq k\leq n$ denote the first integer such that  $\sts(k)\neq \stt^\mu(k)$.  By assumption,   $\sts(k-1)=\stt^\mu(k-1) \in {\mathbb E} {(\alpha,me)}$ and $\sts(k)\not  \in {\mathbb E} {(\alpha,me)}$ for some $\alpha \in \Phi$, $m\in \ZZ$.  
 We let $k<k'\leq n$ denote the minimal integer such that $\sts(k')\in {\mathbb E} {(\alpha,me)}$ if such an integer exists, and be undefined otherwise.  

  By the minimality of both  $k$ and $k'$ and the definition of $\stt^\mu$, we   deduce that 
$\sts(j)  \in {\mathbb E}^{\great } {(\alpha,me)}$  for all $k\leq j\leq k'$ if $k'$ is defined and  
 for all $k\leq j \leq n$ otherwise.  
We let         $\lambda= \lambda^{(1)} $  if 
 $k'$ is defined and set 
$  \lambda  \uparrow   \lambda^{(1)} =s_{\alpha,me}\cdot \lambda\in \mathbb{E}_{h, \ell} $ otherwise.   
We let $\sts^{(1)}\in   \Path_n(  \lambda^{(1)},\stt^\mu)$ denote the path 
 \begin{align*}
\sts^{(1)}  =
\begin{cases} s_ {(\alpha,me)}^{ k'}s_ {(\alpha,me)}^{ k}\cdot \sts 
 &\text{if $k'$ is defined} \\
 s_ {(\alpha,me)}^{ k}\cdot \sts &\text{otherwise}.
   \end{cases}
    \end{align*}  
 Repeat this procedure with the path $\sts^{(1)}$ to obtain a path  $\sts^{(2)}$.  Continuing in this fashion we obtain an ordered sequence of multipartitions 
 $$
\lambda = \lambda^{(0)}   \uparrow
 \lambda^{(1)}   \uparrow 
 \lambda^{(2)}   \uparrow 
 \dots	   \uparrow 
 \lambda^{(k-1)}   \uparrow 
 \lambda^{(k)} =\mu.
 $$
(given by the terminating points  of the corresponding paths) as required.  
\end{proof}

\begin{cor} [Strong linkage principle]\label{Strong linkage principle}
 If   $d_{\lambda,\mu}(t)\neq 0$ for $\lambda,\mu \in \mptn \ell n(h)$,  then $\lambda\uparrow  \mu$.  
  \end{cor}

\begin{proof}
 The statement of the  result follows by \cref{strongman,strongerman} as $\Path^+(\lambda,\stt^\mu)\subseteq 
\Path (\lambda,\stt^\mu)$.
\end{proof}

\begin{rmk}
For $\ell=1$ and $ p=e$ the algebra $\Q_{1,h,n}(\kappa)$ is isomorphic 
to the image of the symmetric group on $n$ letters   in ${\rm End}_{\Bbbk}((\Bbbk^h)^{\otimes n})$.  Therefore the decomposition numbers $d_{\lambda\mu}(t)$ are the (graded) decomposition numbers of symmetric groups and  
 \cref{Strong linkage principle} is equivalent to 
 the strong linkage principle for general linear groups for $p>h$ 
 (as Ringel duality preserves the quasi-hereditary ordering). 
\end{rmk}

We now provide an example which illustrates how (even in level $\ell=1$) our super-strong linkage principle is a significant strengthening of the usual strong linkage principle.

 \begin{eg}\label{egstrong}
Let $\Bbbk$ be an arbitrary field.  Let   $h=3$ and $\ell=1$ and  let $e=6$.  
For $\mu=(2,1^{12})$,  
 there are 6 elements of the set 
$$\Pi=\{\lambda \in \mptn 1 {12}(3) \mid \lambda \uparrow \mu\}=\{\lambda \in \mathbb{E}_{3,1}^\circledcirc 
\mid  \Path(\lambda,\stt^\mu)\neq \emptyset\}.$$
There are a total of 8 paths in the set  $\{\sts \mid \sts \in   \Path(\lambda,\stt^\mu), \lambda \in \Pi\}$.
We have pictured  one path for each $  \lambda \in \Pi$ in the leftmost diagram in  \cref{pathsforme2222} below (for ease of notation, we do not depict the other 2 paths).   
By \cref{Strong linkage principle}, we deduce that if $\lambda \not \in \Pi$, then $d_{\lambda\mu}(t)=0$.  
We now wish to see what additional information can be deduced by \cref{strongerman}.  
Notice that only 4 of the 8  paths are dominant.
These    paths   terminate at the points 
\begin{equation}\label{egstronger}
(2,1^{12}) \quad (2^2,1^{10})
 \quad
(3^2,2^3,1^2)
\quad
(3^3,2^2,1)
\end{equation}
and are pictured in the rightmost diagram in \cref{pathsforme2222}.  
Therefore we can immediately deduce that $d_{(3^2,1^8),(2,1^{12})}=0 = d_{(2^7),(2,1^{12})}$.  
We can also deduce the following bounds on graded decomposition numbers, 
\begin{equation}\label{someexamplesthatififdon}
d_{(2,1^{12}),(2,1^{12})}\leq 1
\quad
d_{ (2^2,1^{10}),(2,1^{12})}\leq t^1
 \quad
d_{ (3^3,2^2,1),(2,1^{12})}\leq t^2
\quad
d_{ (3^2,2^3,1^2),(2,1^{12})}\leq t^1.
\end{equation}
 In fact, we shall see in  \cref{egegeegggg} that all these bounds are sharp.  
 \end{eg}

\!\!\!\!\!
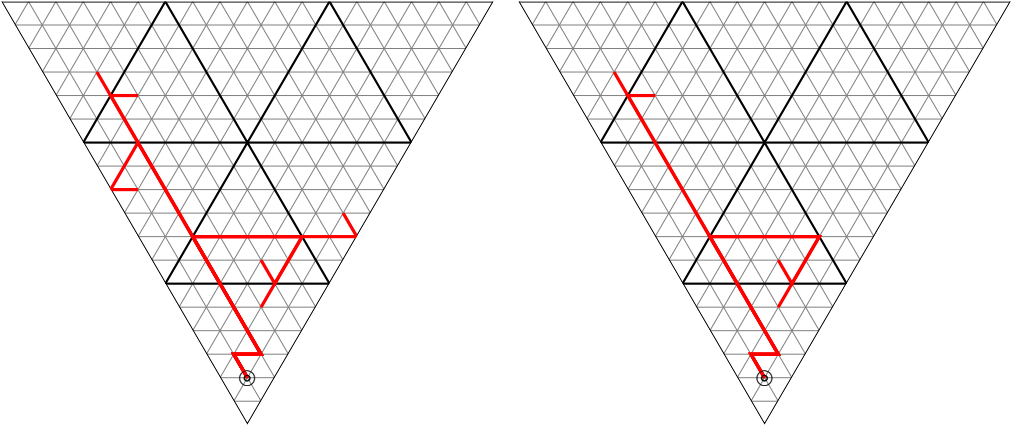
\begin{figure}[ht!]   \[\scalefont{0.8} 
   \begin{tikzpicture}[scale=0.9]
 
   \path  (0,0)  coordinate (origin); 
 
      \clip(0,0)-- (120:0.4*18)--++(0:0.4*18)--(0,0);

       \begin{scope}

        \foreach \i in {0,...,35}
  {
    \path (origin)++(60:0.4*\i cm)  coordinate (a\i);
    \path (origin)++(120:0.4*\i cm)  coordinate (b\i);
    \path (a\i)++(120:12cm) coordinate (ca\i);
    \path (b\i)++(60:12cm) coordinate (cb\i);
}
   \path (b20)++(60:0.4*5)  coordinate (step1);
      \path (step1)++(0:0.4*5)  coordinate (step2);
      \path (step2)++(-60:0.4*5)  coordinate (step3); 
      \path (step3)++(0:0.4*5)  coordinate (step4);       
      \path (step4)++(-60:0.4*5)  coordinate (step5);

  \clip(0,0)-- (b18)--(a18)--(0,0);

 \foreach \i in {0,...,35}    
  { \draw[gray, thin]  (a\i) -- (ca\i)  (b\i) -- (cb\i);
     \draw[gray, thin]   (a\i) -- (b\i)  ; } 

 \foreach \i in {0,6,12,...,24}    
  { \draw[thick]  (a\i) -- (ca\i)  (b\i) -- (cb\i);
     \draw[thick]  (a\i) -- (b\i)  ; } 
 
   \end{scope} 
   
      \path  (0,0)  --++(120:0.4*2) --++(0:0.4*1) coordinate (origin); 
 \draw(origin) node {$\circledcirc$};
 \draw[very thick, red](origin)--++(120:1*0.4)--++(0:1*0.4)
 --++(120:12*0.4);
  \draw[very thick, red](origin)--++(120:1*0.4)--++(0:1*0.4)
 --++(120:11*0.4)--++(0:1*0.4);
   \draw[very thick, red](origin)--++(120:1*0.4)--++(0:1*0.4)
 --++(120:9*0.4)--++(-120:2*0.4)--++(0:1*0.4);
   \draw[very thick, red](origin)--++(120:1*0.4)--++(0:1*0.4)
 --++(120:5*0.4)--++(0:4*0.4)--++(0:2*0.4)--++(120:1*0.4);
  \draw[very thick, red](origin)--++(120:1*0.4)--++(0:1*0.4)
 --++(120:5*0.4)--++(0:4*0.4)--++(-120:2*0.4)--++(120:1*0.4);
  \draw[very thick, red](origin)--++(120:1*0.4)--++(0:1*0.4)
 --++(120:5*0.4)--++(0:4*0.4)--++(-120:2*0.4)--++(-120:1*0.4);
  \draw(origin) node {$\circledcirc$};
\end{tikzpicture}
\quad
\begin{tikzpicture}[scale=0.9]
 
   \path  (0,0)  coordinate (origin); 
 
      \clip(0,0)-- (120:0.4*18)--++(0:0.4*18)--(0,0);

       \begin{scope}

        \foreach \i in {0,...,35}
  {
    \path (origin)++(60:0.4*\i cm)  coordinate (a\i);
    \path (origin)++(120:0.4*\i cm)  coordinate (b\i);
    \path (a\i)++(120:12cm) coordinate (ca\i);
    \path (b\i)++(60:12cm) coordinate (cb\i);
}
   \path (b20)++(60:0.4*5)  coordinate (step1);
      \path (step1)++(0:0.4*5)  coordinate (step2);
      \path (step2)++(-60:0.4*5)  coordinate (step3); 
      \path (step3)++(0:0.4*5)  coordinate (step4);       
      \path (step4)++(-60:0.4*5)  coordinate (step5);

  \clip(0,0)-- (b18)--(a18)--(0,0);

 \foreach \i in {0,...,35}    
  { \draw[gray, thin]  (a\i) -- (ca\i)  (b\i) -- (cb\i);
     \draw[gray, thin]   (a\i) -- (b\i)  ; } 

 \foreach \i in {0,6,12,...,24}    
  { \draw[thick]  (a\i) -- (ca\i)  (b\i) -- (cb\i);
     \draw[thick]  (a\i) -- (b\i)  ; } 
 
   \end{scope} 
   
      \path  (0,0)  --++(120:0.4*2) --++(0:0.4*1) coordinate (origin); 
 \draw(origin) node {$\circledcirc$};
 \draw[very thick, red](origin)--++(120:1*0.4)--++(0:1*0.4)
 --++(120:12*0.4);
  \draw[very thick, red](origin)--++(120:1*0.4)--++(0:1*0.4)
 --++(120:11*0.4)--++(0:1*0.4);
  \draw[very thick, red](origin)--++(120:1*0.4)--++(0:1*0.4)
 --++(120:5*0.4)--++(0:4*0.4)--++(-120:2*0.4)--++(120:1*0.4);
  \draw[very thick, red](origin)--++(120:1*0.4)--++(0:1*0.4)
 --++(120:5*0.4)--++(0:4*0.4)--++(-120:2*0.4)--++(-120:1*0.4);
 \draw(origin) node {$\circledcirc$};
\end{tikzpicture}
\]
\caption{
The leftmost diagram depicts    6  of the 8 paths   obtainable from $\stt^\mu$ in the
 geometry of type ${A}_2\subseteq \widehat{{A}}_2$. 
The rightmost diagram depicts  all 4 {\em dominant} paths   obtainable from $\stt^\mu$ in the
 geometry of type ${A}_2\subseteq \widehat{{A}}_2$. 
  }
 \label{pathsforme2222}
   \end{figure}

\begin{rmk}
The two paths which are not depicted in the leftmost diagram of \cref{pathsforme2222}
are 
$$(+\varepsilon_1,+\varepsilon_2,+\varepsilon_1,+\varepsilon_1,+\varepsilon_1,
+\varepsilon_3,
+\varepsilon_3,
+\varepsilon_2,
+\varepsilon_2,
+\varepsilon_2,
+\varepsilon_2,
+\varepsilon_1,+\varepsilon_1,+\varepsilon_3),$$
$$(+\varepsilon_1,+\varepsilon_2,+\varepsilon_1,+\varepsilon_1,+\varepsilon_1,
+\varepsilon_3,
+\varepsilon_3,
+\varepsilon_2,
+\varepsilon_2,
+\varepsilon_2,
+\varepsilon_2,
+\varepsilon_1,+\varepsilon_1,+\varepsilon_1).$$
\end{rmk}

\begin{rmk}\label{Ringel3}
We note that   \cref{egstrong} calculates 
 some of the decomposition numbers 
for  symmetric groups  labelled by 3-column partitions.  These are equal to the decomposition multiplicities for tilting modules for ${\rm SL}_3(\Bbbk)$ via Ringel duality \cite[Section 4]{Donkin}.
\end{rmk}

\begin{rmk}
By \cref{strongman},  the simplest case of \cref{strongerman}
(the righthand-side of \ref{412321434321442314213} is zero) is already a considerable 
 strengthening of    the classical strong linkage principle (\cref{Strong linkage principle}).   
 This  is illustrated by our discarding of non-dominant paths in   \cref{egstrong}.  
Our super-strong linkage principle is also stronger in the sense that 
it generalises the statement of \cref{strongman} (and hence  (\cref{Strong linkage principle})) to more complicated upper-bounds on decomposition numbers.  
\end{rmk}

\begin{rmk}
It is easy to see, for any $e$-regular partition $(1^n)$ and any $h\in \mathbb{N}$, 
that we can obtain a zero of the decomposition matrix generalising 
the example $d_{(3^2,1^8),(2,1^{12})}=0$ in \cref{egstrong}.  In particular, we easily obtain infinitely many  zeroes of the decomposition matrix of $\mathfrak{S}_n$ not covered by  \cite[Theorem 1]{MR564523}.     
\end{rmk}

\begin{rmk}
In \cref{maximal terms prop} below, we shall obtain the converse statement to \cref{Strong linkage principle}.  
In 
\cref{generic} we shall see that the strong and super-strong linkage orderings  coincide for {\em non-parabolic} geometries.   
\end{rmk}

\section{Generic   behaviour}\label{sec:7}

In this section, we introduce  our idea of ``generic behaviour'' for diagrammatic Cherednik algebras.  
It encapsulates the idea that  ``generically'' the behaviour of 
a {\em parabolic}  geometry can mimic that of  
  a {\em non-parabolic} geometry (and hence simplifies).   
 We prove results concerning  homomorphisms and decomposition numbers  of $A_h(n,\theta,\kappa)$ which  are   independent of the field $\Bbbk$.  

In \cref{maximalparabolic,nonparabolic}, we shall  generalise  
the `local behaviour' seen in the $\ell=1$ case (concerning points which are close together in the alcove geometry) 
to higher levels.  
In \cref{maximalnonparabolic}, we shall encounter a 
 new 
 kind of generic behaviour 
given by relating points which are `as far away from each other as possible' in the alcove geometry.  
We refer the reader to \cref{sec:9} for examples of this generic behaviour.  

\begin{defn}
We say that a subset   $\Gamma \subseteq \mptn \ell n (h)$  is  {\sf generic}    if  for every  $\la,\mu\in \Gamma$ 
we have that $(i)$ $\Path_n^+(\lambda,\stt^\mu)=\Path_n(\lambda,\stt^\mu)$ and $(ii)$ if $\la\uparrow \nu \uparrow\mu$, then $\nu \in \Gamma$.  
\end{defn}

 \begin{eg}
 Let $h=3$, $\ell=1$, and $e=4$.  The set 
 $$\Gamma =\{(3^3),(3,2^2,1^2), (2^4,1)\}$$
 is 
  not generic.   To see this, note that  the (unique) 
 path  $\stt \in \Path_9((3^3), \stt^{(2^4,1)})$ given by 
 $$
 \stt=(+\varepsilon_1+\varepsilon_2+\varepsilon_1+\varepsilon_2
 +\varepsilon_3+\varepsilon_2
  +\varepsilon_3+\varepsilon_1+\varepsilon_3
  )
 $$
does not belong to $  \Path_9^+((3^3), \stt^{(2^4,1)})$.  
This is because the point $\stt(6)=(2,3,1)\not \in \mptn 1 n$ belongs to the  
 $s_{1,2}$-wall of the dominant Weyl chamber.  
In the literature, one would say that the set $\Gamma$ is 
{\em close to the walls} of the dominant chamber.  For a similar example, revisit \cref{egstrong}.  
 \end{eg}

\begin{eg}
For  $\ell=1$,   the points lying ``around the Steinberg weight'' form a generic set.  
For arbitrary $\ell,h\in \mathbb{Z}_{>0}$, any pair of points lying in two adjacent alcoves (of the dominant region) form a generic set (see \cref{stein,stein2}).  
\end{eg}

\begin{prop} \label{maximal terms prop}
Let $ \mu \in \mathbb{E}_{h,\ell}^\circledcirc$, $\lambda \in \mathbb{E}_{h,\ell}$  and suppose 
$\lambda\uparrow \mu$.  We have that
$$\sum_{\sts\in \Path_n(\lambda,\stt^\mu)}t^{\deg(s)}= t^{ \ell(\mu) -\ell(\la)}
+\sum_{0<k<\ell(\mu) -\ell(\la)} a_{k}t^{\ell(\mu) -\ell(\la)-2k} $$
for   coefficients $a_k\in \mathbb{Z}_{\geq 0}.$  \end{prop}

\begin{proof}
Let   $\sts \in \Path_n(\la,\stt^\mu)$.  
The result clearly holds for $n=1$, we shall assume that the result holds for all paths of length less than or equal to $n$.  
Suppose that $\stt^\mu$ is a path of length $n+1$.  
We have that 
\begin{align*}
\ell(\stt^\mu(n+1)) =\ell(\stt^\mu(n)) +& |\{(\alpha,me) \mid \stt^\mu(n) \in \Hyp,\stt^\mu(n+1) \not\in \Hyp		\}| 
\intertext{and $\deg(\stt^\mu)=\deg(\stt^\mu_{\leq n})=0$.  
Given $\sts \in \Path_{n+1}(\lambda,\stt^\mu)$, we have that }
 \ell(\sts (n+1)) =\ell(\sts(n))+& 
|\{(\alpha,me) \mid \sts(n) \in \Hyp,\sts(n+1) \in \Hyp^{>}		\}| \\
&-
|\{(\alpha,me) \mid \sts(n) \in \Hyp^{>},\sts(n+1) \in \Hyp		\}|
\end{align*}
First of all, we note that 
\begin{align*}
& |\{(\alpha,me) \mid \sts(n) \in \Hyp,\sts(n+1) \in \Hyp^{<}		\}| 
\\ =  
 &|\{(\alpha,me) \mid \sts(n) \in \Hyp,\sts(n+1) \not\in \Hyp		\}| 
 -
 |\{(\alpha,me) \mid \sts(n) \in \Hyp,\sts(n+1) \in \Hyp^{>}		\}| 
\\ =  
 &|\{(\alpha,me) \mid \stt^\mu(n) \in \Hyp,\stt^\mu(n+1) \not\in \Hyp		\}| 
 -
 |\{(\alpha,me) \mid \sts(n) \in \Hyp,\sts(n+1) \in \Hyp^{>}		\}|, 
\end{align*}
and therefore, 
\begin{align*}
\ell(\stt^\mu(n+1))-\ell(\sts(n+1))=\ell(\stt^\mu(n)) - \ell(\sts(n))
&+
 |\{(\alpha,me) \mid \sts(n) \in \Hyp,\sts(n+1) \in \Hyp^{<}		\}| \\
&+
 |\{(\alpha,me) \mid \sts(n) \in \Hyp^{>},\sts(n+1) \in \Hyp 		\}| 
\end{align*}
 and by definition, we have that 
 \begin{align*}\deg(\sts)=  \deg(\sts{\downarrow}_{\leq n} )
&+ 
|\{(\alpha,me) \mid \sts(n) \in \Hyp,\sts(n+1) \in \Hyp^{<}		\}| \\
&-
|\{(\alpha,me) \mid \sts(n) \in \Hyp^{>},\sts(n+1) \in \Hyp		\}|.  
\end{align*} 
Putting these two statements together, we have that
\begin{align*}
 &\ell(\stt^\mu(n+1)) - \ell(\sts(n+1)) - \deg(\sts)\\
 =&
 (\ell(\stt^\mu(n)) - \ell(\sts(n)) - \deg(\sts{\downarrow}_{\leq n}))
 +2 |\{(\alpha,me) \mid \sts(n) \in \Hyp^{>},\sts(n+1) \in \Hyp		\}|.
\end{align*}
 The   upper degree bound statement statement and the  degree parity  follow by induction.
 The lower bound on degree follows as the first reflection through a hyperplane always increases the degree of the path (by the definition of $\stt^\mu$).  
 Finally, we  note that  $\ell(\stt^\mu(n+1))-\ell(\sts(n+1))=\deg(\sts)$ if and only if both of the following conditions are satisfied
\begin{align}\label{1}
&|\{(\alpha,me) \mid \sts(n) \in {\mathbb E}^{>} (\alpha,me),\sts(n+1) \in \Hyp		\}|=0  \text{ and }\\  \label{2}
&\deg(\sts {\downarrow}_{\leq n})=\ell(\stt^\mu(n))-\ell(\sts(n)).\end{align}
We now prove that for arbitrary $\lambda\in \mathbb{E}_{h,\ell}$ and $\mu \in \mathbb{E}_{h,\ell}^\circledcirc$ 
there exists a unique (not necessarily dominant)
 path satisfying both these conditions (and therefore is of   degree $ \ell(\mu)-\ell(\lambda)$). 
Set $\mu'= \stt^\mu(n)$ and let $\lambda'\in \mathbb{E}_{h,\ell}$ be such that $\lambda' \uparrow \mu'$.  
We may assume (by  induction) that for any such $\lambda'$, there exists a 
 unique 
path $\sts$ satisfying  condition \ref{2}.
Now suppose that $\lambda'$ is such that 
\begin{equation}\label{1111112222222}
|\{(\alpha,me) \mid \lambda'=\sts(n) \in {\mathbb E}^{\less 0} (\alpha,me),\sts(n+1) \in \Hyp		\}|=k   
\end{equation}
for $k>1$.  We set 
$\lambda'' = s_{(\alpha,me)} \cdot \lambda'$  
(for any $ (\alpha,me) $ in the above set).
We have that $\lambda'' \uparrow \lambda'$ and 
\begin{equation}\label{111111222222} |\{(\alpha,me) \mid \lambda''=\stu(n) \in {\mathbb E}^{\less 0} (\alpha,me),\sts(n+1) \in \Hyp		\}|<k    
\end{equation}
and indeed the set in \cref{111111222222} is a subset of that in 
\cref{1111112222222}.  While the reflection is not unique, there is a unique coset of the stabiliser, ${\rm Stab}(\lambda')$, of the point  $\lambda'$ in  $W^e$ for $e\in \ZZ_{>0}\cup \{\infty\}$.  
   Continuing in this fashion, we eventually obtain the unique path $\sts$ and unique point $\sts(n)=\lambda'''		\in \mathbb{E}_{h,\ell}$ satisfying \ref{1} (where \ref{2} is satisfied by our inductive assumption).
%
     The result follows.  
\end{proof}

\begin{defn}
We let    
 $\stt^\mu_\lambda \in \Path (\lambda,\stt^\mu)$ denote the  unique path of degree
$     \ell(\mu) -\ell(\la)$.  
\end{defn}

\begin{prop}\label{lots of decomp numbers}
Let $\Bbbk$ be an arbitrary field, $\lambda,\mu \in \mathbb{E}_{h,\ell}^\circledcirc$ and suppose that 
$\lambda\uparrow \mu$.  
If 
\begin{equation}\label{statement1}
[\Delta(\lambda):L(\mu)\langle k \rangle]\neq 0
\end{equation}
this implies that    $\ell(\lambda)- \ell(\mu)+2 \leq k \leq   \ell(\mu)-\ell(\lambda) $.  
We  have that 
\begin{equation}\label{statement2}
[\Delta(\lambda):L(\mu)\langle \ell(\mu)-\ell(\lambda) \rangle]=1
\end{equation}
if $\stt^\mu_\lambda \in \Path^+(\lambda,\stt^\mu)$   and is zero otherwise.  

\end{prop} 
\begin{proof}Equation \ref{statement1} follows directly from \cref{maximal terms prop}.  
 We know that $C_{\stt^\mu_\lambda} \in \Delta(\lambda)$ and that 
$C_{\stt^\mu_\lambda}$ is a vector   belonging to some simple module $L(\nu)$ 
(with $\lambda \uparrow \nu \uparrow \mu $) appearing in the submodule lattice of $\Delta(\lambda)$.   
This implies that   there exists $a,b\geq 0$ such that  
 $d_{\lambda\nu}(t) = t^a+\dots $  and $\dim_t({\sf1}_\mu L(\nu))=t^b+\dots $ such that $a+b=\ell(\mu)-\ell(\lambda)$.  
 It remains to show that $\nu=\mu$.   
We have that 
$$( \ell(\mu) -\ell(\nu))+(\ell(\nu)-\ell(\la))=  \ell(\mu) - \ell(\la).$$   
By  \cref{maximal terms prop,humathasprop}   
 we deduce that $a=\ell(\mu)-\ell(\lambda)$ and $b=0$ and $\nu=\mu$, as required.   
\end{proof}

\begin{eg}\label{egegeegggg}
By \cref{lots of decomp numbers}, we immediately  deduce
 that  the first three inequalities in \cref{someexamplesthatififdon} are actually equalities.  
For the final inequality, we suppose   $\dim_t({\sf 1}_{(2,1^{12})}\Delta (3^2,2^3,1^2))=t^1\neq d_{(3^2,2^3,1^2),{(2,1^{12})}}(t)=0$.  
Then there exists  $\nu$ 
such that ${\sf 1}_{(2,1^{12})}L(\nu)  \neq 0$ and  ${\sf 1}_{\nu}L (3^2,2^3,1^2)   \neq 0 $.  
However, we have already seen that all elements  of  $\Path^+(\nu,\stt^{(2,1^{12})})$ are of strictly positive degree. 
Therefore there does not exist any $\nu$ such that  ${\sf 1}_{(2,1^{12})}L(\nu)  \neq 0$ by \cref{humathasprop}.  
Therefore we conclude that $d_{(3^2,2^3,1^2),{(2,1^{12})}}(t)=t^1$.  
Notice that while we were unable to apply \cref{lots of decomp numbers} directly in this final case, we are 
applying the same argument as in the proof of \cref{lots of decomp numbers}.  
 \end{eg}

\begin{cor}\label{fanssdfghjlkdslfhjgdfhljkgdfshljkgsdffgdhsljkgfhldgf}
 Let $\lambda,\mu \in \mptn \ell n(h)$ be a generic pair and suppose 
  $\lambda\uparrow \mu$.    
Then
 $$
[\Delta(\lambda):L(\mu)\langle \ell(\mu)-\ell(\lambda) \rangle]=1.
$$ 

\end{cor}

\begin{proof}
 By assumption  $\Path(\lambda,\stt^\mu)=\Path^+(\lambda,\stt^\mu)$; the result follows by \cref{lots of decomp numbers,maximal terms prop}.
\end{proof}
\begin{eg}\label{generic}
Let $h=1$.  In which case,  any subset of $  \mathbb{E}_{1,\ell}^\circledcirc$ is generic.  
Therefore 
$$
d_{\lambda\mu}(t)=t^{  \ell(\mu)-\ell(\lambda)} + \dots 
$$
for all  $\lambda \uparrow \mu$. 
 In particular, $d_{\lambda\mu}(t)\neq 0$ if and only if $\lambda\uparrow\mu$.  For $h=1$ and $\ell=2$ these decomposition numbers were first calculated in \cite[(9.4) Theorem]{MW00} and \cite[Section 8]{MR1995129}.  
 For $\Bbbk=\mathbb{C}$ and $\ell$ arbitrary,
  this case was studied extensively in \cite{bcs15}.  
\end{eg}

We now present the main result of this section.    
 It will allow us to deduce the existence of many homomorphisms  between Weyl and Specht modules.  
For those not familiar with the diagram combinatorics, we recommend reading the (simpler) proof of  \cref{maximalparabolic} below, first.
We first require a simple lemma.  

\begin{lem}\label{anotherbloominlemmer}
Let $\lambda,\mu \in \mathbb{E}_{h,\ell}^+$ and suppose  $\lambda\uparrow\mu$ and  $\ell(\lambda)=\ell(\mu)-1$. 
 If $(r,c,m)\in {\rm Rem}_i(\lambda\cap\mu)$, then there is no crossing of 
$X_{(r,c,m)}$ (or its ghost) with an  $i$- or $(i+1)$-strand   (or its ghost)   in 
$C_{\stt^\mu_\lambda}$.  
\end{lem}
\begin{proof}
The diagram     $C_{\stt^\mu_\lambda}$ traces out  the   unique residue preserving bijection between the strip of northern nodes 
  $\mu \setminus  (\mu \cap \lambda)$ 
and the strip of southern nodes 
  $\lambda \setminus  (\mu \cap \lambda)$.  
A necessary condition for a  crossing between 
$X_{(r,c,m)}$ and any other strand $Y  $  is that $Y$ is non-vertical (as $X$ is vertical, by assumption) and therefore $Y$
connects the points  $({\bf I}_{y_0},0)$ and $({\bf I}_{y_1},1)$ 
 for   some $y_0\in \lambda \setminus  (\mu \cap \lambda)$,  $y_1 \in \mu \setminus  (\mu \cap \lambda)$.    
By the definition of the path $\stt^\mu_\lambda$, the strand $Y$ is added at a later stage  than the strand $X$  in the process outlined in \cref{howtobuildshit}
(this follows because   $\mathbf{I}_{y_1} > {\bf I}_{(r,c,m)} $).  
Therefore we can suppose that $X$ is added at the $a$th step and 
$Y$ is added at the $b$th step for $1\leq a < b\leq n$.  
By induction we can assume that  
$C_{\stt^\mu_\lambda{\downarrow}_{\leq b-1}}$   contains no
 crossing contradicting the statement of the lemma.  
If   $Y$ is
an $i$-strand which crosses $X_{(r,c,m)}$, then  
$$\ell(\Shape(\stt^\mu_\lambda {\downarrow}_{\leq b}) + y_0 )+2 =
 \ell(\Shape(\stt^\mu_\lambda {\downarrow}_{\leq b})+ y_1 )  
\leq 
\ell(\mu) 
$$
as we have stepped onto a hyperplane from above (since we have added an $i$-node corresponding to $y_0$ to the left of the removable $i$-node corresponding to $X$)
and off-of another hyperplane towards the origin (as we have added an $i$-node $y_0$ to the left of the addable $i$-node $y_1$).  
 If $Y$ is an $(i-1)$-strand which crosses $X_{(r,c,m)}$, 
 then again
 $$\ell(\Shape(\stt^\mu_\lambda {\downarrow}_{\leq b}) + y_0 )+2 =
 \ell(\Shape(\stt^\mu_\lambda {\downarrow}_{\leq b})+ y_1 )  
\leq 
\ell(\mu)
$$
as we have stepped  off-of two hyperplanes towards the origin (as we have added an $(i-1)$-node $y_0$ to the left of two  addable $(i-1)$-nodes $(r,c,m)$ and $y_1$).  
The result follows.  
\end{proof}

\begin{thm}\label{howtobuildshit2}
We let 
\begin{equation}\label{botherit}
\lambda= \mu^{(t)} \uparrow \mu^{(t-1)} \uparrow \dots \uparrow \mu^{(0)} =\mu
\end{equation}
be a sequence of points  in $\mathbb{E}_{h,\ell}^\circledcirc$ such that 
$$s_{\alpha^{(k)},m_k e} \cdot \mu^{(k)}=\mu^{(k-1)}$$
for $1\leq k \leq t$.  
Suppose that the path
$\stt^\mu_\lambda \in \Path (\lambda,\stt^{\mu})$ is dominant and  is  obtained by
$$
\stt^\mu_\lambda= s_{\alpha^{(t)},m_t e}^{i_t}
\dots 
s_{\alpha^{(2)},m_2 e}^{i_2}
s_{\alpha^{(1)},m_1 e}^{i_1}\cdot \stt^\mu
$$
for some sequence $i_1\leq i_2\leq \dots \leq  i_t$. 
Then  for any $\nu =\mu^{(k)}$ for $1\leq k \leq t$ we have that 
 $$\stt^{\nu}_{\lambda}=   
s_{\alpha^{(t)},m_t e}^{i_t}
\dots  
s_{\alpha^{({k+1 })},m_{k+1} e}^{i_{k+1 }}\cdot
  \stt^{\nu}
\qquad   \qquad 
\stt_{\nu}^{\mu}=   
 s_{\alpha^{(k)},m_{k} e}^{i_{k}}
\dots  
s_{\alpha^{(1)},m_1 e}^{i_1}\cdot  
   \stt^\mu_\lambda 
 $$ are both dominant paths 
    and 
we have that 
 $$  C_{\stt_{\nu}^{\mu}}C_{\stt^{\nu}_{\lambda}}  =   C_{\stt_\lambda^\mu}.
$$ 
  
\end{thm}

\begin{proof}
The statements concerning paths follow     from  \cref{strongman,maximal terms prop}.
 For ease of notation, we set 
$$\stt^{k}=\stt^{\mu^{(k-1)} }_{\mu^{(k)} } \in \Path({\mu^{(k)},{\mu^{(k-1)} } })\quad \text{ and }\quad 
\stt_n^{k}  =  \stt^{k} {\downarrow_{\leq n}}$$
for $1\leq k \leq t$ and    
  for  $  n\in \mathbb{Z}_{\geq 0}$.  
 We shall now  inductively construct the elements 
\begin{equation}\label{elements of interest} C_{\stt^{ i }_n}  
\qquad 
  C_{\stt^{ 1 }_n}C_{\stt^{ 2 }_n}\dots C_{\stt^{ t }_n}
\qquad 
\text{ and }
\qquad 
C_{\stt^\mu_\lambda{\downarrow}_{\leq n}}
\end{equation}
for $1\leq i \leq t$ simultaneously by induction on $n \in \mathbb{Z}_{\geq 0}$ 
and  verify that 
\begin{align}\label{07532849157304} C_{\stt^{ 1 } }C_{\stt^{ 2 } }\dots C_{\stt^{ t } }=C_{\stt^\mu_\lambda }
 \end{align}
  For $n=1$,  the   $t$   elements and the product  in \cref{elements of interest} 
are all equal to the same idempotent with one solid strand.  
 Given a fixed $  n$, we obtain each of the diagrams  
$$C_{\stt^{ i }_{n+1}} \text{ 
 and  }
   C_{\stt^{ 1 }_{n+1}}C_{\stt^{ 2 }_{n+1}}\dots C_{\stt^{ t }_{n+1}}$$ from those in \cref{elements of interest}
 by adding a single strand.  
 For $1\leq i \leq t$ we denote this strand by $X^i_{n+1}$ and we let $X_{n+1}=X^1_{n+1} \circ X^2_{n+1}\circ \dots \circ X^t_{n+1}$ denote the composite strand. 
 We suppose that $\stt^\mu({n+1} ) =\stt^\mu({n})+\varepsilon_c$ and therefore (in the notation of \cref{howtobuildshit2}) 
 that $X^i_{n+1}$ connects the northern and southern points labelled by boxes  added in the 
 $$c_{k-1} = s_{\alpha^{(k-1)}}   \dots s_{\alpha^{(1)}}  (c)
  \quad
  \text{and}
  \quad 
c_{k} =  s_{\alpha^{(k)} }\dots s_{\alpha^{(1)} }(c)$$
columns 
respectively.   
 In particular, the elements 
\begin{equation}\label{BOOM!} C_{\stt^{ 1 }_{n+1}}C_{\stt^{ 2 }_{n+1}}\dots C_{\stt^{ t }_{n+1}}\quad \text{ and }\quad C_{\stt^\mu_\lambda{\downarrow}_{\leq {n+1}}}
\end{equation}   trace out the same bijections, as required. It remains to show that the product on the lefthand-side of \cref{BOOM!} contains no double-crossings.    
  In fact, it is enough to show that the product  contains no double-crossing of strands labelled by {\em adjacent} or {\em equal} residues (as all other double-crossings can be trivially removed).  

By induction, we can assume that  there are no double-crossings in   $C_{\stt^{ 1 }_{n}}C_{\stt^{ 2 }_{n}}\dots C_{\stt^{ t }_{n}} $. 
We suppose that there is a double-crossing of the strands  
 $X^{n+1} $ and $X^i $   in 
 $C_{\stt^{ 1 }_{n+1}}C_{\stt^{ 2 }_{n+1}}\dots C_{\stt^{ t }_{n+1}} $ for some $1\leq i \leq n$. 
In which case,  
 there is a crossing both of the strands  $X^{n+1}_a$ and $X^i_a$  
  in the diagram $C_{\stt^{a}_{n+1}}$
  and  the strands  $X^{n+1}_b$ and $X^i_b$  in $C_{\stt^{b}_{n+1}}$ for $1\leq a < b\leq t$.  
 We assume that $b$ is minimal with this property.  
The strands $X^i_a$ and $X^{n+1}_b$ are both vertical.  
We shall identify strands with the corresponding nodes at which they terminate in the obvious fashion.  


 Importantly, 
 $X^{n+1}_b$ is a vertical strand corresponding to a {\em removable} node 
of 
  $\mu^{(b)}\cap \mu^{(b+1)} $.  
Therefore,  
$ \res(X^{i}_b ) \neq    	\res(X^{n+1}_b ), $
  $ \res(X^{n+1}_b )-1 $ 
 by \cref{anotherbloominlemmer}.  
By assumption, the strands  $X^{i} $ and $X^{n+1} $ are of adjacent or equal residue;  therefore
$ \res(X^{i} )= 				\res(X^{n+1} )+1$.  

Suppose that  $X^{i}_b$ connects  nodes $(r,c,m) \in 
\mu^{(b)}\setminus  \mu^{(b)}\cap \mu^{(b+1)}$
and $(r',c',m') \in 
\mu^{(b+1)}\setminus  \mu^{(b)}\cap \mu^{(b+1)}$.  
 Now, we assume that $(r,c,m)$   is {\em not} the removable node in the strip $\mu^{(b)}\setminus  \mu^{(b)}\cap \mu^{(b+1)}$;
in other words we 
suppose that   $(r+1,c,m) \in \mu^{(b)}\setminus  \mu^{(b)}\cap \mu^{(b+1)}$. 
 We let $Y^{i}_b$ denote the strand connecting points 
 $(\mathbf{I}_{(r+1,c,m)},1)$
and  $(\mathbf{I}_{(r'+1,c',m')},0)$. 
  We have assumed that  
$ \res(r,c,m)-1=  	\res(X^{n+1} )$ and  $\mathbf{I}_{(r',c',m')}$ is strictly less than the $x$-coordinate of the vertical strand  $X^{k+1}_b$.  
Therefore by \cref{seperating} it follows that   $\mathbf{I}_{(r'+1,c',m')} <\mathbf{I}_{(r',c',m')}+h\ell$ is strictly less than the $x$-coordinate of the vertical strand  $X^{k+1}_b$.  
 Therefore   
 $X^{k+1}_b$
 and  
$Y^{i}_b$  cross   in $C_{\stt^b_{n+1}}$     and have the same residue.  
 Therefore we can repeat the argument above to get a contradiction; we hence  deduce that 
$X^{i}_b$ is a removable node of $\mu^{(b)}$ and $\mu^{(b+1)}$.   
 Therefore we can assume that 
$X^{i}_b$
is a removable node of $\mu^{(b)}$ and 
$\res(X^{i}_b)= \res(X^{k+1}_b)+1$.  

Now, suppose that $X^{i}_b$
 is not a removable node of $\mu^{(a)}\cap \mu^{(a+1)}$.  This implies that there is some reflection labelled by $a< d <b$, which adds a strip at the end of the column containing $X^{i}_b$.  
This  results in  either a double-crossing between   strands 
in 
$$ C_{\stt^{ d }_k}\dots C_{\stt^{ b }_k}\quad \text{ or }\quad C_{\stt^{ 1 }_{k+1}}C_{\stt^{ 2 }_{k+1}}\dots C_{\stt^{ d }_{k+1}}$$ and hence a contradiction either   by induction, 
 or by  our assumption of the minimality of $b$, respectively.  
Therefore we can assume that $X^{i}_b$
 {\em is}   a removable node of $\mu^{(a)}\cap \mu^{(a+1)}$. 
 Finally, we have that 
 $X^{i}_a$ is a removable node 
of the partition $\mu^{(a)}$ and   $ \res(X^{i}_a )= 				\res(X^{k+1}_a )+1$.  
Therefore,  
 we obtain a contradiction  by \cref{anotherbloominlemmer}.  
     Thus we conclude that there are no double-crossings in the product.  The result follows.  
 \end{proof}

\begin{cor}\label{howtobuildshit3}
If  $\la,\mu$ are a generic pair such that $\lambda\uparrow\mu$, we have that
\begin{align*}
\dim_t(\Hom_{A_h(n,\theta,\kappa)}(  \Delta(\mu) ,\Delta(\lambda))) 
= t ^{ \ell(\mu)-\ell(\lambda)} + \dots 	
\end{align*}
where the remaining terms are all of strictly smaller degree.  
This highest-degree homomorphism is given (up to scalar multiplication)   by 
 $$\varphi_\la^\mu: C_{\stt^\mu} \mapsto C_{\stt_\la^\mu}.$$ 
If $\lambda\uparrow \nu \uparrow \mu$ with $\nu$ belonging to the sequence \cref{botherit}, then 
$$\varphi_\la^\nu\circ\varphi_\nu^\mu= \varphi_\la^\mu.$$
 
  \end{cor}

\begin{proof}
  If $\ell(\lambda)=\ell(\mu)-1$ and $\lambda \uparrow \mu$, then $d_{\lambda\mu}(t)=t$.
Therefore
$\dim_t(\Hom_{A_h(n,\theta,\kappa)}(  P(\mu) ,\Delta(\lambda))) 
= t ^1
$.  That this homomorphism factors through the projection $P(\mu)\to \Delta(\mu)$ 
 follows   from highest weight theory and 
the fact
that there does not exist any $\alpha$ such that $\lambda\uparrow \alpha \uparrow \mu$.  
Now, let $\lambda$ and $\mu$ be an  arbitrary generic pair. 
By \cref{07532849157304}, we have that the composition of the degree 1 homomorphisms along the sequence \cref{botherit} is itself a homomorphism and equal to $\varphi_\la^\mu$.  The result follows.  
 \end{proof}

\subsection{Maximal parabolic behaviour}\label{maximalparabolic}
We shall now consider  generic sets  whose elements  are permuted by some finite group, $\mathfrak{S}_{a+b}\leq \widehat{\mathfrak{S}_{h\ell}}$,  
and such that the  stabiliser of any given point  $\lambda \in \Gamma$ is a  maximal parabolic subgroup  $\mathfrak{S}_{a} \times \mathfrak{S}_{b} \leq \mathfrak{S}_{a+b}$.

\begin{defn} Let  $a,b \in \mathbb{Z}_{>0}$ and $x \in \ZZ/e\ZZ$.  
Let $\gamma\in \mptn \ell n (h)$ be any multipartition with precisely $a+b$  
addable $x$-nodes and zero  removable $x$-nodes.    
We let   $\Gamma_{a,b}\subseteq \mptn \ell n (h)$  denote the set of multipartitions which can be obtained by adding a total of $a$ 
distinct  $x$-nodes to
$\gamma\in \mptn \ell n (h)$. 
\end{defn}

\begin{eg}   If $a+b=h $ and  $\ell=1$ then $\gamma$ is a translate of the Steinberg point $(e-1)\rho$.  
\end{eg}

\begin{eg}[Stepping off of a wall]
Let $\mu \in \mathbb{E}_{h,\ell}^\circledcirc$ be an $e$-regular point.  We have that  $\gamma=\mu-\varepsilon_X$ for any $X \in {\rm Rem}(\mu)$ belongs to either one or zero hyperplanes.  If  $\gamma=\mu-\varepsilon_X \in \mathbb{E}_{\alpha,me}$ for some  $(\alpha,me)$,  then we say that $\mu$ is obtained from $\gamma$ by {\sf stepping off the $(\alpha,me)$-wall}.  
If $\mu$ is obtained from $\gamma$ by {\sf stepping off the $(\alpha,me)$-wall}, then
$\Gamma_{1,1} =\{\mu, s_{(\alpha,me)}\cdot \mu\}$ forms a generic set.  
\end{eg}

We let $\mathfrak{S}_{a+b} \leq \widehat{\mathfrak{S}_{h\ell}}$ denote the group which acts by faithfully permuting
 the elements of $\Gamma_{a,b}$.  We remark that $\mathfrak{S}_{a+b}$ fixes the point  $\gamma\in \mathbb{E}_{h,\ell}^\circledcirc$.    
Given any fixed choice $\lambda \in \Gamma_{a,b}$, we let   
$\mathfrak{S}_{a,b}=\mathfrak{S}_{a} \times \mathfrak{S}_{b} \leq \mathfrak{S}_{a+b}$ denote the subgroup which fixes 
$\lambda \in \Gamma_{a,b}$  -- this is the subgroup whose elements   
  trivially permute the columns with a removable (respectively addable) $x$-node amongst themselves.  

\begin{lem}\label{littlelema}
For $\lambda,\mu \in  \Gamma_{a,b}$  we have that $\Path_n^+(\lambda,\stt^\mu)=\Path_n(\lambda,\stt^\mu)$ (hence
  the set  $ \Gamma_{a,b}$ is generic).  
Any pair of partitions $\lambda$ and $\mu$ can be written in the form 
$$
\mu = \gamma + A_{c_1}+ A_{c_2}+ \dots + A_{c_a}
\quad 
\lambda = \gamma + A_{q_1}+ A_{q_2}+ \dots + A_{q_a}
$$
where $ A_{c_1} \rhd A_{c_2} \rhd  \dots  \rhd A_{c_a}$ and $ A_{q_1} \rhd A_{q_2} \rhd  \dots  \rhd A_{q_a}$.  
For $\lambda \uparrow \mu$, we have that 
\begin{align}\label{lhlfadsqywuridhskjluyweqt}
\stt^\mu_\lambda = 
s_{\varepsilon_{q_a-c_a},m_re} ^{i_a} \dots 
s_{\varepsilon_{q_2-c_2},m_2e} ^{i_2} 
s_{\varepsilon_{q_1-c_1},m_1e}^{i_1}  \cdot \stt^\mu  
\end{align}
for some $i_1<i_2<\dots <i_a$.  
\end{lem}
\begin{proof}
 By construction,  any $\sts \in \Path(\lambda,\stt^\mu)$ is of the form
\begin{align}\label{tagmeimyours}
  s^{i_{a+b-1}}_{\varepsilon_{\sigma(c_{a+b-1})}-\varepsilon_{c_{a+b}},m_{a+b-1}e}
  \dots 
 s^{i_2}_{\varepsilon_{\sigma(c_2)}-\varepsilon_{c_2 },m_2e}
  s^{i_1}_{\varepsilon_{\sigma(c_1)}-\varepsilon_{c_1 },m_1e} \cdot \stt^\mu. 
\end{align}
 where $\lambda = \sigma \cdot \mu$ for  some $\sigma \in \mathfrak{S}_{a+b}$ (the converse is not true).
All such paths are dominant, by construction.  Set   $\sigma(c_i)=q_i$ for $1\leq i \leq a$.   
The $i_k$th step in  the path \cref{tagmeimyours} corresponds to adding a box   as far to the left as possible.
Therefore by \cref{seperating}, there are 
 no removable $x$-nodes  to the right of $A_{q_i}$.
 Hence this step satisfies \cref{1}.  Repeating for every step, we deduce that \cref{2} is also satisfied.  The result follows.  
 \end{proof}
 
\begin{thm}\label{eisarbandellisarb}
Let $\Bbbk$ be an arbitrary field.  Given $\lambda,\mu \in \Gamma_{a,b}\subseteq \mptn \ell n (h)$, we have that 
$$
d_{\lambda,\mu}(t)=n_{\lambda,\mu}(t)
$$
is the associated Kazhdan--Lusztig polynomial of type $A_{a-1}\times A_{b-1}  \subseteq {A}_{a+b}$.  A closed  combinatorial formula  for these   polynomials is given in \cite[Section 3]{MR3105756} and \cite[Section 5]{MR2600694}.
\end{thm}

\begin{proof} 

Our assumption that $\kappa\in I^\ell$ is $h$-admissible implies that we can apply 
  \cite[Theorem 4.30]{BS15}. The result follows.  
\end{proof}   

\begin{rmk}
 These Kazhdan--Lusztig polynomials have recently made two prominent appearances in representation theory.  
 The first is in the work of Kleshchev, Chuang, Miyachi, Tan  for $\ell=1$ and $e\in\mathbb{N}$ \cite{klesh1box,cmt08,tanteo13}
 and the latter is in the work of Brundan--Stroppel and Mathas--Hu   for $\ell=2$ and  $e=\infty$ \cite{MR2781018} and  \cite[Theorems B3 \& B5]{MR3356809}.  
  \cref{eisarbandellisarb} applies to the former family of results ($\ell=1$ and $e\in\mathbb{N}$) and generalises all these results to higher levels.
 In the latter case ($\ell=2$ and  $e=\infty$) these results follow easily from \cref{eisinfinity} as the algebra $A(n,\theta,\kappa)$ is a basic algebra.  
 We are unaware of any direct link explaining   these two distinct appearances of the same family of Kazhdan--Lusztig polynomials.
\end{rmk}

As the combinatorics of this case is particularly simple, we are able to prove a strengthened version of \cref{howtobuildshit2}.
 Namely, we can understand the composition of {\em any} chain of these homomorphisms.  

\begin{thm}\label{maximalparabolic}
Let $\Bbbk$ be a field of arbitrary characteristic. 
For  $\la,\mu\in \Gamma_{a,b}$ such that $\lambda\uparrow\mu$, we have that
\begin{align*}
\dim_t(\Hom_{A_h(n,\theta,\kappa)}(  \Delta(\mu) ,\Delta(\lambda))) 
= t ^{ \ell(\mu)-\ell(\lambda)} + \dots 	
\end{align*}
where the remaining terms are all of strictly smaller degree.  
This highest-degree homomorphism is given  (up to scalar multiplication)    by 
 $$\varphi_\la^\mu: C_{\stt^\mu} \mapsto C_{\stt_\la^\mu}$$ 
and the composition of these homomorphisms is given by 
$$\varphi_\la^\nu \circ \varphi_\nu^\mu= \varphi_\la^\mu $$
for all $\lambda, \nu , \mu \in \Gamma_{a,b}$ with $\lambda \uparrow \nu \uparrow \mu$.\end{thm}
\begin{proof}
 
    If $\ell(\lambda)=\ell(\nu)-1$ and $\lambda \uparrow \nu$, then the result follows as in \cref{howtobuildshit2}.
To deduce the result, it will now suffice to show that, for {\em any} sequence of the form
 $$
\lambda 
 = \mu^{(t)} \uparrow 
 \mu^{(t-1)} \uparrow  
 \dots	 \uparrow  
 \mu^{(1)} \uparrow  
 \mu^{(0)} = \mu, $$ 
 that 
\begin{align}\label{align}
C_{\stt_t^{t-1}}
C_{\stt_{t-1}^{t-2}}
\dots 
C_{\stt^0_{1}}
= C_{\stt^\mu_\lambda}
\quad
\text{ for   }  {\stt^{k-1}_{k}}= {\stt^{\mu^{(k-1)}}_{\mu^{(k)}}}\in \Std({\mu^{(k )}},{\mu^{(k-1)}}).  
\end{align}
 Let $\lambda,\mu \in \Gamma_{a,b}$.  
We continue  with the notation   of \cref{littlelema}.   By \cref{lhlfadsqywuridhskjluyweqt}, 
 we have that 
$ {\stt^\mu_\lambda}(A_{q_i})=\mathbf{I}_{A_{p_i}}  $    and 
$ {\stt^\mu_\lambda}(r,c,m)=\mathbf{I}_{(r,c,m)}  $ for $(r,c,m) \neq A_{q_i}$ for $1\leq i \leq a$.  
In particular,  
$ C_{\stt^\mu_\lambda}$ is obtained from $C_{\SSTT^\gamma}$ by adding a total of $a$ strands  connecting 
the northern points $\{(\mathbf{I}_{A_{c_i}},1)\mid 1\leq i \leq a\}$ with the southern points  
 $\{(\mathbf{I}_{A_{q_i}},0) \mid 1\leq i \leq a\}$ in such a fashion that these $a$ strands {\em do not cross  each other} at any point.  
 Now, composing any number of such diagrams, we clearly obtain a diagram with $a$ strands which do not cross, and trace out the bijection between two such sets.  Hence \cref{align} and the result follows.  
\end{proof}

\begin{eg} \label{contannnnn}
 We let $\gamma=(\varnothing, (1), (1^2))\in \mptn 3 4 (1)$ with $e=4$ and $\kappa=(0,1,2)$.  
We have that 
$$
\Gamma_{2,1}=
\{
 (\varnothing, (1^2), (1^3))
((1), (1), (1^3)),
 ((1), (1^2), (1^2))
 \}.  
$$
The graded decomposition  matrix for this subquotient  is given by 
$$
\left(\begin{array}{ccc}1 & 0 & 0 \\t & 1 & 0 \\t^2 & t & 1\end{array}\right)
$$
and every decomposition number is given by a homomorphism between Weyl modules.  
  The degree 2 homomorphism is determined as the  composition of the two degree 1 homomorphisms, this can be seen by concatenating the two diagrams  in \cref{yetanotherfgure}.  
\begin{figure}[ht!]

$$   \scalefont{0.7} \begin{tikzpicture}[scale=0.7] 
\draw (-3,-1) rectangle (13,2);   
     \draw[wei2]   (0.1,-1)--(0.1,2) ;  
     \draw[wei2]   (0.5,-1)--(0.5,2) ;  
             \draw[wei2]   (0.9,-1)--(0.9,2) ;  
 \draw     (3.6,2)    to[in=90,out=-90]   (0.2,-1) ;  
 \draw[densely dashed]     (3.6-3+0.1,2)    to[in=90,out=-90]   (0.2-3+0.1,-1) ;  
 
  \draw     (7,2)    to[in=90,out=-90]   (7,-1) ;  
 \draw[densely dashed]     (7-3+0.1,2)    to[in=90,out=-90]   (7-3+0.1,-1) ;
   \node [below] at (7,-1)  {  $0$};   
 \draw     (.6,-1)   to    (.6,2) ;  
 \draw[densely dashed]     (.6-3+0.1,2)   to    (.6-3+0.1,-1) ;  
  \draw     (4,-1)      to  (4,2) ;  
 \draw[densely dashed]     (4-3+0.1,-1)   to   (4-3+0.1,2) ;  
  \draw     (1,2)   to    (1,-1) ;  
 \draw[densely dashed]     (1-3+0.1,2)   to    (1-3+0.1,-1) ;  
  \node [below] at  (4,-1)    {  $1$};    \node [below] at (0.2,-1)  {  $0$};     
 \node [below] at (0.6,-1)  {  $1$};     
 \node [below] at (1,-1)  {  $2$};     


           \end{tikzpicture}
$$ 
$$   \scalefont{0.7} \begin{tikzpicture}[scale=0.7] 
\draw (-3,-1) rectangle (13,2);   
     \draw[wei2]   (0.1,-1)--(0.1,2) ;  
     \draw[wei2]   (0.5,-1)--(0.5,2) ;  
             \draw[wei2]   (0.9,-1)--(0.9,2) ;  
 \draw     (0.2,2)   to[in=90,out=-90]   (0.2,-1) ;  
 \draw[densely dashed]     (0.2-3+0.1,2)    to[in=90,out=-90]   (0.2-3+0.1,-1) ;  
 
  \draw     (7,2)    to[in=90,out=-90]   (3.6,-1) ;  
 \draw[densely dashed]     (7-3+0.1,2)    to[in=90,out=-90]   (3.6-3+0.1,-1) ;
   \node [below] at (3.5,-1)  {  $0$};   
 \draw     (.6,-1)   to    (.6,2) ;  
 \draw[densely dashed]     (.6-3+0.1,2)   to    (.6-3+0.1,-1) ;  
  \draw     (4,-1)      to  (4,2) ;  
 \draw[densely dashed]     (4-3+0.1,-1)   to   (4-3+0.1,2) ;  
  \draw     (1,2)   to    (1,-1) ;  
 \draw[densely dashed]     (1-3+0.1,2)   to    (1-3+0.1,-1) ;  
  \node [below] at  (4,-1)    {  $1$};    \node [below] at (0.2,-1)  {  $0$};     
 \node [below] at (0.6,-1)  {  $1$};     
 \node [below] at (1,-1)  {  $2$};     


           \end{tikzpicture}
$$

\caption{Concatenating  basis elements labelled by 
$\stt^{ (\varnothing, (1^2), (1^3))}_{
((1), (1), (1^3)) }$
and $\stt^{((1), (1), (1^3)) }_{
 ((1), (1^2), (1^2))}$ as in 
\cref{contannnnn} }
\label{yetanotherfgure}
\end{figure}
\end{eg}

\begin{rmk}\label{blob}
Fix our field to be $\mathbb{C}$.
For the type $A$ Temperley Lieb-algebra, $\Q_{2,1,n}(\kappa)$, 
these homomorphisms   provide {\em all possible} homomorphisms between Weyl modules.
For the type $B$ Temperley Lieb-algebra (or blob algebra), $\Q_{1,2,n}(\kappa)$, 
these homomorphisms are obtained by lifting homomorphisms from 
$\Q_{2,1,n}(\kappa)$ (and provide `half' of all homomorphisms between Weyl modules \cite{MW00}).  
\end{rmk}

\subsection{Non-parabolic finite behaviour}\label{nonparabolic}
We now consider  generic sets $\Gamma_{h\ell}$ whose elements  are permuted by the finite group, $\mathfrak{S}_{h\ell  }\leq \widehat{\mathfrak{S}_{h\ell}}$,  
and such that the  stabiliser of any given point  $\lambda \in \Gamma_{h\ell}$ is the trivial subgroup.

\begin{defn}\label{stein} We say that $\gamma \in \mathbb{E}_{h,\ell}^\circledcirc$ is a {\sf maximal vertex} if for every $\alpha \in \Phi$ there exists some $m\in \ZZ$ such that $\gamma \in \mathbb{E}_{\alpha,me}$.  
   By assumption, the parts  are $\gamma$ are all distinct and  therefore  there exists $\sigma \in {\mathfrak{S}_{h\ell}}$ such that 
$$
\langle \gamma , \varepsilon_{\sigma(h\ell)} \rangle
<\dots <
\langle \gamma , \varepsilon_{\sigma(2)} \rangle  
<\langle \gamma , \varepsilon_{\sigma(1)} \rangle
$$
We let  $\mathfrak{S}_{h\ell} \leq \widehat{\mathfrak{S}_{h\ell}}$ denote the   group generated by the reflections which fix $\gamma \in \mptn \ell n (h)$.  
We let $\alpha \in \mptn \ell n$ denote any multipartition such that 
$$
\langle \alpha , \varepsilon_{\sigma(h\ell)} \rangle
<\dots <
\langle \alpha , \varepsilon_{\sigma(2)} \rangle  
<\langle \alpha , \varepsilon_{\sigma(1)} \rangle < e.
$$
We let $\Gamma_{h\ell}$ denote  the set $ \Gamma_{h\ell}= \mathfrak{S}_{h\ell} \cdot  (\gamma+\alpha) $.  
 \end{defn}
There are  $(h\ell)!$ distinct elements of $\Gamma_{h\ell}$ and the group     $\mathfrak{S}_{h\ell}$ acts faithfully by permuting the elements of  $ \Gamma_{h\ell}$.  
 We have chosen  $\alpha$ in such a way that $\gamma+w\alpha$ is $e$-regular 
and  
$\ell(\gamma+w\alpha)= \ell(\gamma + \alpha) - \ell (w)$ for any $w \in \mathfrak{S}_{h\ell}$.  

\begin{eg}\label{stein2}
For $\ell=1$, the point  
  $\gamma$ is a 
 translate of the Steinberg point $(e-1)\rho$. 
 The corresponding set $
 \Gamma_{h\ell}$ is given by an $e$-regular set of points   ``around a Steinberg vertex''.   
\end{eg}

\begin{thm}
For  $\la,\mu\in \Gamma_{h\ell}$ such that $\lambda\uparrow\mu$, we have that
\begin{align*}
\dim_t(\Hom_{A_h(n,\theta,\kappa)}(  \Delta(\mu) ,\Delta(\lambda))) 
= t ^{ \ell(\mu)-\ell(\lambda)} + \dots 	
\end{align*}
where the remaining terms are all of strictly smaller degree.  
 \end{thm}

\begin{proof} 

Given $\lambda,\mu \in \Gamma_{h,\ell}$ we have that $\ell_\alpha(\mu)-\ell_\alpha(\lambda)=0,1$ for any $\alpha \in \Phi$.  Therefore any $\sts  \in \Path(\lambda,\stt^\mu)$ is obtained from $\stt^\mu$ by applying a sequence of reflections from $\mathfrak{S}_{h\ell}$ (and {\em not} any of the parallel  translates of these reflections). 
Now,  any hyperplane labelled by such a reflection belongs to 
$ \mathbb{E}_{h,\ell}^\circledcirc$ and so this path is dominant, as required.  The result follows by  \cref{howtobuildshit2}.   
   \end{proof}
 
\subsection{Non-parabolic affine behaviour }\label{maximalnonparabolic} One of the features of the previous two sections (and the classical  results that these sections generalise) is 
that they relate Weyl and Specht  modules which are labelled by points which are ``close together'' in the alcove geometry.
We now consider a family of homomorphisms from alcoves which are ``as far apart as possible''.  
These points are related by  permuting like-labelled columns between distinct components.

\begin{defn}\label{affinevertec}
   We say that 
$ 
\gamma\in \mathbb{E}_{h,\ell}^\circledcirc $ is an {\sf affine vertex} if 
$\gamma^{(i)}_{t+1} < 
\gamma^{(j)}_{t-1}
 \text{ for all }1\leq i,j \leq \ell \}
$   for some fixed $1\leq t \leq h$.
We let $\mathfrak{S}_\ell$ denote the subgroup generated by the reflections 
\begin{align}\label{somrecltions}
\langle s_{\varepsilon_{hi+t}-\varepsilon_{hj+t},me}\mid  
1\leq i < j \leq \ell,	m\in \ZZ	\rangle 
\end{align}
and we set 
$$\Gamma_t = \{	\lambda \in  \mathbb{E}_{h,\ell}^\circledcirc
 \mid 
\lambda \in \mathfrak{S}_{\ell} \cdot \gamma 
	\}.
  $$ \end{defn}

\begin{thm}\label{affinevertec2}
Let $1\leq t \leq h$.  
For  $\la,\mu\in \Gamma_{t}$ such that $\lambda\uparrow\mu$, we have that
\begin{align*}
\dim_t(\Hom_{A_h(n,\theta,\kappa)}(  \Delta(\mu) ,\Delta(\lambda))) 
= t ^{ \ell(\mu)-\ell(\lambda)} + \dots 	
\end{align*}
where the remaining terms are all of strictly smaller degree.  
  \end{thm}

\begin{proof} 
For $h=t=1$ the fact that $ \Gamma_t$ is generic is clear.  
For more general $h,t \in \mathbb{Z}_{>0}$,  the result follows by translating the paths within the geometry.  
\end{proof}

\begin{rmk}
For $\ell=1$, we have that  $\mathfrak{S}_\ell$ is the trivial group and   these sets are clearly trivial.  For any $\ell>1$, we will {\em always} obtain many infinite chains (as $n\rightarrow \infty$) of homomorphisms of the form 
in \cref{affinevertec,affinevertec2}.
\end{rmk}

  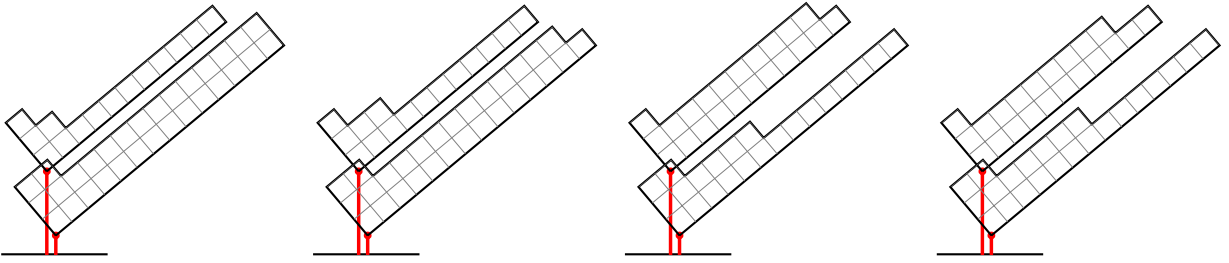
\begin{figure}[ht]\captionsetup{width=0.9\textwidth}
\[
\begin{tikzpicture}[scale=0.4]
\begin{scope}
\draw[thick](2,-1.95)--(-1.5,-1.95); 
{     \path (0,0.8) coordinate (origin);     \draw[wei2] (origin)   circle (2pt);
 
\clip(-2.2,-1.95)--(-2.2,8.3)--(8.7,8.3)--(8.7,-1.95)--(-2.2,-1.95); 
\draw[wei2] (origin)   --(0,-2); 
\draw[thick] (origin) 
--++(130:3*0.7)
--++(40:1*0.7)
--++(-50:1*0.7)	
--++(40:1*0.7)   --++(-50:1*0.7) 
--++(40:9*0.7)   --++(-50:1*0.7) 
--++(-50-90:11*0.7);
 \clip (origin) 
--++(130:3*0.7)
--++(40:1*0.7)
--++(-50:1*0.7)	
--++(40:1*0.7)   --++(-50:1*0.7) 
--++(40:9*0.7)   --++(-50:1*0.7) 
--++(-50-90:11*0.7);
\path  (origin)--++(40:0.35)--++(130:0.35)  coordinate (111);  
\foreach \i in {1,...,100}
{
\path (origin)++(40:0.7*\i cm)  coordinate (a\i);
\path (origin)++(130:0.7*\i cm)  coordinate (b\i);
\path (a\i)++(130:20cm) coordinate (ca\i);
\path (b\i)++(40:20cm) coordinate (cb\i);
\draw[thin,gray] (a\i) -- (ca\i)  (b\i) -- (cb\i); } 
}
\end{scope}
\begin{scope}
{   
\path (0,-1.3)++(40:0.3*0.7)++(-50:0.3*0.7) coordinate (origin);  
\draw[wei2]  (origin)   circle (2pt);
\clip(-2.2,-1.95)--(-2.2,8.3)--(8.7,8.3)--(8.7,-1.95)--(-2.2,-1.95); 
\draw[wei2] (origin)   --++(-90:4cm); 
\draw[thick] (origin) 
--++(130:3*0.7)
--++(40:2*0.7)
--++(-50:1*0.7)	
--++(40:12*0.7) --++(-50:2*0.7)
 --++(-50-90:14*0.7);
 \clip (origin)  
--++(130:3*0.7)
--++(40:2*0.7)
--++(-50:1*0.7)	
--++(40:12*0.7) --++(-50:2*0.7)
 --++(-50-90:14*0.7);
 \path (40:1cm) coordinate (A1);
\path (40:2cm) coordinate (A2);
\path (40:3cm) coordinate (A3);
\path (40:4cm) coordinate (A4);
\path (130:1cm) coordinate (B1);
\path (130:2cm) coordinate (B2);
\path (130:3cm) coordinate (B3);
\path (130:4cm) coordinate (B4);
\path (A1) ++(130:3cm) coordinate (C1);
\path (A2) ++(130:2cm) coordinate (C2);
\path (A3) ++(130:1cm) coordinate (C3);
\foreach \i in {1,...,19}
{
\path (origin)++(40:0.7*\i cm)  coordinate (a\i);
\path (origin)++(130:0.7*\i cm)  coordinate (b\i);
\path (a\i)++(130:12cm) coordinate (ca\i);
\path (b\i)++(40:12cm) coordinate (cb\i);
\draw[thin,gray] (a\i) -- (ca\i)  (b\i) -- (cb\i); } 
}  \end{scope}
\end{tikzpicture}\!\!\!\!
 \begin{tikzpicture}[scale=0.4]
\begin{scope}
\draw[thick](2,-1.95)--(-1.5,-1.95); 
{     \path (0,0.8) coordinate (origin);     \draw[wei2] (origin)   circle (2pt);
 
\clip(-2.2,-1.95)--(-2.2,8.3)--(8.7,8.3)--(8.7,-1.95)--(-2.2,-1.95); 
\draw[wei2] (origin)   --(0,-2); 
\draw[thick] (origin) 
--++(130:3*0.7)
--++(40:1*0.7)
--++(-50:1*0.7)	
--++(40:2*0.7)   --++(-50:1*0.7) 
--++(40:8*0.7)   --++(-50:1*0.7) 
--++(-50-90:11*0.7);
 \clip (origin) 
--++(130:3*0.7)
--++(40:1*0.7)
--++(-50:1*0.7)	
--++(40:2*0.7)   --++(-50:1*0.7) 
--++(40:8*0.7)   --++(-50:1*0.7) 
--++(-50-90:11*0.7); 
\path  (origin)--++(40:0.35)--++(130:0.35)  coordinate (111);  
\foreach \i in {1,...,100}
{
\path (origin)++(40:0.7*\i cm)  coordinate (a\i);
\path (origin)++(130:0.7*\i cm)  coordinate (b\i);
\path (a\i)++(130:20cm) coordinate (ca\i);
\path (b\i)++(40:20cm) coordinate (cb\i);
\draw[thin,gray] (a\i) -- (ca\i)  (b\i) -- (cb\i); } 
}
\end{scope}
\begin{scope}
{   
\path (0,-1.3)++(40:0.3*0.7)++(-50:0.3*0.7) coordinate (origin);  
\draw[wei2]  (origin)   circle (2pt);
\clip(-2.2,-1.95)--(-2.2,8.3)--(8.7,8.3)--(8.7,-1.95)--(-2.2,-1.95); 
\draw[wei2] (origin)   --++(-90:4cm); 
\draw[thick] (origin) 
--++(130:3*0.7)
--++(40:2*0.7)
--++(-50:1*0.7)	
--++(40:11*0.7)
--++(-50:1*0.7)	
--++(40:1*0.7)
 --++(-50:1*0.7)
 --++(-50-90:14*0.7);
 \clip (origin)  
--++(130:3*0.7)
--++(40:2*0.7)
--++(-50:1*0.7)	
--++(40:11*0.7)
--++(-50:1*0.7)	
--++(40:1*0.7)
 --++(-50:1*0.7)
 --++(-50-90:13*0.7);
  \path (40:1cm) coordinate (A1);
\path (40:2cm) coordinate (A2);
\path (40:3cm) coordinate (A3);
\path (40:4cm) coordinate (A4);
\path (130:1cm) coordinate (B1);
\path (130:2cm) coordinate (B2);
\path (130:3cm) coordinate (B3);
\path (130:4cm) coordinate (B4);
\path (A1) ++(130:3cm) coordinate (C1);
\path (A2) ++(130:2cm) coordinate (C2);
\path (A3) ++(130:1cm) coordinate (C3);
\foreach \i in {1,...,19}
{
\path (origin)++(40:0.7*\i cm)  coordinate (a\i);
\path (origin)++(130:0.7*\i cm)  coordinate (b\i);
\path (a\i)++(130:12cm) coordinate (ca\i);
\path (b\i)++(40:12cm) coordinate (cb\i);
\draw[thin,gray] (a\i) -- (ca\i)  (b\i) -- (cb\i); } 
}  \end{scope}
\end{tikzpicture}\!\!\!\!
\begin{tikzpicture}[scale=0.4]
\begin{scope}
\draw[thick](2,-1.95)--(-1.5,-1.95); 
{     \path (0,0.8) coordinate (origin);     \draw[wei2] (origin)   circle (2pt);
 
\clip(-2.2,-1.95)--(-2.2,8.3)--(8.7,8.3)--(8.7,-1.95)--(-2.2,-1.95); 
\draw[wei2] (origin)   --(0,-2); 
%
%
%
%
\draw[thick] (origin) 
--++(130:3*0.7)
--++(40:1*0.7)
--++(-50:1*0.7)	
--++(40:9*0.7) 
--++(-50:1*0.7)	
--++(40:1*0.7)
  --++(-50:1*0.7) 
 --++(-50-90:11*0.7);
 \clip (origin) 
--++(130:3*0.7)
--++(40:1*0.7)
--++(-50:1*0.7)	
--++(40:9*0.7) 
--++(-50:1*0.7)	
--++(40:1*0.7)
  --++(-50:1*0.7) 
 --++(-50-90:11*0.7);

\path  (origin)--++(40:0.35)--++(130:0.35)  coordinate (111);  
\foreach \i in {1,...,100}
{
\path (origin)++(40:0.7*\i cm)  coordinate (a\i);
\path (origin)++(130:0.7*\i cm)  coordinate (b\i);
\path (a\i)++(130:20cm) coordinate (ca\i);
\path (b\i)++(40:20cm) coordinate (cb\i);
\draw[thin,gray] (a\i) -- (ca\i)  (b\i) -- (cb\i); } 
}
\end{scope}
\begin{scope}
{   
\path (0,-1.3)++(40:0.3*0.7)++(-50:0.3*0.7) coordinate (origin);  
\draw[wei2]  (origin)   circle (2pt);
\clip(-2.2,-1.95)--(-2.2,8.3)--(8.7,8.3)--(8.7,-1.95)--(-2.2,-1.95); 
\draw[wei2] (origin)   --++(-90:4cm);

\draw[thick] (origin) 
 
 --++(130:3*0.7)
--++(40:2*0.7)
--++(-50:1*0.7)	
--++(40:4*0.7)
--++(-50:1*0.7)	
--++(40:8*0.7)
 --++(-50:1*0.7)
 --++(-50-90:14*0.7);

 \clip (origin)

 --++(130:3*0.7)
--++(40:2*0.7)
--++(-50:1*0.7)	
--++(40:4*0.7)
--++(-50:1*0.7)	
--++(40:8*0.7)
 --++(-50:1*0.7)
 --++(-50-90:14*0.7);

\foreach \i in {1,...,100}
{
\path (origin)++(40:0.7*\i cm)  coordinate (a\i);
\path (origin)++(130:0.7*\i cm)  coordinate (b\i);
\path (a\i)++(130:12cm) coordinate (ca\i);
\path (b\i)++(40:12cm) coordinate (cb\i);
\draw[thin,gray] (a\i) -- (ca\i)  (b\i) -- (cb\i); } 
}  \end{scope}
\end{tikzpicture}\!\!\!\!
\begin{tikzpicture}[scale=0.4]
\begin{scope}
\draw[thick](2,-1.95)--(-1.5,-1.95); 
{     \path (0,0.8) coordinate (origin);     \draw[wei2] (origin)   circle (2pt);
 
\clip(-2.2,-1.95)--(-2.2,8.3)--(8.7,8.3)--(8.7,-1.95)--(-2.2,-1.95); 
\draw[wei2] (origin)   --(0,-2); 
%
%
%
%

\draw[thick] (origin) 
--++(130:3*0.7)
--++(40:1*0.7)
--++(-50:1*0.7)	
--++(40:8*0.7)
--++(-50:1*0.7)	
--++(40:2*0.7)
   --++(-50:1*0.7) 
 --++(-50-90:11*0.7);
 \clip (origin) 
--++(130:3*0.7)
--++(40:1*0.7)
--++(-50:1*0.7)	
--++(40:8*0.7)
--++(-50:1*0.7)	
--++(40:2*0.7)
   --++(-50:1*0.7) 
 --++(-50-90:11*0.7);

\path  (origin)--++(40:0.35)--++(130:0.35)  coordinate (111);  
\foreach \i in {1,...,100}
{
\path (origin)++(40:0.7*\i cm)  coordinate (a\i);
\path (origin)++(130:0.7*\i cm)  coordinate (b\i);
\path (a\i)++(130:20cm) coordinate (ca\i);
\path (b\i)++(40:20cm) coordinate (cb\i);
\draw[thin,gray] (a\i) -- (ca\i)  (b\i) -- (cb\i); } 
}
\end{scope}
\begin{scope}
{   
\path (0,-1.3)++(40:0.3*0.7)++(-50:0.3*0.7) coordinate (origin);  
\draw[wei2]  (origin)   circle (2pt);
\clip(-2.2,-1.95)--(-2.2,8.3)--(8.7,8.3)--(8.7,-1.95)--(-2.2,-1.95); 
\draw[wei2] (origin)   --++(-90:4cm);

\draw[thick] (origin)  

 --++(130:3*0.7)
--++(40:2*0.7)
--++(-50:1*0.7)	
--++(40:5*0.7)
--++(-50:1*0.7)	
--++(40:7*0.7)
 --++(-50:1*0.7)
 --++(-50-90:14*0.7);

 \clip (origin)

 --++(130:3*0.7)
--++(40:2*0.7)
--++(-50:1*0.7)	
--++(40:5*0.7)
--++(-50:1*0.7)	
--++(40:7*0.7)
 --++(-50:1*0.7)
 --++(-50-90:14*0.7);

%
%
%
%
%
\foreach \i in {1,...,100}
{
\path (origin)++(40:0.7*\i cm)  coordinate (a\i);
\path (origin)++(130:0.7*\i cm)  coordinate (b\i);
\path (a\i)++(130:12cm) coordinate (ca\i);
\path (b\i)++(40:12cm) coordinate (cb\i);
\draw[thin,gray] (a\i) -- (ca\i)  (b\i) -- (cb\i); } 
}  \end{scope}
\end{tikzpicture}
\]
\caption{The set of bi-partitions $\Gamma_2$ as in \cref{beloweg}.  
}
\label{below}
\end{figure}

\begin{eg}\label{beloweg}
We let $\kappa=(0,3)$ and $e=7$.  
Given $\gamma=(  (3,2,1^9),(3^2,2^{12})) \in \mathbb{E}^\circledcirc_{h,\ell}$, we have that   
$$
\Gamma_2=
\{
   ((3,2,1^9), (3^2,2^{12})) 
,
(  (3,2^2,1^8),(3^2,2^{11},1))  
,
(  (3,2^{9},1),		(3^2,2^{4},1^8)) 
,
((3,2^8,1^2), (3^2,2^{5},1^7) ) 
 \}.  
$$
We picture these bi-partitions as in \cref{below}; we depict the final three by $\alpha, \beta$ and $\delta$ respectively. 
Notice that the third column (of any given component)  is much shorter
than the first column (of any given component). 

These multipartitions belong to a single line in the 6-dimensional space $\mathbb{E}_{3,2}^\circledcirc$.  Projecting down onto this line, we obtain the set of 4 points depicted in \cref{BOO!}.  We have illustrated the direction of the two affine vertex homomorphisms with arrows. 
Notice that they correspond to a reflection through a wall of the alcove containing the origin.

 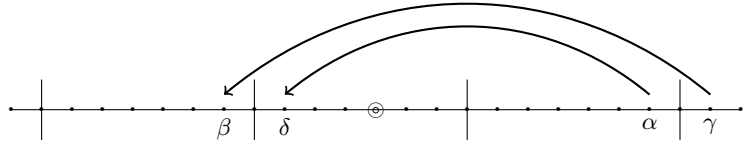
\begin{figure}[ht!]
$$\scalefont{0.8}
\begin{tikzpicture} 

%
%
\begin{scope}

  \draw  (origin)  node {$\circledcirc$}; 
     \foreach \i in {0,...,12}
{ \path (origin)++(0:0.4*\i cm)  coordinate (a\i);
 \path (origin)++(180:0.4*\i cm)  coordinate (b\i);
 }
   \foreach \i in {1,...,12}
   {\draw[thick](a\i) node {$ \boldsymbol{\cdot}$};
 \draw(b\i) node {$\boldsymbol{\cdot}$};
 }

 \draw(b12)--(a12);
 \draw(a3) --++(90:0.4) (a3) --++(-90:0.4);
  \draw(a10) --++(90:0.4) (a10) --++(-90:0.4);
   \draw(b4) --++(90:0.4) (b4) --++(-90:0.4);
      \draw(b11) --++(90:0.4) (b11) --++(-90:0.4);
  \draw(a11)      node[below] {$\gamma$};
   \draw(a9)      node[below] {$\alpha$};
   \draw(b3)      node[below] {$\delta$};
   \draw(b5)      node[below] {$\beta$}; 
      \foreach \i in {0,...,12}
{ \path (a\i)++(90:0.2cm)  coordinate (c\i);
\path (b\i)++(90:0.2cm)  coordinate (d\i);;
 }

     \foreach \i in {0,...,12}
{ \path (a\i)++(-90:0.4cm)  coordinate (e\i);
\path (b\i)++(-90:0.4cm)  coordinate (f\i);;
 }
 
  \draw[thick, ->] (c11) to[in=40,out=140] (d5);
  \draw[thick, ->] (c9) to[in=40,out=140] (d3);


\end{scope}
\end{tikzpicture}$$
\caption{Some simple  homomorphisms from affine vertices.  This is the projection of $\mathbb{E}_{3,2}^\circledcirc$ onto the line 
$\RR\{\varepsilon_2,\varepsilon_5\} / (\varepsilon_2+\varepsilon_5)$.  
}
\label{BOO!}
\end{figure}
Finally, we note that this example can easily be  generalised
 to an infinite sequence (as $n\to \infty$) of homomorphisms whose piecewise composition is non-zero (by verifying the conditions of \cref{howtobuildshit3} for these paths).  
The first four homomorphisms of such a chain are depicted in   \cref{BOO!!}.   
This makes clear the fact that these homomorphisms are from points ``as far apart as possible'' in the geometry.

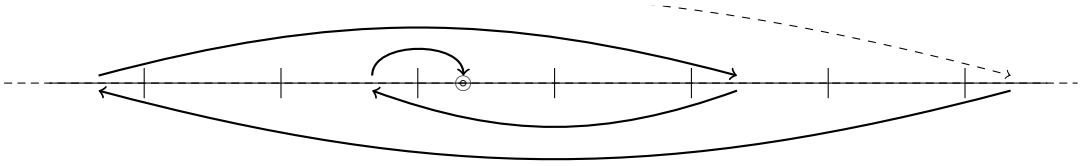
\begin{figure}[ht!]
$$\scalefont{0.8}
\begin{tikzpicture} 
   \clip(1.8,-1.6)--(26*0.6+0.4,-1.6)--(26*0.6+0.4,0.5)--(1.8,0.5)--(1.8,-1.5);
 \draw[densely dashed](origin)--++(0:20);
   \clip(2,-1.6)--(26*0.6,-1.6)--(26*0.6,0.5)--(6,0.5)--(2.7,0)--(2,-1.6);
   \foreach \i in {-1,0,...,26}
{ \path (origin)++(0:0.6*\i cm)  coordinate (a\i);}
           \foreach \i in {0,3,6,9,...,24}
           { \path (origin)++(0:0.6*\i cm)  coordinate (b\i); }
 \draw(a2)--(a26);
  \foreach \i in {3,6,...,24}
   \draw(b\i) --++(90:0.2) (b\i) --++(-90:0.2) ;
 
   \draw  (a13)  node {$\circledcirc$};

    \foreach \i in {-1,0,...,25}
{ \path (origin)++(90:0.1cm)++(0:0.6*\i cm)  coordinate (x\i);}

   \foreach \i in {-1,0,...,25}
{ \path (origin)++(-90:0.1cm)++(0:0.6*\i cm)  coordinate (y\i);}

   \draw[thick, ->] (x11) to[in=90,out=90] (x13);

   \draw[thick, ->] (y19) to[in=-20,out=-160] (y11);

   \draw[thick, <-] (x19) to[in=15,out=165] (x5);
    \draw[thick, <-] (y5) to[in=-165,out=-15] (y25);

    \draw[   dashed, ->] (origin) to[in=165,out=15] (x25);

 \end{tikzpicture}$$
 
 \caption{A simple infinite  chain  of homomorphisms arising from   an affine vertex. 
}
\label{BOO!!}
\end{figure}
\end{eg}

 \begin{eg}
 An example of affine vertices in a more complicated geometry is given in \cref{bigeg}.      
 \end{eg}
 
 \begin{rmk}
  
 For the  blob algebra, $\Q_{1,2,n}(\kappa)$, over $\mathbb{C}$, 
the   homomorphisms of  this subsection 
were first constructed in \cite[(9.1) Theorem]{MW00}.  
Moreover, these homomorphisms in conjunction with those of \cref{blob} provide all homomorphisms between Weyl modules for the blob algebra over $\mathbb{C}$.
Both families of homomorphisms continue to play an important role (and to be defined) over fields of positive characteristic \cite[Section 6]{MR1995129}.  
 \end{rmk}

\section{Decomposition numbers over fields of large and infinite characteristic}\label{sec:8}

In this section we shall give a combinatorial method for calculating
certain parabolic affine Kazhdan--Lusztig polynomials in terms of the
degrees of our paths in our alcove geometry.  
Over $\mathbb{C}$, this allows us to compute the decomposition matrix of these algebras (thus generalising earlier work of \cite{bcs15}).  However, our main interest (as in the previous chapters) will be in what can be said over fields of positive characteristic.  
 We
begin by reviewing the necessary Kazhdan--Lusztig theory, in the spirit
(and notation) of \cite{Soergel}.
 
Let $({\mathcal W},{\mathcal S})$ be a Coxeter system, and ${\mathcal
  S}_f\subset {\mathcal S}$. Then there is an associated parabolic
subgroup ${\mathcal W}_f< {\mathcal W}$ generated by the set
${\mathcal S}_f$. Deodhar \cite{DD1} showed how to associate to the
pair $({\mathcal W}_f,{\mathcal W})$ certain parabolic Kazhdan-Lusztig
polynomials In particular, let ${\mathcal W}^f$ be a set of minimal
length representatives of right cosets of ${\mathcal W}_f$ in
${\mathcal W}$ (with respect to the usual Coxeter length
function). Then for any pair of elements $x,y\in{\mathcal W}^f$, there
is an associated (inverse) parabolic Kazhdan-Lusztig polynomial
$n_{x,y}\in\ZZ[t]$. We are interested in the special case where
${\mathcal W}$ is the 
affine Weyl group of type $\widehat{A}_{h\ell -1}$ (or type $A_{h\ell -1}$ if
$e=\infty$) and the parabolic is the Weyl group of type 
$$\underbrace{A_{h-1}\times A_{h-1}\times\cdots\times A_{h-1}}_{ \ell \text{ copies}}$$
 As we have already seen, we can define a geometry associated with this
choice of ${\mathcal W}$ on ${\mathbb E}_{h,\ell}$. We call the connected
components of the complement of the union of the various reflecting
hyperplanes the alcoves of this geometry. As described in
\cite[Section 4]{Soergel} there is a natural bijection between
${\mathcal W}$ and the set of alcoves ${\mathcal A}$.

Soergel goes on to consider a certain set ${\mathcal A}^+\subset
{\mathcal A}$ which is precisely the set of alcoves lying in some
fixed choice of fundamental domain for the finite Weyl group of type
$A_{h\ell -1}$ inside $\widehat{A}_{h\ell -1}$.  
 Then there is an induced
bijection from the set of right cosets of this finite Weyl group
inside the affine Weyl group to the alcoves in ${\mathcal
  A}^+$. However, we can instead consider the finite parabolic
${\mathcal W}_f$ inside ${\mathcal W}$ as defining our choice of
fundamental domain ${\mathcal A}^+$, and then identify ${\mathcal
  W}^f$ with ${\mathcal A}^+$. This choice of ${\mathcal A}^+$
consists precisely of those alcoves contained in the dominant Weyl
chamber ${\mathbb E}_{h,\ell}^\circledcirc$  defined in   \cref{1.1}.

Choosing a fixed weight $\lambda$ in an alcove, there is for each
alcove in ${\mathcal A}$ an associated weight $w\cdot\lambda$ in
that alcove, and in this way we can identify pairs of alcoves with
pairs of weights in a given ${\mathcal W}$-orbit. Via these various
identifications we can now define for any pair of weights $\lambda$
and $\mu$ in the same ${\mathcal W}$-orbit a parabolic Kazhdan-Lusztig
polynomial $n_{\lambda\mu}\in\ZZ[t]$.

In \cite{gw01}, Goodman and Wenzl have shown how to associate a
parabolic Kazhdan-Lusztig polynomial to any pair of dominant weights,
not just those lying in the interior of an alcove. Further, they give
an algorithm for determining these polynomials in terms of certain
piecewise linear paths in the geometry. As in \cite{Soergel}, their
results are all stated for the case of a finite Weyl group considered
as a parabolic in the associated affine Weyl group, but by making the
same modifications as described above it is straightforward to see
that this procedure extends to pairs of weights in the dominant Weyl
chamber ${\mathbb E}_{h,\ell}^\circledcirc$ for our choice of
${\mathcal W}$ and ${\mathcal W_f}$.  

We can now re-express the main result from \cite{gw01} in terms of the
paths which we have defined in \cref{1.2}.  

\begin{thm}\label{KLpoly}
Let $\lambda,\mu\in {\mathbb E}_{h,\ell}^\circledcirc$ with
$\lambda\in{\mathcal W}\cdot \mu$ and  let 
 $\stt^\mu$ be any admissible path of degree $n$ from $\circledcirc $
to $\mu$.   The associated polynomial 
\begin{equation}\label{adefn}
 N_{\lambda \stt^\mu} = \sum_{\sts \in \Path^+_n(\lambda,\stt^\mu)} t^{\deg(\sts)} 
\end{equation}
is an element of $\ZZ_{\geq 0}[t,t^{-1}].$  
  We can rewrite the polynomials $ N_{\lambda \stt^\mu}$ in the form 
$$
 N_{\lambda \stt^\mu} = \sum_{\lambda \uparrow \nu \uparrow \mu}
 n_{\la \nu} 
 \underline{N}_{  \nu \mu}
 $$
 for some $  n_{\la \nu} 
  \in \ZZ_{\geq 0}[t]$, and 
  $  \underline{N}_{  \nu \mu}\in \ZZ_{\geq 0}[t+t^{-1}]$;  
  this expression is unique.  
Further, we have 
 $n_{\la\nu} $ is 
the (inverse) parabolic affine Kazhan--Lusztig polynomial
 associated to the pair $({\mathcal W}_f,{\mathcal W})$.
 \end{thm}

\begin{proof}
All that remains to be verified is that our paths are examples of the
piecewise linear paths considered in \cite{gw01}, and that the
associated degree functions for the two definitions coincide. That our
paths are examples of their piecewise linear paths is obvious.
  
It  remains to show that \cite[equation (2.4)]{gw01} agrees with
\cref{Soergeldegreee}. Recall from \cref{Soergeldegreee}, that the
 degree of a path $\sts(k)$ is obtained from that of the path $\sts(k-1)$
by counting  hyperplanes with alternating signs; it is simple to see
rephrase this alternating sign in terms of the length function used in
\cite{gw01}. Setting $\Lambda$ in \cite[equation (2.4)]{gw01}
to be equal to the segment  $(\sts(k-1), \sts(k))$ in the path $\sts$,
we obtain the desired equality.   
\end{proof}
 
\subsection{Decomposition numbers over the complex field} 

 We now use the results above to determine
the decomposition numbers for our algebras in characteristic zero.  
 
 \begin{prop}\label{troll}
If $R=\mathbb{C}$, then \begin{itemize}
\item[$(i)$]   $ \dim_t{(\Delta_\mu(\lambda))} = N_{\lambda\stt^\mu} \in \ZZ_{\geq 0}[t,t^{-1}]$ and $ \dim_t{(L_\mu(\lambda))}  \in \ZZ_{\geq 0}[t+t^{-1}]$;
\item[$(ii)$] if $ \dim_t{(\Delta_\mu(\lambda))}=0$, then $d_{\lambda\mu}(t)=0$;
\item[$(iii)$] we have $ \dim_t{(\Delta_\mu(\mu))} =  \dim_t{(L_\mu(\mu))} =1$;
\item[$(iv)$] if $\Path(\lambda,\stt^\mu)=\emptyset$, then $\dim_t{(\Delta_\mu(\lambda))}=0$;
\item[$(v)$] if $\Path(\lambda,\stt^\mu)=\emptyset$, then $\dim_t{(L_\mu(\lambda))}=0$;
\item[$(vi)$] we have that
\[
\dim_t{(\Delta_\mu(\lambda)})= \;
\sum_
{  \mathclap{\begin{subarray}c \nu \neq \mu \\ 
\la \uparrow \nu \uparrow \mu  
   \end{subarray} }} \; d_{\lambda\nu}(t)\dim_t{(L_\mu(\nu))}+d_{ \lambda\mu}(t).
  \]
\end{itemize}
   \end{prop}
\begin{proof}
Part $(i)$ is clear by Proposition \ref{humathasprop} and \cref{adefn} $(iii)$ is a restatement of the condition that $\stt^\mu$ is the only path  in $\Path(\mu,\stt^\mu)$.  
A necessary condition for $\dim_t{(\Hom (P(\mu),\Delta(\lambda))}\neq 0$ is that $\Delta_\mu(\lambda)\neq 0$, therefore $(ii)$ follows.
Part $(iv)$ is by definition, and part $(v)$ follows from the cellular structure.
Finally, $(vi)$ follows from $(i), (iii), (v)$ and \cref{step1}.  
\end{proof}

    \begin{thm}\label{main:theorem:general}
Let $R=\mathbb{C}$.     The graded decomposition numbers of 
$A_h(n,\theta,\kappa)$ are given by  
\begin{align*} 
d_{\lambda\mu}	(t)= n_{   \lambda\mu }.   \end{align*}
       \end{thm}
   
\begin{proof}
By Proposition \ref{troll} $(ii)$, we may assume $\Path(\lambda,\stt^\mu)\neq \emptyset$.  We now calculate $d_{ \lambda\mu}(t)$ and $\dim_t{(L_\mu(\lambda))}$ by induction on the length ordering.  Induction begins when $\ell(\mu)-\ell(\lambda)=0$, hence $\mu=\lambda$, and we have $d_{\mu\mu}(t)=1$ by Proposition \ref{troll} $(iii)$ and $\dim_t{(L_\mu(\mu))}=1$.

Let $\ell(\mu)-\ell(\lambda)\geq 1$.  By induction on the length ordering, we know $d_{\lambda\nu}(t)$   and $\dim_t{(L_\mu(\nu))}$ for points 
$\nu\in \Lambda_n$ such that  $\lambda \uparrow \nu \uparrow \mu$.
 By Proposition \ref{troll} $(vi)$ we have
\[
\dim_t{(L_\mu(\lambda))} + d_{ \lambda\mu} (t)
 =
\dim_t{(\Delta_\mu(\lambda))} 
- \;
\sum_
{  \mathclap{\begin{subarray}c 
\nu \neq \mu, \,
\nu \neq \lambda \\ 
\la \uparrow \nu \uparrow \mu
\end{subarray} }} \;d_{\lambda\nu}(t)\dim_t{(L_\mu(\nu))}.
\]
By induction and Proposition \ref{troll} $(i)$, we have that 
\[
\dim_t{(L_\mu(\lambda))} + d_{ \lambda\mu} (t)
 =
N_{\lambda\stt^\mu} - \;
\sum_
{  \mathclap{\begin{subarray}c 
\nu \neq \mu, \, 
\nu \neq \lambda \\ 
\la \uparrow \nu \uparrow \mu  \end{subarray} }} \; n_{ \lambda\nu}\underline{N}_{\nu\mu}.
\] Recall that 
$\underline{N}_{\nu\mu}=\dim_t{(L_\mu(\nu))} \in \mathbb{Z}_{\geq 0} [t+t^{-1}]$  and $n_{ \lambda\nu}=d_{  \lambda\mu}(t)\in t\ZZ_{\geq 0}[t]$. Therefore there is a unique  solution to the equality (see \cite[Section 4.1: Basic Algorithm]{KN10}).  By \cref{KLpoly} this solution is given  by 
$ d_{\lambda\mu}(t)=n_{  \lambda\mu}$.  
\end{proof}

 \subsection{Decomposition numbers over fields of large characteristic} \label{onject}
As we have made clear throughout the paper, our principal interest is in fields of positive characteristic.  
Indeed, we informally defined our   algebra  $\Q_{\ell,h,n}(\kappa)$  as the {\em largest quotient of $H_n(\kappa)$ for which a generalised Lusztig conjecture could possibly hold} for fields of positive characteristic.

\begin{defn}
We say that  $\lambda \in \mathbb{E}_{h,\ell}^+$ belongs to {\sf the first $ep$-alcove} if $\ell_\alpha(\lambda) <p$ for all $\alpha \in \Phi$.
\end{defn}

    \begin{conj} \label{conj1} 
Let   
$e>h\ell$,   $\kappa\in I^\ell$ be an $h$-admissible  multicharge, and   $\Bbbk$ be  field of  characteristic $p \gg  h\ell$.  
 The decomposition numbers  of   $A_h(n,\theta,\kappa)$    are    given by 
$$d_{\lambda\mu}(t) = n_{\lambda\mu} (t)  $$ 
for   $\lambda,\mu\in \mptn \ell n(h)$  in 
 the first $ep$-alcove   and $n_{\lambda\mu} (t)$ 
  the associated  affine  parabolic Kazhdan--Lusztig polynomial    
of type ${A}_{ h-1} \times {A}_{ h-1} \times \dots \times {A}_{ h-1} \subseteq 
\widehat{A}_{\ell h-1}.$

       \end{conj}

For $e=\infty$ we make a stronger conjecture concerning the algebras ${A(n,\theta,\kappa)}$ in  which the statement of \cref{conj1}  simplifies in two ways.    
The simpler geometry controlling the $e=\infty$ case
 means that there is no notion of a ``first $ep$-alcove''.   
 Secondly, 
   the subalgebra isomorphic to the Hecke algebra of the 
  symmetric group is semisimple; 
 therefore we need not restrict  the characteristic of the field according to the   number of columns   of $\lambda, \mu \in \mptn \ell n$.

 \begin{conj} \label{conj2}
 Let  
$e>n$  
and  $\Bbbk$ be  field of  characteristic $p \gg   \ell$.  Suppose that  $\kappa\in I^\ell$ has no repeated entries.     
The decomposition numbers of $ {A(n,\theta,\kappa)}$   are given by 
$$d^{A(n,\theta,\kappa)}_{\lambda\mu}(t)=n_{\lambda\mu} (t)  $$ 
for $\lambda,\mu \in \mptn \ell n$ where $n_{\lambda\mu} (t) $ is the associated    Kazhdan--Lusztig polynomials of type $ A_{n-1}\times A_{n-1}\times  \dots \times  A_{n-1} \subseteq {A}_{n\ell-1}$. 
  
       \end{conj}
       
\begin{rmk}
Notice that the Kazhdan--Lusztig polynomials in the latter conjecture are non-affine
(see \cref{see}) and that the parabolic is determined by products of $A_{n-1}$ (rather than $A_{h-1}$) and that it covers {\em all} decomposition numbers of the diagrammatic Cherednik algebra.  Finally, recall that our assumption on $\kappa\in I^\ell$ is less strict than in \cref{conj2} is weaker than in \cref{conj1}.  
\end{rmk}

\begin{rmk}
For $e=\infty$, the unique graded decomposition matrix 
of $H_n(\kappa)$ appears as the (in general proper) submatrix of that of $A(n,\theta,\kappa)$ whose columns are labelled by so-called FLOTW multipartitions.  
\end{rmk}       
       
   \cref{conj1} is the first conjectural attempt
    to describe an infinite family of decomposition numbers of 
    (quiver) Hecke algebras for $e\in \mathbb{Z}_{>0}$.  
A more optimistic version of  \cref{conj2} appeared in \cite[Conjecture 7.3]{MR2777040} (without the assumptions on $\kappa\in I^\ell$ or $p\gg \ell$) but was later disproven in \cite[Section 4.2]{MR3163410}.  
This counterexample is for $\ell=5$ and $p<\ell$ and does not seem surprising from the point-of-view of this paper (indeed we posed 
 \cref{conj2} in its current form 
before being informed about the results of \cite{MR2777040,MR3163410} by Liron Speyer).

\begin{eg}
In \cref{maximalparabolic}, we saw that the above conjecture holds for any $\lambda,\mu$ in a  maximal finite parabolic   orbit with $\Bbbk$ arbitrary.  
\end{eg}

\begin{eg}
Let  $\ell=2$ and $h=1$ and $e >2$ and  $\Bbbk$ be  arbitrary. 
In this case $\Q_{h,\ell,n}(\kappa)$ is isomorphic to the blob algebra.  
 The conjecture follows from \cite[Section 8]{MR1995129}.  
\end{eg}

\begin{eg}
If $\ell=1$, then $\Q_{h,\ell,n}(\kappa)$ is isomorphic to the generalised Temperley--Lieb algebra of H\"arterich \cite[Section 1]{MR1680384}.  The result therefore follows by       \cite[Theorem 1.9]{w16}.  
\end{eg}

\begin{eg}
If $\ell=2$ and $e>n$, then $A (n,\theta,\kappa)$ is positively graded.  Therefore any simple module is 1-dimensional.  Therefore the decomposition numbers of $A_h(n,\theta,\kappa)$ are independent of the characteristic of the   field and the result follows.  An identical proof of this is given in \cite[Appendix B]{MR3356809} (in fact, the quiver Schur algebra  of Hu and Mathas is isomorphic to our algebra $A (n,\theta,\kappa)$ in this case) and an  
 earlier proof of this result  for $e=\infty$ was given  in \cite{MR2781018}.  
 \end{eg}

\begin{eg}
 Consider  $A (n,\theta,\kappa)$  with 
 $\ell=3$ and $e=\infty$.  This algebra is non-negatively graded.  
   For any $\lambda,\mu \in \mptn 3 n$ there is at most one path, $\sts$ say, of degree zero in $\Path(\lambda,\stt^\mu)$.  
   For example if $\kappa=(0,1,2)$ then $\Path((2),\varnothing, (2)), (1),(1),(2))$ has two  elements; one  of degree 0 and one of degree 2. 
Therefore it suffices to check that $C_{\sts}^\ast C_{\sts}= C_{\stt^\mu_\lambda}$ in order to conclude that $L(\lambda)$ is 2-dimensional   with basis $\{C_{\sts},C_{\stt^\lambda}\}$.  Having done so, one can conclude that the decomposition numbers are independent of the characteristic of the field. 
 We leave this as an exercise for the reader.  
 \end{eg}

\section{Alcove geometries in $\mathbb{R}^2$}\label{sec:9}

We now discuss the algebras $\Q_{\ell,h,n}(\kappa)$ which are controlled by geometries which can be visualised within the plane  $\mathbb{R}^2$.   
We   encounter all the usual Lie theoretic geometries, as well as new geometries which do not arise in the representation theory of (affine) Lie algebras and related objects.   

\medskip
\noindent There are a total of five distinct affine geometries that we can 
picture via an embedding  into $\mathbb{R}^2$.  
\medskip

The first 2 of these are of type 
$A_1\subseteq \widehat{A}_1$   for $h=2$ and $\ell=1$ and 
of type  $\widehat{A}_1$ for $\ell=2$ and $h=1$. 
In the former (respectively latter) case,  we  visualise  
 $\mathbb{E}_{h,\ell}^\circledcirc$ as half of the real line
  (respectively the whole real line).  These geometries already appear in classical Lie theory.    The former (respectively latter) controls the representation theory of the Lie algebra ${\mathfrak{sl}}_2$ (respectively the Kac--Moody algebra  $\widehat{\mathfrak{sl}}_2$).  
    
\medskip
\noindent
There are three further affine geometries which can be visualised in $\mathbb{R}^2$ (depicted in \cref{geometries}).  
\medskip

The first is that of type ${A}_2\subseteq \widehat{{A}}_2$ which controls the representation theory of $\Q_{1,3,n}(\kappa)$.  The algebra $\Q_{1,3,n}(\kappa)$ is the Ringel dual of the image of   $U_q(\mathfrak{sl}_3)$ in ${\rm End}((\Bbbk^3)^{\otimes n})$. Obviously, this is the same as the geometry controls the representation theory of the Lie algebra $\mathfrak{sl}_3$.  
 The dominant region  $\mathbb{E}_{h,\ell}^\circledcirc$ is one sixth of the plane $\mathbb{R}^2$.  

The second is that of type $ \widehat{{A}}_2$ which controls the representation theory of $\Q_{3,1,n}(\kappa)$.  This is the same as the geometry controls the representation theory of the Kac--Moody algebra $\widehat{\mathfrak{sl}}_3$.  
 The dominant region  $\mathbb{E}_{h,\ell}^\circledcirc$ is   the entirety of the plane $\mathbb{R}^2$.  

\!\!\!\!\!\!\!\!
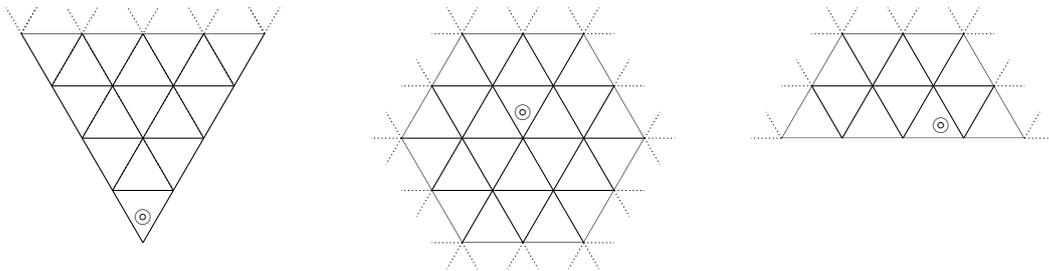
\begin{figure}[ht!]
$$\scalefont{0.8}
\begin{tikzpicture}[scale=2]

\clip(4.5,1)--(16,1)--(16,-1)--(4.5,-1);
            
\begin{scope}

 \path  (6,-.75)  coordinate (origin); 
  \draw  (origin)++(60:0.2)++(180:0.1)  node {$\circledcirc$}; 
  \path  (origin)++(60:1.6cm) coordinate (a1);
  \draw (origin)--(a1);
  \draw  (origin)--++(120:1.6cm);
    \path  (4,0) coordinate (cc);
      \foreach \i in {1,...,4}
  {
    \path (origin)++(60:0.4*\i cm)  coordinate (a\i);
    \path (origin)++(120:0.4*\i cm)  coordinate (b\i);
    \path (a\i)++(120:2cm) coordinate (ca\i);
    \path (b\i)++(60:2cm) coordinate (cb\i);
}
 \draw[densely dotted]  (a4)--++(60:0.2);
  \draw[densely dotted]  (a4)--++(120:0.2);
   \draw[densely dotted]  (b4)--++(60:0.2);
  \draw[densely dotted]  (b4)--++(120:0.2);

  \path (b4)++(120:0.2)  coordinate (b5);
    \path (a4)++(60:0.2)  coordinate (a5);  
 \clip (origin)--(a5)--(b5)--(origin);      
 \foreach \i in {1,...,4}    
   { \draw[densely dotted]  (a\i) -- (ca\i)  (b\i) -- (cb\i);
  }

 \clip (origin)--(a4)--(b4)--(origin);    
  \foreach \i in {1,...,4}    
   { \draw  (a\i) -- (ca\i)  (b\i) -- (cb\i);
    \draw  (a\i) -- (b\i)  ; } 
 
   \end{scope}


\begin{scope}

 \path  (8.5,-.75)--++(-60:0.4)--++(-120:0.4)  coordinate (origin); 
 
    \path  (4,-2.5) coordinate (cc);
      \foreach \i in {0,...,12}
  {
    \path (origin)++(60:0.4*\i cm)  coordinate (a\i);
    \path (origin)++(120:0.4*\i cm)  coordinate (b\i);
    \path (a\i)++(120:2cm) coordinate (ca\i);
    \path (b\i)++(60:2cm) coordinate (cb\i);
}
%
   \path (b4)++(60:0.8)  coordinate (cc1);
     \path (cc1)++(-120:0.4)  coordinate (ee1);  
     
     \path (cc1)++(00:0.4)  coordinate (ee2);  
     \path (cc1)++(00:0.8)  coordinate (cc2);  
          \path (cc2)++(-60:0.4)  coordinate (ee3);  
     \path (cc1)++(-60:0.8)  coordinate (dd);  
            \draw  (dd)  ++(60:0.2)++(180:0.1)  node {$\circledcirc$}; 


     \path (dd)++(-60:0.4)++(-120:0.4)  coordinate (XXXX);  
     
\draw[densely dotted] (XXXX)--++(-120:0.2);
\draw[densely dotted]  (XXXX)--++(-60:0.2);
     
\draw[densely dotted] (a4)--++(60:0.2);
\draw[densely dotted]  (a4)--++(-60:0.2);
\draw[densely dotted]  (a4)--++(0:0.2);

\draw[densely dotted]  (a3)--++(-60:0.2);
\draw[densely dotted]  (a3)--++(0:0.2);

\draw[densely dotted] (a2)--++(-120:0.2);
\draw[densely dotted]  (a2)--++(-60:0.2);
\draw[densely dotted]  (a2)--++(0:0.2);

\draw[densely dotted] (b2)--++(-120:0.2);
\draw[densely dotted]  (b2)--++(-60:0.2);
\draw[densely dotted]  (b2)--++(180:0.2);

\draw[densely dotted] (b3)--++(-120:0.2);
\draw[densely dotted]  (b3)--++(180:0.2);

\draw[densely dotted] (b4)--++(120:0.2);
\draw[densely dotted]  (b4)--++(-120:0.2);
\draw[densely dotted]  (b4)--++(180:0.2);

\draw[densely dotted] (cc1)--++(120:0.2);
\draw[densely dotted]  (cc1)--++(60:0.2);
\draw[densely dotted]  (cc1)--++(180:0.2);

\draw[densely dotted] (cc2)--++(120:0.2);
\draw[densely dotted]  (cc2)--++(60:0.2);
\draw[densely dotted]  (cc2)--++(0:0.2);

\draw[densely dotted]  (ee1)--++(120:0.2);
\draw[densely dotted]  (ee1)--++(180:0.2);

\draw[densely dotted] (ee2)--++(120:0.2);
\draw[densely dotted]  (ee2)--++(60:0.2);

\draw[densely dotted] (ee3)--++(0:0.2);
\draw[densely dotted]  (ee3)--++(60:0.2);

\clip(a4)--(a2)--(b2)-- (b4)--(cc1)--(cc2)--(a4);
 \foreach \i in {0,...,8}    
  { \draw  (a\i) -- (ca\i)  (b\i) -- (cb\i);
     \draw  (a\i) -- (b\i)  ; } 
   
%
 
   \end{scope}

\begin{scope}
 \path  (11,-.75)--++(-60:0.4)--++(-120:0.4)  coordinate (origin); 

 
    \path  (4,-4.75) coordinate (cc);
      \foreach \i in {0,...,12}
  {
    \path (origin)++(60:0.4*\i cm)  coordinate (a\i);
    \path (origin)++(120:0.4*\i cm)  coordinate (b\i);
    \path (a\i)++(120:2cm) coordinate (ca\i);
    \path (b\i)++(60:2cm) coordinate (cb\i);
}
%
   \path (b4)++(60:0.8)  coordinate (cc1);
     \path (cc1)++(-120:0.4)  coordinate (ee1);  
     
     \path (cc1)++(00:0.4)  coordinate (ee2);  
     \path (cc1)++(00:0.8)  coordinate (cc2);  
          \path (cc2)++(-60:0.4)  coordinate (ee3);  
     \path (cc1)++(-60:0.8)  coordinate (dd);  
            \draw  (dd)  ++(60:0.1)++(0:0.2)  node {$\circledcirc$}; 

\draw[densely dotted] (a4)--++(60:0.2);
\draw[densely dotted]  (a4)--++(0:0.2);

%
%
%
%
%

\draw[densely dotted] (b4)--++(120:0.2);
\draw[densely dotted]  (b4)--++(180:0.2);

\draw[densely dotted] (cc1)--++(120:0.2);
\draw[densely dotted]  (cc1)--++(60:0.2);
\draw[densely dotted]  (cc1)--++(180:0.2);

\draw[densely dotted] (cc2)--++(120:0.2);
\draw[densely dotted]  (cc2)--++(60:0.2);
\draw[densely dotted]  (cc2)--++(0:0.2);

\draw[densely dotted]  (ee1)--++(120:0.2);
\draw[densely dotted]  (ee1)--++(180:0.2);

\draw[densely dotted] (ee2)--++(120:0.2);
\draw[densely dotted]  (ee2)--++(60:0.2);

\draw[densely dotted] (ee3)--++(0:0.2);
\draw[densely dotted]  (ee3)--++(60:0.2);

\clip(a4)-- (b4)--(cc1)--(cc2)--(a4);

 \foreach \i in {0,...,8}    
  { \draw  (a\i) -- (ca\i)  (b\i) -- (cb\i);
     \draw  (a\i) -- (b\i)  ; } 
   
%
 
   \end{scope}

\end{tikzpicture}
$$  

\!\!\!\!\!\!\!\! 
\caption{Geometries of type ${A}_2\subseteq \widehat{{A}}_2$, $ \widehat{{A}}_2$, and 
${A}_1\subseteq \widehat{{A}}_2$ respectively.  
These tile one sixth of $\mathbb{R}^2$, the whole of $\mathbb{R}^2$, and one half of $\mathbb{R}^2$, respectively.  }
\label{geometries}
   \end{figure}
Thirdly, we encounter    a geometry of type 
${A}_1\subseteq \widehat{{A}}_2$ which controls a portion of the representation theory of 
$\Q_{2,2,n}(\kappa)$; namely  the subcategory of   representations
whose simple constituents are labelled by points in the space
$$ \mathbb{E}_{h,\ell}^\circledcirc  \cap\mathbb{R}_{\geq0}\{\varepsilon_1,\varepsilon_3,\varepsilon_4\} \text{ or }
  \mathbb{E}_{h,\ell}^\circledcirc  \cap\mathbb{R}_{\geq0}\{\varepsilon_1,\varepsilon_2,\varepsilon_3\}.$$    
The corresponding  quotient algebra can be explicitly constructed explicitly as in \cref{bs16iguess}, but this is unnecessary
as all the results are identical.
 This geometry does not appear in classical Lie theory.  
In either case, the dominant region     is   one half  of the plane $\mathbb{R}^2$.

We now consider an example in this new geometry (of type ${A}_1\subseteq \widehat{{A}}_2$) 
which illustrates both the local generic behaviour of \cref{nonparabolic} and the non-local generic behaviour of \cref{maximalnonparabolic}.  
   
\begin{figure}[ht!]   \[\scalefont{0.8}\begin{tikzpicture}
 \begin{scope}
   \path  (0,0)  coordinate (origin); 

  \path (0,0)++(120:0.4*10 cm)  coordinate (step1);
  \path (step1)++(180:0.4*10)  coordinate (stepY);   
      \path (stepY)++(60:0.4*15)  coordinate (step175);      
      \path (step175)++(0:0.4*10)  coordinate (step275);         
      \path (step275)++(-60:0.4*15)  coordinate (step375);         
  \path (step1)++(0:0.4*3)  coordinate (step111);   

 \draw(step111) node[below] {\text{edge of dominant region}};

\clip    (origin)--(step1)--(stepY)--(step175)--(step275)--(step375)--(step1);  
    
       \begin{scope}

        \foreach \i in {0,...,35}
  {
    \path (origin)++(60:0.4*\i cm)  coordinate (a\i);
    \path (origin)++(120:0.4*\i cm)  coordinate (b\i);
    \path (a\i)++(120:12cm) coordinate (ca\i);
    \path (b\i)++(60:12cm) coordinate (cb\i);
}
   \path (b20)++(60:0.4*5)  coordinate (step1);
      \path (step1)++(0:0.4*5)  coordinate (step2);
      \path (step2)++(-60:0.4*5)  coordinate (step3); 
      \path (step3)++(0:0.4*5)  coordinate (step4);       
      \path (step4)++(-60:0.4*5)  coordinate (step5);       

\path (origin)++(120:0.4*7 cm) ++(60:0.4*4 cm)coordinate (origin) ;


\draw(origin) node{$\circledcirc$};

\path(origin) --++(120:0.4*9 cm)--++(60:0.4*2 cm)coordinate (alpha);

\path(origin) --++(120:0.4*7 cm)--++(60:0.4*3 cm)coordinate (beta');

\path(origin) --++(120:0.4*10 cm) coordinate (beta);
\path(origin) --++(120:0.4*9 cm)--++(-120:0.4*1 cm)coordinate (gamma);
 
\path(origin) --++(120:0.4*6 cm)--++(60:0.4*2 cm)coordinate (gamma');
 \path(origin) --++(120:0.4*9 cm)--++(-60:0.4*2 cm)coordinate (delta);

\path(origin) --++(60:0.4*8 cm)--++(-60:0.4*3 cm)coordinate (lambda);

\path(origin) --++(180:0.4*2 cm) coordinate (hello);
\path(origin) --++(180:0.4*3 cm) --++(120:0.4*1 cm)  coordinate (hello2);
\path(origin)  --++(180:0.4*5 cm)  coordinate (hello3);
 
\path(lambda) --++(-60:0.4*1 cm)--++(-120:0.4*1 cm)coordinate (mu);
\path(lambda) --++(120:0.4*1 cm)--++(180:0.4*1 cm)coordinate (mu');

\path(lambda) --++(-60:0.4*3 cm)--++(-180:0.4*3 cm)coordinate (nu);
\path(lambda) --++(120:0.4*0 cm)--++(180:0.4*3 cm)coordinate (nu');

\path(lambda) --++(-120:0.4*2 cm)--++(180:0.4*2 cm)coordinate (pi);

 \clip(b10)-- (b20)--(step1)--(step2)--(step3)--(step4)--(step5)--(a10)--(b10);

 \foreach \i in {0,...,35}    
  { \draw[gray, thin]  (a\i) -- (ca\i)  (b\i) -- (cb\i);
     \draw[gray, thin]   (a\i) -- (b\i)  ; } 

 \foreach \i in {0,5,10,...,35}    
  { \draw[thick]  (a\i) -- (ca\i)  (b\i) -- (cb\i);
     \draw[thick]  (a\i) -- (b\i)  ; } 
 
   \end{scope}

      \foreach \i in {0,...,35}
  {
    \path (0,0)++(0:0.4*5)++(60:0.4*\i cm)  coordinate (a\i);
    \path (0,0)++(0:0.4*5)++(120:0.4*\i cm)  coordinate (b\i);
    \path (a\i)++(120:12cm) coordinate (ca\i);
    \path (b\i)++(60:12cm) coordinate (cb\i);
}

 \foreach \i in {0,...,35}    
  { \draw[gray, thin]  (a\i) -- (ca\i)  (b\i) -- (cb\i);
     \draw[gray, thin]   (a\i) -- (b\i)  ; } 

 \foreach \i in {0,5,10,...,35}    
  { \draw[thick]  (a\i) -- (ca\i)  (b\i) -- (cb\i);
     \draw[thick]  (a\i) -- (b\i)  ; } 

   \foreach \i in {0,...,35}
  {
    \path (0,0)++(180:0.4*10)++(60:0.4*\i cm)  coordinate (a\i);
    \path (0,0)++(180:0.4*10)++(120:0.4*\i cm)  coordinate (b\i);
    \path (a\i)++(120:12cm) coordinate (ca\i);
    \path (b\i)++(60:12cm) coordinate (cb\i);
}

 \foreach \i in {0,...,35}    
  { \draw[gray, thin]  (a\i) -- (ca\i)  (b\i) -- (cb\i);
     \draw[gray, thin]   (a\i) -- (b\i)  ; } 

 \foreach \i in {0,5,10,...,35}    
  { \draw[thick]  (a\i) -- (ca\i)  (b\i) -- (cb\i);
     \draw[thick]  (a\i) -- (b\i)  ; }

\draw(alpha) node{$\bm\alpha$};
\draw(beta) node{$\bm\beta$};
\draw(beta') node{$\bm\beta'$};

\draw(gamma) node{$\bm\gamma$};
\draw(gamma') node{$\bm\gamma'$};

\draw(delta) node{$\bm\delta$};

\draw(lambda) node{$\bm\lambda$};

\draw(hello) node{$\bm\sigma$};
\draw(hello2) node{$\bm\tau$};
\draw(hello3) node{$\bm\rho$};

\draw(mu) node{$\bm\mu$};
\draw(mu') node{$\bm\mu'$};

\draw(nu) node{$\bm\nu$};
\draw(nu') node{$\bm\nu'$};
 \draw(pi) node{$\bm\pi$};
\draw[thick, red](origin)--++(120:1*0.4)--++(-120:1*0.4)--++(120:8*0.4);
\draw[thick, red](origin)--++(120:1*0.4)--++(-120:1*0.4)--++(120:7*0.4)--++(0:1*0.4);
\draw[thick, red](origin)--++(120:1*0.4)--++(-120:1*0.4)--++(120:4*0.4)--++(-120:3*0.4)--++(0:1*0.4);
\draw[thick, red](origin)--++(120:1*0.4)--++(-120:1*0.4)--++(120:4*0.4)--++(-120:4*0.4);

\draw[thick, red](origin)--++(120:1*0.4)--++(-120:1*0.4)--++(120:2*0.4)--++(0:2*0.4)--++(0:3*0.4)--++(0:1*0.4);
\draw[thick, red](origin)--++(120:1*0.4)--++(-120:1*0.4)--++(120:2*0.4)--++(0:2*0.4)--++(0:3*0.4)--++(120:1*0.4);

\end{scope}
\path(0,0)--(120:10*0.4)--++(180:10*0.4) coordinate (bottom);
\draw[densely dotted] (bottom)--++(180:0.2);
\draw[densely dotted] (bottom)--++(120:0.2);
\path(bottom)--++(60:1*0.4) coordinate (bottom);
\draw[densely dotted] (bottom)--++(180:0.2);
\draw[densely dotted] (bottom)--++(120:0.2);
\path(bottom)--++(60:1*0.4) coordinate (bottom);
\draw[densely dotted] (bottom)--++(180:0.2);
\draw[densely dotted] (bottom)--++(120:0.2);
\path(bottom)--++(60:1*0.4) coordinate (bottom);
\draw[densely dotted] (bottom)--++(180:0.2);
\draw[densely dotted] (bottom)--++(120:0.2);
\path(bottom)--++(60:1*0.4) coordinate (bottom);
\draw[densely dotted] (bottom)--++(180:0.2);
\draw[densely dotted] (bottom)--++(120:0.2);
\path(bottom)--++(60:1*0.4) coordinate (bottom);
\draw[densely dotted] (bottom)--++(180:0.2);
\draw[densely dotted] (bottom)--++(120:0.2);
\path(bottom)--++(60:1*0.4) coordinate (bottom);
\draw[densely dotted] (bottom)--++(180:0.2);
\draw[densely dotted] (bottom)--++(120:0.2);
\path(bottom)--++(60:1*0.4) coordinate (bottom);
\draw[densely dotted] (bottom)--++(180:0.2);
\draw[densely dotted] (bottom)--++(120:0.2);
\path(bottom)--++(60:1*0.4) coordinate (bottom);
\draw[densely dotted] (bottom)--++(180:0.2);
\draw[densely dotted] (bottom)--++(120:0.2);
\path(bottom)--++(60:1*0.4) coordinate (bottom);
\draw[densely dotted] (bottom)--++(180:0.2);
\draw[densely dotted] (bottom)--++(120:0.2);
\path(bottom)--++(60:1*0.4) coordinate (bottom);
\draw[densely dotted] (bottom)--++(180:0.2);
\draw[densely dotted] (bottom)--++(120:0.2);
\path(bottom)--++(60:1*0.4) coordinate (bottom);
\draw[densely dotted] (bottom)--++(180:0.2);
\draw[densely dotted] (bottom)--++(120:0.2);
\path(bottom)--++(60:1*0.4) coordinate (bottom);
\draw[densely dotted] (bottom)--++(180:0.2);
\draw[densely dotted] (bottom)--++(120:0.2);
\path(bottom)--++(60:1*0.4) coordinate (bottom);
\draw[densely dotted] (bottom)--++(180:0.2);
\draw[densely dotted] (bottom)--++(120:0.2);
\path(bottom)--++(60:1*0.4) coordinate (bottom);
\draw[densely dotted] (bottom)--++(180:0.2);
\draw[densely dotted] (bottom)--++(120:0.2);
\path(bottom)--++(60:1*0.4) coordinate (bottom);
\draw[densely dotted] (bottom)--++(180:0.2);
\draw[densely dotted] (bottom)--++(120:0.2);
\draw[densely dotted] (bottom)--++(60:0.2);
\path(bottom)--++(0:1*0.4) coordinate (bottom);
\draw[densely dotted] (bottom)--++(180:0.2);
\draw[densely dotted] (bottom)--++(120:0.2);
\draw[densely dotted] (bottom)--++(60:0.2);
\path(bottom)--++(0:1*0.4) coordinate (bottom);
\draw[densely dotted] (bottom)--++(180:0.2);
\draw[densely dotted] (bottom)--++(120:0.2);
\draw[densely dotted] (bottom)--++(60:0.2);
\path(bottom)--++(0:1*0.4) coordinate (bottom);
\draw[densely dotted] (bottom)--++(180:0.2);
\draw[densely dotted] (bottom)--++(120:0.2);
\draw[densely dotted] (bottom)--++(60:0.2);
\path(bottom)--++(0:1*0.4) coordinate (bottom);
\draw[densely dotted] (bottom)--++(180:0.2);
\draw[densely dotted] (bottom)--++(120:0.2);
\draw[densely dotted] (bottom)--++(60:0.2);
\path(bottom)--++(0:1*0.4) coordinate (bottom);
\draw[densely dotted] (bottom)--++(180:0.2);
\draw[densely dotted] (bottom)--++(120:0.2);
\draw[densely dotted] (bottom)--++(60:0.2);
\path(bottom)--++(0:1*0.4) coordinate (bottom);
\draw[densely dotted] (bottom)--++(180:0.2);
\draw[densely dotted] (bottom)--++(120:0.2);
\draw[densely dotted] (bottom)--++(60:0.2);
\path(bottom)--++(0:1*0.4) coordinate (bottom);
\draw[densely dotted] (bottom)--++(180:0.2);
\draw[densely dotted] (bottom)--++(120:0.2);
\draw[densely dotted] (bottom)--++(60:0.2);
\path(bottom)--++(0:1*0.4) coordinate (bottom);
\draw[densely dotted] (bottom)--++(180:0.2);
\draw[densely dotted] (bottom)--++(120:0.2);
\draw[densely dotted] (bottom)--++(60:0.2);
\path(bottom)--++(0:1*0.4) coordinate (bottom);
\draw[densely dotted] (bottom)--++(180:0.2);
\draw[densely dotted] (bottom)--++(120:0.2);
\draw[densely dotted] (bottom)--++(60:0.2);
\path(bottom)--++(0:1*0.4) coordinate (bottom);
\draw[densely dotted] (bottom)--++(0:0.2);
\draw[densely dotted] (bottom)--++(120:0.2);
\draw[densely dotted] (bottom)--++(60:0.2);
\path(bottom)--++(-60:1*0.4) coordinate (bottom);
\draw[densely dotted] (bottom)--++(0:0.2);
\draw[densely dotted] (bottom)--++(120:0.2);
\draw[densely dotted] (bottom)--++(60:0.2);
\path(bottom)--++(-60:1*0.4) coordinate (bottom);
\draw[densely dotted] (bottom)--++(0:0.2);
\draw[densely dotted] (bottom)--++(120:0.2);
\draw[densely dotted] (bottom)--++(60:0.2);
\path(bottom)--++(-60:1*0.4) coordinate (bottom);
\draw[densely dotted] (bottom)--++(0:0.2);
\draw[densely dotted] (bottom)--++(120:0.2);
\draw[densely dotted] (bottom)--++(60:0.2);
\path(bottom)--++(-60:1*0.4) coordinate (bottom);
\draw[densely dotted] (bottom)--++(0:0.2);
\draw[densely dotted] (bottom)--++(120:0.2);
\draw[densely dotted] (bottom)--++(60:0.2);
\path(bottom)--++(-60:1*0.4) coordinate (bottom);
\draw[densely dotted] (bottom)--++(0:0.2);
\draw[densely dotted] (bottom)--++(120:0.2);
\draw[densely dotted] (bottom)--++(60:0.2);
\path(bottom)--++(-60:1*0.4) coordinate (bottom);
\draw[densely dotted] (bottom)--++(0:0.2);
\draw[densely dotted] (bottom)--++(120:0.2);
\draw[densely dotted] (bottom)--++(60:0.2);
\path(bottom)--++(-60:1*0.4) coordinate (bottom);
\draw[densely dotted] (bottom)--++(0:0.2);
\draw[densely dotted] (bottom)--++(120:0.2);
\draw[densely dotted] (bottom)--++(60:0.2);
\path(bottom)--++(-60:1*0.4) coordinate (bottom);
\draw[densely dotted] (bottom)--++(0:0.2);
\draw[densely dotted] (bottom)--++(120:0.2);
\draw[densely dotted] (bottom)--++(60:0.2);
\path(bottom)--++(-60:1*0.4) coordinate (bottom);
\draw[densely dotted] (bottom)--++(0:0.2);
\draw[densely dotted] (bottom)--++(120:0.2);
\draw[densely dotted] (bottom)--++(60:0.2);
\path(bottom)--++(-60:1*0.4) coordinate (bottom);
\draw[densely dotted] (bottom)--++(0:0.2);
\draw[densely dotted] (bottom)--++(120:0.2);
\draw[densely dotted] (bottom)--++(60:0.2);
\path(bottom)--++(-60:1*0.4) coordinate (bottom);
\draw[densely dotted] (bottom)--++(0:0.2);
\draw[densely dotted] (bottom)--++(120:0.2);
\draw[densely dotted] (bottom)--++(60:0.2);
\path(bottom)--++(-60:1*0.4) coordinate (bottom);
\draw[densely dotted] (bottom)--++(0:0.2);
\draw[densely dotted] (bottom)--++(120:0.2);
\draw[densely dotted] (bottom)--++(60:0.2);
\path(bottom)--++(-60:1*0.4) coordinate (bottom);
\draw[densely dotted] (bottom)--++(0:0.2);
\draw[densely dotted] (bottom)--++(120:0.2);
\draw[densely dotted] (bottom)--++(60:0.2);
\path(bottom)--++(-60:1*0.4) coordinate (bottom);
\draw[densely dotted] (bottom)--++(0:0.2);
\draw[densely dotted] (bottom)--++(120:0.2);
\draw[densely dotted] (bottom)--++(60:0.2);
\path(bottom)--++(-60:1*0.4) coordinate (bottom);
\draw[densely dotted] (bottom)--++(0:0.2);
\draw[densely dotted] (bottom)--++(120:0.2);
\draw[densely dotted] (bottom)--++(60:0.2);

 \end{tikzpicture}\]

\!\!\!\!\!\!\!\!\!\!\!\!\!\!\!\!\!\!\!\!\!\!\!\!\!\!\!\!\!\!\!\!\!\!\!\!\!\!\!\!\!\!\!\!\!\!\!\!
\caption{
The six elements of  $\Path^+(-,\stt^\gamma)$ in the
 geometry of type ${A}_1\subseteq \widehat{{A}}_2$. 
Each of the pairs $(\alpha,\lambda)$, $(\beta,\mu)$, $(\gamma,\nu)$,
 and $(\delta,\pi)$ are   related by an affine vertex reflection  (as are their primed versions) and so these pairs are generic sets as in \cref{maximalnonparabolic}.  
The pairs $(\mu',\rho)$, $(\nu',\tau)$, $(\pi,\sigma)$
 are all generic sets as in  \cref{maximalnonparabolic}.   }
 \label{bigeg}
   \end{figure}

   \begin{eg}\label{bigegeg}
   We let $e=5$, $h=\ell=2$, $n=13$, and  $\kappa=(0,2)\in (\ZZ/5\ZZ)^2$.  
   We let $\Bbbk$ be an arbitrary field.  
      Thus we are considering a chunk of the modular representation theory of the Hecke algebra of type $B$. 
    We consider the linkage class containing the element $\alpha=((1^{11}),(1^2))$.  
   We consider the closed subset   $$\{\xi =((\xi^{(1)}_1,\xi^{(1)}_2),(\xi^{(2)}_1,\xi^{(2)}_2)) \in \mptn 2 {13}  (2) \mid \xi \trianglerighteq \alpha \text{ and }\xi_{2}^{(1)}=0\}\subset \mptn 2 {13}  (2)$$ which can be embedded in $\mathbb{E}_{2,2}^\circledcirc\cap \mathbb{Z}\{\varepsilon_1,\varepsilon_2,\varepsilon_3\}$ as depicted in \cref{bigeg}.  
   This subalgebra of $A_2(13,\theta,(0,2))$ has 15 simple modules.  
     The decomposition matrix of this algebra (which appears as a submatrix of $A_2(13,\theta,(0,2))$) is as follows:  
   $$  \begin{array}{cc|cccccc|cccccc|ccc}
  &  & \alpha & \beta & \beta' & \gamma & \gamma' & \delta & \lambda & \mu & \mu' & \nu & \nu' &\pi &\rho &\tau &\sigma	\\
\hline 
& \alpha	&	&	&	&	&	&	&	&	&	&	&	&	&	\\
&\beta	&	&	&	&	&	&	&	&	&	&	&	&	&	\\
& \beta' 	&	&	&   	&	&	&	&	&	&	&	&	&	&	\\
&\gamma	& 	& \multicolumn{4}{c}{\smash{\raisebox{.5\normalbaselineskip}{$\scalefont{4}M $}}}	  	&	&	&  
	&	&&&	&	\\
&\gamma'	&	&	&	&	&	&	&	&	&	&	&	&	&	\\
& \delta	&	&	&	&	&	&	&	&	&	&	&	&	&	\\
\hline
&\lambda	&	&	&	&	&	&	&	&	&	&	&	&	&	\\
& \mu 	&	&	&	&	&	&	&	&	&	&	&	&	&	\\
& \mu '	&	&	&	&	&	&	&	&	&	&	&	&	&	\\
& \nu 	&	& 	 \multicolumn{4}{c}{\smash{\raisebox{.5\normalbaselineskip}{$\scalefont{4}t^1M $}}}	   	& 	&& \multicolumn{4}{c}{\smash{\raisebox{.5\normalbaselineskip}{$\scalefont{4}M$}}}	   &	\\
&\nu'		&	&	&	&	&	&	&	&	&	&	&	&	&	\\
&	\pi	&	&	&	&	&	&	&	&	&	&	&	&	&	\\
\hline
&\rho	&0&t^2	&t^2	&t	&0	&0	&0	&0	&t	&0	&0	&0	&1	&	&	\\
&\tau		&t^2&t^3	&t 	&t^2	&0	&t	&0	&0	&t^2	&0	&t	&0	&t	&1	&	\\
&\sigma	&t^3&0	&t^2	&t^3	&t	&t^2	&0	&0	&0	&0	&t^2	&t	&0	&t	&1	\\
   	\end{array} 
   $$
   where the matrix $M$ records the (non-parabolic)   Kazhdan--Lusztig polynomials of type $ {A}_2$.  
   This is given as follows,
$$
 M=\left(\begin{array}{cccccc}
 1 	&   &   &   &   &   \\
 t 	& 1 &   &   &   &   \\
 t 	& 0  & 1 &   &   &   \\
 t^2 	& t & t & 1 &   &   \\
 t^2 	& t & t &0   & 1 &   \\
 t^3 	& t^2 & t^2 & t & t & 1 \\
  \end{array}\right).
 $$   
All the zero entries in the above can be deduced via \cref{strongerman}.

The set of points 
$\{\alpha,\beta,\beta',\gamma,\gamma',\delta,		\lambda,\mu,\mu',\nu,\nu',\tau\}$
is a generic set.  
One can deduce that $M$ appears twice along the diagonal using the ``Steinberg-like'' 
results in \cref{nonparabolic}.    
 That $t^1 M$ appears as a submatrix follows by combining the ``maximal non-parabolic'' results of \cref{maximalnonparabolic} (which gives the diagonal entries of $t^1M$) with the 
 ``Steinberg-like'' 
results in \cref{nonparabolic} (which give the off-diagonal entries).  
 
 We now consider the final 3 rows of the decomposition matrix.  
   The non-zero entries in the final 9 columns  can all be calculated using the generic results of previous sections.  
In particular, the final 3 columns  are given by the Kazhdan--Lusztig polynomials of type $A_2 \subseteq \widehat{A}_2$.

Finally, we consider  the decomposition numbers  in the final  3 rows intersected with the first 6 columns. 
These are the most interesting decomposition numbers in our matrix as they  cannot be calculated using the generic results or \cref{strongerman}.  These can be calculated using the same considerations as in \cref{egegeegggg}.  
   \end{eg}
   
   \begin{rmk}
     Note that  the generic results only give lower bounds on decomposition numbers in $M$ and $t^1M$ in \cref{bigegeg}. One must also verify that the paths of degree zero   $\sts\sim \stt^\alpha$ and $\stt\sim \stt^\lambda$ do correspond to weight spaces in certain simple modules. 

   \end{rmk}

\begin{Acknowledgements*}
The   authors are grateful for the financial support received from the Royal Commission for the Exhibition of 1851  and  EPSRC grant EP/L01078X/1.

\end{Acknowledgements*}

\bibliographystyle{amsalpha}  
\bibliography{master} 
\end{document}